\documentclass{amsart}
\usepackage[utf8]{inputenc}
\usepackage{amsmath, amssymb, amsthm, stmaryrd}
\usepackage{xypic}
\usepackage{hyperref}
\usepackage{xcolor}
\definecolor{dark-red}{rgb}{0.5,0.15,0.15}
\definecolor{dark-blue}{rgb}{0.15,0.15,0.6}
\definecolor{dark-green}{rgb}{0.15,0.6,0.15}
\hypersetup{
	colorlinks, linkcolor=dark-red,
	citecolor=dark-blue, urlcolor=dark-green
}
\usepackage[nameinlink,capitalise,noabbrev]{cleveref}
\usepackage{mathtools}

\usepackage{csquotes}
\usepackage{ifthen}
\usepackage{tikz}
\usetikzlibrary{arrows}
\usepackage{tikz-cd}
\usetikzlibrary{cd}
\usetikzlibrary{arrows.meta}
\tikzcdset{arrow style=tikz,
diagrams={>=Straight Barb}
}
\usepackage{changes}

\newtheorem{constructionletter}{Construction}
\newtheorem{conjletter}{Conjecture}

\numberwithin{equation}{section}
\newtheorem{thm}[equation]{Theorem}
\newtheorem{lemma}[equation]{Lemma}

\newtheorem{prop}[equation]{Proposition}
\newtheorem{conj}[equation]{Conjecture}

\newtheorem*{Crucial}{Crucial Observation}
\newtheorem{principle}[equation]{Principle}

\theoremstyle{definition}
\newtheorem{defi}[equation]{Definition}
\newtheorem{construction}[equation]{Construction}
 \newtheorem{example}[equation]{Example}
 \newtheorem{question}[equation]{Question}

\theoremstyle{remark}

\newtheorem{remark}[equation]{Remark}

\newcommand{\Escr}{\mathcal{E}}
\newcommand{\EE}{\mathcal{E}}

\newcommand{\KK}{\mathcal{K}}
\newcommand{\Lscr}{\mathcal{L}}
\newcommand{\Mscr}{\mathcal{M}}
\newcommand{\MM}{\mathcal{M}}

\newcommand{\Oscr}{\mathcal{O}}
\newcommand{\Fscr}{\mathcal{F}}

\newcommand{\Zscr}{\mathcal{Z}}
\newcommand{\Xscr}{\mathcal{X}}
\newcommand{\CT}
{\mathcal{T}}
\newcommand{\CZ}
{\mathcal{Z}}
\newcommand{\CB}
{\mathcal{B}}
\newcommand{\FT}
{\mathfrak{T}}
\newcommand{\FI}
{\mathfrak{I}}
\newcommand{\CN}
{\mathcal{N}}

\newcommand{\R}{\mathbb{R}}
\newcommand{\C}{\mathbb{C}}
\newcommand{\CP}{\mathbb{CP}}
\renewcommand{\H}{\mathbb{H}}
\newcommand{\Z}{\mathbb{Z}}

\newcommand{\Q}{\mathbb{Q}}

\newcommand{\Otop}{\mathcal{O}^{\mathrm{top}}}
\newcommand{\sm}{\wedge}
\newcommand{\tensor}{\otimes}
\renewcommand{\top}{\mathrm{top}}

\DeclareMathOperator{\id}{id}
\DeclareMathOperator{\pr}{pr}
\DeclareMathOperator{\op}{op}

\DeclareMathOperator{\rk}{rk}
\DeclareMathOperator{\ind}{ind}
\DeclareMathOperator{\res}{res}
\DeclareMathOperator{\tr}{tr}
\DeclareMathOperator{\pt}{pt}
\DeclareMathOperator{\SL}{SL}
\DeclareMathOperator{\TMF}{TMF}
\DeclareMathOperator{\Bil}{Bil}
\DeclareMathOperator{\Mod}{Mod}
\DeclareMathOperator{\gMod}{grMod}
\DeclareMathOperator{\QF}{QF}
\DeclareMathOperator{\Lat}{Lat}
\DeclareMathOperator{\Hom}{Hom}

\DeclareMathOperator{\Pic}{Pic}
\DeclareMathOperator{\Spec}{Spec}

\DeclareMathOperator{\AbGroup}{AbGroup}

\DeclareMathOperator{\QCoh}{QCoh}
\DeclareMathOperator{\Cob}{Cob}
\DeclareMathOperator{\Ho}{Ho}

\begin{document}

\title[A new approach to (3+1)-dimensional TQFTs via TMF]{A new approach to (3+1)-dimensional TQFTs via topological modular forms}

\author{Sergei Gukov}
\address{California Institute of Technology, Division of Physics, Mathematics and Astronomy, 1200 E. California Blvd., Pasadena, CA 91125, USA}
\email{\href{mailto:gukov@math.caltech.edu}{gukov@math.caltech.edu}}

\author{Vyacheslav Krushkal}
\address{Department of Mathematics, University of Virginia, Charlottesville VA 22904-4137}
\email{\href{mailto:krushkal@virginia.edu}{krushkal@virginia.edu}}

\author{Lennart Meier}
\address{Mathematisch Instituut,
Universiteit Utrecht, 
Budapestlaan 6,
3584 CD Utrecht}
\email{\href{mailto:f.l.m.meier@uu.nl}{f.l.m.meier@uu.nl}}

\author{Du Pei}
\address{Center for Quantum Mathematics, University of Southern Denmark, Odense 5220, Denmark}
\email{\href{mailto:dpei@imada.sdu.dk}{dpei@imada.sdu.dk}}

\begin{abstract}
In this paper, we present a construction toward a new type of TQFTs at the crossroads of low-dimensional topology, algebraic geometry, physics, and homotopy theory. It assigns TMF-modules to closed $3$-manifolds and maps of TMF-modules to $4$-dimensional cobordisms.
This is a mathematical proposal for one of the simplest examples in a family of ${\pi}_*({\rm TMF})$-valued invariants of $4$-manifolds which are expected to arise from $6$-dimensional superconformal field theories.
As part of the construction, we define TMF-modules associated with symmetric bilinear forms, using (spectral) derived algebraic geometry. The invariant of unimodular bilinear forms takes values in ${\pi}_*({\rm TMF})$, conjecturally generalizing the theta function of a lattice.  We discuss gluing properties of the invariants. We also demonstrate some interesting physics applications of the TMF-modules such as distinguishing phases of quantum field theories in various dimensions.
\end{abstract}

\maketitle
\setcounter{tocdepth}{1}
\tableofcontents

\section{Introduction}

Since Donaldson introduced gauge theory into the study of smooth 4-manifolds \cite{Donaldson:1983wm}, much progress in low-dimensional topology in the past half-century has been aided by insights from physics. Among the most powerful tools in this area are invariants closely connected with supersymmetric quantum field theories in four dimensions. These include the original Donaldson's polynomial, the Seiberg--Witten invariant \cite{Witten:1988ze,Witten:1994cg}, and their various generalizations. These invariants enjoy nice properties under cutting-and-gluing, giving rise to structures satisfying some of the axioms of topological quantum field theories (TQFTs). 

A recent proposal from theoretical physics suggests the existence of novel invariants based not on 4-dimensional quantum field theories, but instead on those in \emph{six} dimensions \cite{GPPV}. These invariants are still expected to be assembled into TQFTs, but with a crucial difference: the traditional ground rings $\C$, $\Z$, or polynomial rings are replaced with the $E_{\infty}$-ring spectrum $\mathrm{TMF}$ of \emph{topological modular forms}.\footnote{We refer to \cref{sec:BackgroundOnTMF} for background on $\TMF$.} As we discuss below, one of the benefits of working with modules over $\TMF$ is the existence of {\em nilpotent} elements. The essence of the proposal of \cite{GPPV} is illustrated as follows,
\begin{equation}
    \text{\{6d theories\}}\longrightarrow  \text{\{4d TQFTs over TMF\}}.
\end{equation}

In this paper, we take the first step in making this ambitious program mathematically rigorous by constructing the invariants corresponding to one of the simplest possible 6-dimensional theories. 

\subsection{Manifold invariants}
We construct a manifold invariant $\Zscr$, some 
 of whose properties we summarize next; see Sections \ref{sec: invariants of manifolds}, \ref{sec: computations of cobordisms and transfers} for more details.\footnote{A priori, our constructions depend on extra data (like bounding four-manifolds), and we only show partially that they do not depend a posteriori on these. For example, $\Zscr(X)$ for a closed oriented and simply-connected 4-manifold $X$ is only well-defined up to sign in our construction.}

\begin{constructionletter}\label{constrA} \ 
\begin{enumerate}
    \item Given a closed oriented 3-manifold $M$, $\Zscr(M)$ is a module over $\TMF$. 
    \item Given an oriented 4-manifold cobordism $M_0\to M_1$ admitting a relative handle decomposition with all handles of index 2, \[\Zscr(W)\colon \Zscr(M_0)[d] \to \Zscr(M_1)\]
    is a map of $\TMF$ modules 
    for a certain degree shift $d$.
    \item The invariant $\Zscr(X)$ of a closed simply-connected oriented 4-manifold $X$ is an element of $\pi_d(\TMF)$ where $d=3b^+(X) - 2b^-(X)$. If $X$ has a handle decomposition without 1- and 3-handles, $\Zscr(X)$ factors as a composition of maps corresponding to individual handles. 
    \item The invariants of 3- and 4-manifolds are multiplicative under connected~sum.
\end{enumerate}
\end{constructionletter}

To underline the significance of working with modules over $\TMF$, note that 
properties of exotic smooth structures on 4-manifolds impose a strong restriction on (3+1)-dimensional TQFTs. 
It follows from work of Wall and Gompf \cite{Wall64, Gom84} that if two closed smooth oriented 4-manifolds $X_1$ and $X_2$ are oriented homeomorphic, then they are {\em stably diffeomorphic}, that is $X_1\#^n(S^2\times S^2)$ and $X_2\#^n(S^2\times S^2)$ are diffeomorphic for some $n$. Under mild hypotheses, this implies $\mathcal{Z}(X_1) \cdot \mathcal{Z}(S^2\times S^2)^n  = \mathcal{Z}(X_2) \cdot \mathcal{Z}(S^2\times S^2)^n$. For the traditional ground field $\mathbb{C}$ of TQFTs, this implies $\mathcal{Z}(X_1) = \mathcal{Z}(X_2)$ if $\Zscr(S^2\times S^2) \neq 0$; thus such TQFTs do not distinguish exotic smooth structures (cf.\ \cite{ReutterSemisimple}).

In this article, we propose a way to construct a TQFT where $\Zscr(S^2\times S^2)$ is nonzero but nilpotent, a property made possible by working over $\TMF$. For example, $\Zscr(\CP^2) = \nu$ and (conjecturally) $\Zscr(S^2\times S^2) = \eta$, where $\eta\in \pi_1\TMF$ and $\nu\in \pi_3\TMF$ denote the images of the first two Hopf maps in $\TMF$. 
In general, the invariants based on 6-dimensional supersymmetric field theories, proposed in \cite{GPPV}, are expected to be powerful invariants of smooth 4-manifolds, where the nilpotence property of $\Zscr(S^2\times S^2)$ will play an important role. The simplest version, rigorously defined in this paper, depends only on the ``quadratic'' homological data: the intersection form on 4-manifolds and the torsion linking pairing on 3-manifolds. 

\subsection{Invariants for bilinear forms}
 \cref{constrA} is based on the following:

\begin{constructionletter}\label{constrB} \
    \begin{enumerate}
        \item To every symmetric bilinear form $b\colon \Z^d\otimes \Z^d \to \Z$, we define $\TMF$-modules $L^b$ and $L_b = \Hom_{\TMF}(L^{-b}, \TMF)$; 
        \item these are covariantly and contravariantly functorial (up to shift) for morphisms $(\Z^d, b) \to (\Z^{d'}, b')$;
        \item if $b$ is unimodular, $L_b$ is equivalent to a shift of $\TMF$ and the morphism $(\Z^0, 0) \to (\Z^d, b)$ induces a morphism $\TMF \to \TMF[2b^+-3b^-]$ corresponding to an element $\mathfrak{d}_b \in \pi_{3b^--2b^+}\TMF$, with $b^+$ (resp.\ $b^-$) the number of positive (resp.\ negative) eigenvalues. 
    \end{enumerate}
\end{constructionletter}
Up to a shift, the $\TMF$-module $\Zscr(M)$ associated to a closed $3$-manifold $M$ is $L_{b(W)}$ where $W$ is a simply-connected $4$-manifold bounding $M$; we show that this depends up to equivalence just on the torsion linking form on $M$. For a closed simply-connected oriented $4$-manifold $X$, the invariant $\Zscr(X)$ is $\mathfrak{d}_{-b(X)}$ for the intersection form $b(X)$. 

If $b$ is unimodular and positive definite of rank $d$, we can consider the image of $\mathfrak{d}_b \in \pi_{-2d}\TMF$ along the edge homomorphism $\pi_{-2d}\TMF \to M_{-d}$, the vector space of weakly holomorphic modular forms of weight $-d$. (This edge homomorphism is an isomorphism upon complexifying the source, but $\pi_{-2d}\TMF$ can contain torsion and further interesting integral information.) We provide evidence for the conjecture that $\mathfrak{d}_b$ maps to $\Delta^{-\frac{d}8}$ times the theta function $\Theta_b=\sum_{v\in \Z^d} q^{\frac12b(v,v)}$ if $b$ is even. Thus, our construction $\mathfrak{d}_b$ is (conjecturally) both a \emph{refinement of the classical Jacobi theta function} and a \emph{generalization} to unimodular bilinear forms that are not positive definite and even. 

\subsection{Quantum field theories and (spectral) derived algebraic geometry}
\Cref{constrB} is simultaneously motivated by (spectral) derived algebraic geometry and quantum field theories (QFTs). 

The Witten genus associates a modular form to a $2$-dimensional QFT  with $(0,1)$-supersymmetry (2d SQFT) \cite{witten1987elliptic} (see also \cite{tachikawa2025remarks} for a recent review). Significantly refining this, Segal, Stolz and Teichner conjectured that the spectrum $\TMF$ provides a deformation classification of 2d SQFTs, i.e.\ that $\TMF^k(M)$ is in bijection with deformation classes of $M$-indexed families of 2d SQFTs with gravitational anomaly given by $k$ \cite{segal1988elliptic, stolz2004elliptic, stolz2011supersymmetric}. (As QFTs do not have a generally accepted mathematical definition, the mathematical reader can take every mention of QFTs as a guiding heuristic and a challenge for a rigorous treatment.) Under this correspondence, the classes $\mathfrak{d}_b \in \pi_*\TMF$ are meant to correspond to lattice conformal field theories if $b$ is unimodular.\footnote{For an introduction of the subject of (supersymmetric) lattice CFTs from the mathematical viewpoint, with $b$ positive definite, see \cite{frenkel1989vertex} or \cite[Section 4.4]{FrenkelBenZvi}. The kind of 2d physical theories relevant for our construction are those with $(0,1)$ supersymmetry with $b$ being (possibly) indefinite. We refer the reader to \cite[Sec.~2.7]{GPPV} for more details.}

To motivate our construction of $L_b$ for general $b$, recall that a lattice conformal field theory associated to a symmetric bilinear form $b$ on $\Z^d$, which can also be viewed as a $U(1)^d$-WZW model, sits at the boundary of a $U(1)^d$-Chern--Simons theory of level given by $b$ (i.e.\ in the language of \cite{freed2014relative}, it is a \emph{relative quantum field theory}). For example, for positive definite $b$, this means that the partition function and Witten genus of the lattice conformal field theory are not just a normalized theta function (as in the unimodular case), i.e.\ a $\C$-valued function, but rather a ``function'' taking values in the $U(1)^d$-Chern--Simons theory on the $2$-torus. More precisely, we obtain a holomorphic section of a certain pre-quantum or Looijenga line bundle $\Lscr_b^{\C}$ on the moduli of complex elliptic curves with a flat connection on a topologically trivial $U(1)^d$-principal bundle (cf.\ \cite{axelrod1991geometric}). These moduli are equivalent to $\Escr_{\C}^{\times_{\Mscr_{\C}} d}$, where $\Mscr_{\C}$ denotes the moduli stack of complex elliptic curves and $\Escr_{\C}$ the universal complex elliptic curve over it.

Jumping to (spectral) derived algebraic geometry: like modular forms can be defined as sections of line bundles on $\Mscr_{\C}$, topological modular forms $\TMF$ are defined as global sections of a sheaf of $E_{\infty}$-ring spectra on $\Mscr$, the moduli stack of elliptic curve \cite{TMFBook, LurieSurvey}. Likewise, Lurie defined a canonical sheaf of $E_{\infty}$-spectra $\Oscr_{\Escr}^{\top}$ on the universal elliptic curve $\Escr$ and also sheaves of $E_{\infty}$-ring spectra $\Oscr_{\Escr^{\times_{\Mscr}d}}^{\top}$ on all powers $\Escr^{\times_{\Mscr}d}$ \cite{LurieSurvey, LurEllII}. For every symmetric bilinear $b$, we define a derived line bundle $\Lscr_b$ on $\Escr^{\times_{\Mscr}d}$ (i.e.\ an invertible $\Oscr_{\Escr^{\times_{\Mscr}d}}^{\top}$-module) refining the Looijenga line bundle $\Lscr_b^{\C}$ considered above. We define $L^b$ from \cref{constrB} as its global sections and $L_b$ as the $\Hom_{\TMF}(L^{-b}, \TMF)$.

All $L_b$ come with a natural class in $t_b\in \pi_0L_b$ from its covariant functoriality along $0\to \Z^d$. Recall that $\mathfrak{d}_b\in \pi_*\TMF$ should define the class the absolute lattice conformal field theory associated to $b$ for $b$ unimodular. In contrast, $t_b$ should define the class of the lattice conformal field theory associated to $b$ (for general $b$) viewed as a boundary theory for the corresponding $U(1)^d$-Chern--Simons theory of level $b$. In particular, we posit that $L_b$ should classify 2d SQFTs living at the boundary of $U(1)^d$-Chern--Simons theory of level $b$. 

More generally, we conjecture\footnote{We stress that the conjecture is necessarily vague since the necessary definitions of QFTs do not exist. Parts of the conjecture could be made precise in the setting of TQFTs or in the $1$-dimensional setting, but we do not pursue this here. We also stress that some forms of this conjecture have been folklore.} the following categorification\footnote{\emph{Categorification} can have different meanings. In our viewpoint here, $\TMF$-modules provide a categorification of $\TMF$, like $\C$-vector spaces provide a categorification of the complex numbers. This is in contrast to the philosophy of (de)categorification focusing on $K_0$, i.e.\ where vector spaces are a categorification of the integers.} of the Segal--Stolz--Teichner program, which is discussed in more detail in \cref{sec:TMFQFT}:
\begin{conjletter}\label{conjA}
    To every $3$-dimensional topological or supersymmetric quantum field theory $Z$, there is a $\TMF$-module $M(Z)$ such that the underlying space of $M(Z)$ classifies 2d SQFTs living at the boundary of $Z$ (i.e.\ relative SQFTs). 
\end{conjletter}
This both increases the scope of the Segal--Stolz--Teichner program (allowing also to conjecturally classify relative 2d SQFTs) and has the potential to study phases of 3d theories via $\TMF$ as well. We also conjecture the analogous statement replacing $\TMF$ by KO or KU if we reduce the dimension by one. 

\subsection{(3+1)-dimensional TQFTs}
In conclusion of this introduction, we revisit the manifold invariants in Construction~\ref{constrA} and discuss the TQFT context. The existing constructions of $(3+1)$-dimensional TQFTs can be roughly summarized as follows. (We mention just a few citations focusing on relevant aspects of the invariants, rather than including a comprehensive list of references.) A broad family of such theories, where $3$- and $4$-manifolds are endowed with spin$^{\rm c}$ structures, encompasses various kinds of Floer homology theories, cf.~\cite{KM07, OS06, Zemke}, and is geometric in nature. A stable homotopy refinement of Seiberg--Witten invariants and of Seiberg--Witten--Floer homology was given in \cite{BF04, Manolescu, KLS}. (While our work also uses methods of stable homotopy theory, we are not aware of a direct relation.) 
Several approaches are currently being pursued to extend the Khovanov--Rozansky homology \cite{Kh00, KhovanovRozansky} from links in $S^3$ to a functorial invariant of links in general $3$-manifolds and their cobordisms.
Another strand of research \cite{CraneYetter, DouglasReutter} aims to extend constructions with origins in quantum topology that have been well established in $(2+1)$ dimensions. 

The methods in this paper give a novel approach using methods of spectral derived algebraic geometry. 
We establish some TQFT properties and state Conjecture \ref{conj: TQFT} that our invariants give rise to a (3+1)-dimensional TQFT, valued in modules over the ring spectrum $\TMF$.  In Section \ref{sec: invariants of manifolds} we explain some technical issues that need to be resolved in order to approach this conjecture.

\subsection{Guide to the paper} \cref{sec:PhysicsMotivation} provides physical motivation for our constructions via six-dimensional quantum field theories, and \cref{sec:TMFQFT} discusses \cref{conjA}. Both sections are phrased in physical language and we recommend physics-oriented readers to start with these sections. \cref{sec:BackgroundOnTMF} provides background on $\TMF$ and \cref{sec: Looijenga line bundles} gives \cref{constrB}. Its basic properties are established in \cref{sec: Properties of Looijenga line bundles} and \cref{sec: cobordisms}. In \cref{sec: invariants of manifolds}, we switch to manifolds and explain \cref{constrA}. In \cref{sec: computations of cobordisms and transfers}, we compute some examples, and in \cref{sec:thetaFunctions} we discuss the conjectural relationship of our invariants to theta functions. 

\subsubsection*{Acknowledgments} We would like to thank Pavel Putrov, Stephan Stolz and Peter Teichner for insightful conversations. We especially thank Mayuko Yamashita for sharing with us her ideas about the relationship between topological Jacobi forms and $\TMF$-classes associated to lattices and  about $\TMF$-modules associated to Chern--Simons theories. 

SG was supported by the Simons Collaboration Grant on New Structures in Low-Dimensional Topology and by the NSF grant DMS-2245099. VK was supported in part by NSF Grant DMS-2405044. LM was supported by the NWO grant VI.Vidi.193.111. DP was supported partly by research grant 42125 from VILLUM FONDEN,  and ERC-SyG project No.~810573 ``Recursive and Exact New Quantum Theory.'' DP also wants to thank YMSC, Tsinghua University for hospitality during his visits.

\section{Physics motivation: from 6d SCFTs to invariants of manifolds}\label{sec:PhysicsMotivation}

\subsection{$\pi_*(\mathrm{TMF})$-valued invariants of 4-manifolds.}
In \cite{GPPV}, it was proposed that, using superconformal quantum field theories (SCFTs) in six dimensions, one can obtain
powerful smooth invariants of 4-manifolds valued in topological modular forms. The physics construction consists of  two steps. The first step is known as ``compactification,'' which leads to a 2d supersymmetric quantum field theory $T[M_4]$\footnote{Here and in \cref{sec:TMFQFT} we follow the physics convention where the dimension of a manifold is denoted by a subscript.}
\begin{equation*}
    \text{6d $(1,0)$ theory $T$ on $M_4$} \quad\leadsto\quad \text{2d $(0,1)$ theory $T[M_4]$}. 
\end{equation*}
Here $(1,0)$ and $(0,1)$ respectively denote the amount of supersymmetry of the 6d and 2d theory. In both cases, one has the minimal amount allowed to have super-Poincar\'e invariance. The second step involves interpreting a conjecture by Segal, Stolz and Teichner \cite{segal1988elliptic,stolz2004elliptic,stolz2011supersymmetric}, as the homotopy equivalence between the $E_\infty$-ring spectrum TMF and the space $\mathcal{\CT}$ of 2d $(0,1)$ theories,
\begin{equation*}
    \text{TMF}\simeq \mathcal{T}.
\end{equation*}
Here the grading on the right-hand side is given by the ``gravitational anomaly'' of the theory, which can be used to decompose
\begin{equation*}
\mathcal{\CT}=\bigcup_{d\in\mathbb{Z}} \mathcal{\CT}_d.
\end{equation*}
Although $T[M_4]$ will in general depend on the metric on $M_4$, one expects that when $b_2^+>1$, any two generic metrics on $M_4$ lead to two theories that can be smoothly deformed into each other. Therefore, the class $\CZ(M_4)$ in $\pi_*(\mathrm{TMF})$ is expected to be a smooth invariant.

\begin{example}
When the 6d theory $T$ is the ``free abelian tensor multiplet,''\footnote{More precisely, one also needs to decouple a non-compact scalar in the theories $T[M_4]$ and $T[M_3]$ after the compactification. The resulting theory is referred to as the ``toy model'' in \cite{GPPV}.}  the invariants for $S^2\times S^2$, $\mathbb{CP}^2$, and $\mathbb{CP}^2\#\mathbb{CP}^2$ are respectively
$\CZ(S^2\times S^2)=\eta$, $\CZ(\mathbb{CP}^2)=\nu$, and $\CZ(\mathbb{CP}^2\#\mathbb{CP}^2)=\nu^2$.
\end{example}

\subsection{4d TQFT over TMF}
Defining and computing (even at the level of physics) the new invariants when the 6d theory is no longer a free theory appears to be a daunting task. Our approach 
to overcoming this difficulty is to utilize a novel type of 4d TQFT structure not over the usual field of complex numbers $\C$, but over TMF. In other words, we now associate to a 3-manifold a TMF-module,
\begin{equation*}
    M_3\quad\leadsto\quad \CZ(M_3)\in \text{TMF-mod},
\end{equation*}
and to a 4d cobordism $W$ from $M_3$ to $M'_3$ a morphism of modules 
\begin{equation*}
    W\quad\leadsto \quad\CZ(W): \;\CZ(M_3)\rightarrow \CZ(M_3').
\end{equation*}
When $M_3=M_3'=\emptyset$, one has $\CZ(M_3)= \CZ(M_3')=$ TMF, and taking the image of 1 under $\CZ(W)$ in the homotopy group allows one to interpret $\CZ(W)\in\pi_*$TMF.

The existence of such TQFT follows from considering the compactification of the 6d theory on open 4-manifolds. If $M_4$ has a boundary $\partial M_4=M_3$, then, after compactifying the 6d theory on it, instead of having a stand-alone theory $T[M_4]$, it is now a boundary condition for a three-dimensional theory $T[M_3]$. In the set-up considered in \cite{GPPV}, $T[M_3]$ has 3d $\mathcal{N}=1$ supersymmetry, again the minimal amount in 3d. Then, as we will discuss in greater detail and generality in Section~\ref{sec:TMFQFT}, there is a natural TMF-module associated with it, which in physics is interpreted as the space of 2d $(0,1)$ boundary conditions. And $T[M_4]$ precisely takes value in this TMF-module as it is a specific boundary condition. When $W$ is a cobordism from $M_3$ to $M_3'$, $T[W]$ is now an interface between $T[M_3]$ and $T[M_3']$, and similarly it gives a map between two TMF-modules. 

The 4d TQFT that we are going to build corresponds to taking the 6d theory $T$ to be the ``free tensor multiplet.'' Although, being a free theory, it is not expected to lead to powerful invariants of four-manifolds that can detect smooth structures, it will serve as a testing ground for the existence and functionality for such TQFTs from more general 6d theories.

Furthermore, the computation of the TMF-modules are also of interest from the quantum field theory point of view. In this simplest case, $T[M_3]$ will be given by an abelian gauge theory, possibly with Chern--Simons levels, and, as we will discuss in Section~\ref{sec:TMFQFT}, the TMF-modules constructed are often compatible with our knowledge about the physics of such 3d theories, sometimes leading to new insights and new predictions.

\section{Background on topological modular forms}\label{sec:BackgroundOnTMF}
\subsection{Modular forms and moduli of elliptic curves}
A \emph{(weakly holomorphic) modular form of weight $k$} is a holomorphic function $f$ on the upper half-plane $\H$ such that 
\[f\left(\frac{a\tau +b}{c\tau +d}\right) = (c\tau+d)^kf(\tau)\]
for all $\tau\in \H$ and all $\begin{pmatrix}a&b\\c&d\end{pmatrix}\in \SL_2(\Z)$, having at most a pole at $\infty$. Equivalently, a modular form of weight $k$ is a holomorphic section of the line bundle $\omega_{\C}^{\tensor k}$ on the orbifold $[\H/\SL_2(\Z)]$, where $\omega_{\C}$ by definition descends from $\H \times \C$ via the $\SL_2(\Z)$-action 
\[\begin{pmatrix}a&b\\c&d\end{pmatrix}\cdot (\tau, z) = \left(\frac{a\tau +b}{c\tau +d}, (c\tau+d)z\right). \]

The orbifold $[\H/\SL_2(\Z)]$ classifies complex elliptic curves (which are all of the form $\C/(\tau\Z + \Z)$ for some $\tau\in \H$) \cite{HainModuli}. In algebraic geometry, we can take elliptic curves over arbitrary commutative rings (and even more generally) and thus it makes sense to replace $[\H/\SL_2(\Z)]$ by the moduli stack of elliptic curves $\MM$ defined over the integers. We will also denote $[\H/\SL_2(\Z)]$ by $\MM_{\C}$ as it agrees with the complex orbifold associated to the complexification of $\MM$. A map $\varphi\colon X\to \MM$ corresponds exactly to an elliptic curve over $X$. More precisely, $\MM$ carries the universal elliptic curve $\EE$ and $\varphi$ corresponds to the elliptic curve $\varphi^*\EE$. Over the complex numbers, $\EE$ can be constructed as $[\H \times \C/(\SL_2(\Z)\ltimes\Z^2)]$. 

As every elliptic curve, $\EE$ comes with a ``neutral element'' of its group structure; since it is an elliptic curve over $\MM$, this takes the form of a section $e\colon \MM \to \EE$ of the structure map $p\colon \EE \to \MM$. (The section $e$ defines a closed immersion and more precisely a relative Cartier divisor; we will sometimes identify $e$ with this divisor.) We obtain a line bundle $\omega = e^* \Omega^1_{\EE/\MM}$ as the pullback of the sheaf of differentials. Using that $dz$ defines a trivialization of the sheaf of differentials of a complex elliptic curve $\C/\Z+\tau\Z$, one produces an isomorphism of the complexification of $\omega$ with $\omega_{\C}$. 

The space $M_k = H^0(\MM;\omega^{\tensor k})$ of global sections of $\omega^{\tensor k}$ on $\MM$ is called the space of \emph{integral modular forms of weight $k$}. By the $q$-expansion principle, the map $M_k \to M_k^{\C}$ into complex modular forms identifies $M_k$ with the complex modular forms with integral $q$-expansion. For more details and background see e.g.\ \cite[Appendix A]{MeierOzornova} and the references therein. 

\subsection{Topological modular forms}
Topological modular forms $\TMF$ are a refinement of classical modular forms. They are a multiplicative cohomology theory whose coefficients $\TMF^*(\pt)$ resemble the ring of integral modular forms, but incorporate torsion. More precisely, there is a homomorphism $\TMF^{-2k}(\pt) \to M_k$, which becomes an isomorphism upon inverting $6$. Variants of topological modular forms are also known under the name of elliptic cohomology. The theory of topological modular forms was built by Hopkins together with Goerss, Miller and Mahowald and later revamped by Lurie. We refer to \cite{TMFBook,BehrensHandbook, Hopkins} and \cite{LurieSurvey} for details, background and references, or to \cite[Appendix A]{GPPV} for another short overview.

Both for the construction of $\TMF$ and for our purposes, it is important to view $\TMF$ not just as a multiplicative cohomology theory but as an $E_{\infty}$-ring spectrum. These are spectra with a highly-coherent commutative multiplication. Many frameworks exist to make this precise: for an $\infty$-categorical treatment see \cite{HA} or the more succinct \cite{GepnerHigher}. Alternatively, they can e.g.\ be modeled as commutative monoids in symmetric or orthogonal spectra; see e.g.\ \cite{MMSS}, \cite{Sch12} or \cite{Malkiewich}. For most of our paper, the details are not important; we just need a few facts, in particular:
\begin{itemize}
    \item $E_{\infty}$-ring spectra $R$ have a graded commutative ring of homotopy groups $\pi_*R$. They moreover represent a multiplicative cohomology theory such that $R^*(\pt) \cong \pi_{-*}R$.
    \item $E_{\infty}$-ring spectra behave in many ways like usual commutative rings (or commutative differential graded algebras). In particular, they have an $\infty$-category of modules $\Mod_R$; we will usually only work with its homotopy category $\Ho(\Mod_R)$, which is a classical (triangulated) category. 
\end{itemize}
Like modular forms arise as the global sections of a sheaf on $\MM$, so does $\TMF$: there is a sheaf $\Otop_{\MM}$ of $E_{\infty}$-ring spectra on (the \'etale site of) $\MM$ and $\TMF$ is defined as $\Gamma(\Otop_{\MM})$. The homotopy group sheaves $\pi_{2k}\Otop_{\MM}$ are isomorphic to $\omega^{\tensor k}$, while $\pi_{2k-1}\Otop_{\MM}$ is zero for all $k\in \Z$. This results in a spectral sequence 
\[E_2^{st} = H^s(\MM; \omega^{\tensor t}) \Rightarrow \pi_{2t-s}\TMF,\]
whose associated edge homomorphism $\pi_{2k}\TMF \to M_k$ we have already seen above. While an isomorphism after inverting $6$, it is neither injective nor surjective. For example, the discriminant $\Delta \in M_{12}$ is not in the image, only $24\Delta$. The kernel consists of torsion; many elements in the kernel come from the homomorphism $\pi_*\mathbb{S} \to \pi_*\TMF$ from the stable homotopy groups of spheres. This includes $\eta\in \pi_1\TMF$, $\nu\in \pi_3\TMF$, $\varepsilon \in \pi_8\TMF$, $\kappa \in \pi_{14}\TMF$ and $\overline{\kappa} \in \pi_{20}\TMF$. A depiction of the homotopy groups of $\TMF$ can be found in \cite{TMFBook}. 

\section{Looijenga line bundles} \label{sec: Looijenga line bundles}
As in the last section, let us denote by $\MM$ the moduli stack of elliptic curves and by $\EE$ the universal elliptic curve over $\MM$. Let $b$ be an integral symmetric bilinear form, i.e.\ a symmetric bilinear map $b\colon \Z^d \tensor \Z^d \to \Z$. (For brevity, we will assume from now on that all our bilinear forms are integral and symmetric.) Following Looijenga \cite{Looijenga}, one can associate to $b$ a line bundle on $\EE_{\C}^{\times_{\MM}d}$. We will discuss this below, as well as a novel extension to derived line bundles, yielding $\TMF$-modules upon taking global sections. 

\subsection{The complex picture}\label{sec:ComplexPicture}
Throughout this section, we will view the complex universal elliptic curve $\EE_{\C}$ as the orbifold $[\H \times \C/(\SL_2(\Z)\ltimes\Z^2)]$. Correspondingly, we view the fiber product $\EE_{\C}^{\times_{\MM}d}$  as the orbifold $[\H \times \C^d/\SL_2(\Z)\ltimes(\Z^{2})^d]$. Given an element $(\tau, z) \in \H \times \C^d$ and elements $A\in \SL_2(\Z)$ and $(m_1, m_2) \in (\Z^{2})^d$, the action is given by
\begin{align*}
    A\cdot (\tau, z)  &= \left(\frac{a\tau+b}{c\tau+d}\,,  \,(c\tau+d)^{-1}z\right) \\
    (m_1, m_2) \cdot (\tau, z) &= (\tau, z+m_1\tau + m_2).
\end{align*}
Following \cite{Rezk} we define:
\begin{defi}\label{def:ComplexLooijenga}
    Let $b$ be a bilinear form on $\Z^d$ and denote its $\C$-bilinear extension $\C^d\tensor \C^d \to \C^d$ by the same symbol. We define the \emph{complex Looijenga line bundle} $\Lscr_b^{\C}$ on $\EE_{\C}^{\times_{\MM}d}$ by defining an action of $\SL_2(\Z)\ltimes (\Z^{2})^d$ on $\H \times \C^d \times \C$ as follows (using notation as above): 
    \begin{align*}
    A\cdot (\tau, z, y)  &= \left(\frac{a\tau+b}{c\tau+d},  (c\tau+d)^{-1}z, e^{\pi i(c(c\tau+d)^{-1}b(z,z))} y\right) \\
    (m_1, m_2) \cdot (\tau, z, y) &= \left(\tau, z+m_1\tau + m_2, e^{-2\pi i(b(z, m_1) + \frac12b(m_1,  m_1)\tau)}y\right).
\end{align*}
\end{defi}
In the case that $b$ is positive definite and \emph{even} (i.e.\ $b(v,v) \in 2\Z$ for all $v\in \Z^d$), global sections $f$ of $\Lscr_b^{\C}\tensor \omega_{\C}^{\tensor k}$ are called a \emph{Jacobi form of weight $k$ and index $b$} in \cite[Definition 1.3]{BoylanSkoruppa} (see also \cite{KriegJacobi}) if their  Fourier expansion is well-behaved, i.e.\ if the pullback of $f$ to a function on $\H \times \C^d$ can be written as 
\begin{equation}\label{eq:JacobiFourier}f(\tau, z) = \sum_{n\geq 0}q^n\sum_{r\in L^{\vee}: b(r,r)\leq 2n}c(n,r)e^{2\pi ib(z,r)}.\end{equation}
Here, $q = e^{2\pi i\tau}$ and $L^{\vee} = \{z\in \C^d\,:\, b(z,m)\in \Z\text{ for all } m\in \Z^d$\}. 

The precise formulae are closely related to the transformation behavior of theta functions. If $b$ is positive definite, even and \emph{unimodular} (i.e.\ defining an isomorphism from $\Z^d$ to its dual), we define 
\begin{equation}\label{eq:multivariabletheta}\theta_b\colon \H\times\C^d \to \C, \qquad (\tau, z) \mapsto \sum_{v\in \Z^d}e^{\pi i \tau b(v,v) + 2\pi i b(v,z)}.\end{equation}
The function $\theta_b$ is an example of a Jacobi form of weight $\frac{d}2$ with respect to $b$ (see e.g.\ \cite[Section 4]{KriegJacobi} or \cite[Proposition 3.4, 3.5]{KacPeterson}). Of course, the theory of theta functions is much older than these sources, going back to the 19th century, and was invented and developed by Jacobi, Rosenhain, Riemann etc. 

From the formulae in \cref{def:ComplexLooijenga}, we directly deduce the following properties:
\begin{enumerate}
    \item (Additivity) $\Lscr_{b+b'}^{\C} \cong \Lscr_{b}^{\C} \tensor \Lscr_{b'}^{\C}$ and $\Lscr_{0}$ is the structure sheaf.
    \item (Functoriality) A morphism $\Z^{d_1} \to \Z^{d_2}$ corresponds to a matrix $A$, which induces a morphism $\Escr^A\colon \Escr^{\times_{\Mscr} d_1}_{\C} \to \Escr^{\times_{\Mscr} d_2}_{\C}$. Moreover, given a bilinear form $b \in \Bil(\Z^d)$ with matrix $B$, we can define a new bilinear form $A^*b \in \Bil(\Z^e)$ with matrix $A^TBA$. With this notation, there is an isomorphism $\Lscr_{A^*b}^{\C} \cong (\Escr^A)^*\Lscr(b)^{\C}$. 
\end{enumerate}
Moreover, there is an isomorphism $\Lscr^{\C}_{(1)}$ with $\Oscr_{\EE}^{\C}(e) \tensor \omega_{\C}$ \cite{BauerJacobi}. Here and in the following, we will often identify a bilinear form with its defining matrix; thus, e.g., $(1)$ stands for the bilinear form $\Z \times \Z \to \Z$, sending $(v,w)$ to $vw$. 

Beyond the transformation behavior of theta functions, there is also a second motivation for the definition of the Looijenga line bundles: there is essentially no other choice with the properties above. (A third motivation from algebraic topology is given in \cite{Rezk}.)

\begin{prop}\label{UniquenessOfComplexLooijenga}
     Suppose that for every bilinear form $b\in \Bil(\Z^d)$, there is a line bundle $\KK_b$ on $\Escr^{\times_{\Mscr} d}_{\C}$, satisfying additivity and functoriality as above. Assume moreover $\KK_{(1)} \cong \Oscr_{\EE}^{\C}(e) \tensor \omega_{\C}$. Then $\Lscr^{\C}_{b} \cong \KK_b$ for all $b\in \Bil(\Z^d)$.  
\end{prop}
\begin{proof}
    Every symmetric integer matrix is a sum of symmetric matrices with at most two nonzero entries, which equal $\pm 1$. Thus, every bilinear form on $\Z^d$ is a sum of bilinear forms pulled back along a projection $\Z^d \to \Z^2$ or $\Z^d\to \Z$ from $h = \begin{pmatrix}0&1\\1&0\end{pmatrix}$ or $(1)$, respectively, or their negatives. By additivity and functoriality, it remains to show that $\Lscr^{\C}_{h} \cong \KK_h$. 

    Note that $h\oplus (1) \cong (1) \oplus (-1) \oplus (1)$. Since pullback along the first two factors of $h\oplus (1)$ is $h$ again and $\Lscr_{(-1)}^{\C} \cong \left(\Lscr_{(1)}^{\C}\right)^{-1}$ by additivity (and likewise for $\KK$), the result follows. (See \cref{rem:Poincare} for more details.)
\end{proof}

This implies that every approach to Looijenga line bundles satisfying the same formal properties and normalization will yield (up to isomorphism) the same line bundles.

\subsection{Derived Looijenga line bundles and $\TMF$-modules associated to bilinear forms}\label{sec: bilinear to TMF}

As recalled above, the ring spectrum $\TMF$ arises as the global sections of a sheaf of $E_{\infty}$-ring spectra $\Otop_{\Mscr}$ on $\Mscr$. Lurie also constructs sheaves of $E_{\infty}$-ring spectra $\Otop_{\Escr^d}$ on $\Escr^{\times_{\Mscr} d}$.\footnote{More precisely, it is explained in Section 4.1 of \cite{MeierDecompositions} how the results of \cite{LurEllII} define a sheaf $\Otop_{\Escr}$ on $\Escr$, making $(\Escr, \Otop_{\Escr})$ into a derived stack (or nonconnective spectral Deligne--Mumford stack in the parlance of \cite{SAG}). This allows to take the $d$-fold iterated fiber product of $(\Escr, \Otop_{\Escr})$ over $(\Mscr, \Otop_{\Mscr})$. The flatness of $\Escr \to \Mscr$ ensures that the underlying stack of this fiber product is precisely $\Escr^{\times_{\Mscr}d}$. This equips the latter stack with the required sheaf of $E_{\infty}$-ring spectra. See also \cite[Lemma B.3]{MeierDecompositions}.} For us, a \emph{derived line bundle} on $\Escr^{\times_{\Mscr} d}$ will be an $\Otop_{\Escr^{\times_{\Mscr}d}}$-module that is locally trivial, i.e.\ is \'etale locally equivalent to $\Otop_{\Escr^d}$. Our goal in this section is to discuss derived analogs $\Lscr_b$ of the Looijenga line bundles from the preceding section. This will be the key input to define $\TMF$-modules associated to bilinear forms and to $3$-manifolds. As taking $\pi_0$ of $\Lscr_b$ yields a (non-derived) Looijenga line bundle $\Lscr_b^{\Z}$ on $\Escr^{\times_{\Mscr} d}$ without the need to complexify, we will discuss this theory in parallel. We will only give a sketch of the construction. Full details will appear in \cite{MeierTwistings}.

To associate to a $b\in\Bil(\Z^d)$ a line bundle $\Lscr_b^{\Z}$ and a derived line bundle $\Lscr_b$ on $\Escr^{\times_{\Mscr} d}$, we will follow the following three rules:
\begin{enumerate}
    \item (Additivity) $\Lscr_{b+b'} \simeq \Lscr_{b} \tensor \Lscr_{b'}$ (and similar for $\Lscr_b^{\Z}$).
    \item (Functoriality) A morphism $\Z^e \to \Z^d$ corresponds to a matrix $A$, which induces a morphism $\Escr^A\colon \Escr^{\times_{\Mscr} e} \to \Escr^{\times_{\Mscr} d}$. Given a bilinear form $b \in \Bil(\Z^d)$ with matrix $B$, we can define a new bilinear form $A^*b \in \Bil(\Z^e)$ with matrix $A^TBA$. We demand an equivalence $\Lscr_{A^*b} \simeq (\Escr^A)^*\Lscr(b)$ (and similar for $\Lscr_b^{\Z}$). 
    \item (Normalization) $\Lscr_{(1)} \simeq \Otop_{\Escr}(e)[-2]$ and $\Lscr_{(1)}^{\Z} \cong \Oscr_{\Escr}(e) \tensor \omega$.  
\end{enumerate}

Here, $[-2]$ denotes a shift so that $\pi_k\Otop_{\Escr}(e)[-2] \cong \pi_{k+2}\Otop_{\Escr}(e)$. Moreover, $\Otop_{\Escr}(e)$ is the dual of $\Otop_{\Escr}(-e)$, which in turn is the fiber of the morphism $\Otop_{\Escr} \to e_*\Otop_{\Mscr}$ ``evaluating at the zero-section $e\colon \Mscr\to \Escr$'' that is adjoint to the equivalence $e^*\Otop_{\Escr} \simeq \Otop_{\Mscr}$. 

Analogously to \cref{UniquenessOfComplexLooijenga}, one shows that the three properties above define the derived and integral Looijenga line bundles uniquely up to equivalence/isomorphism, but a priori these conditions might overdetermine the (derived) line bundle. 

To be more precise, let $\Lat$ be the category of lattices, i.e.\ finitely generated free abelian groups. (Equivalently, its objects might be viewed as $\Z^d$ for $d\geq 0$ and its morphisms as integral matrices, but the more abstract perspective helps to get the correct functoriality.) For a lattice $\Lambda$, we can define $\Escr \tensor \Lambda$ so that $\Escr \tensor \Z^d = \Escr^{\times_{\Mscr}d}$.\footnote{More precisely, the $(2,1)$-category of abelian group objects in Deligne--Mumford stacks over $\Mscr$ is tensored over abelian groups, and $\Escr$ has the structure of such an abelian group object. Likewise, the $\infty$-category of abelian group objects in derived stacks over $\Mscr$ is tensored over abelian groups, and $(\Escr, \Otop_{\Escr})$ has the structure of such an abelian group object.} Consider the three functors 
\begin{align*}
    \Bil \colon \Lat^{\op}\to \AbGroup, &\qquad \Lambda \mapsto \Bil(\Lambda),\\
    \Pic^{\Z}\colon \Lat^{\op}\to \AbGroup, &\qquad \Lambda \mapsto \Pic(\Escr\tensor \Lambda),\\
    \Pic^{\top}\colon \Lat^{\op}\to \AbGroup, &\qquad \Lambda \mapsto \Pic(\Escr\tensor \Lambda, \Otop_{\Escr\tensor \Lambda}),
\end{align*}
the last two yielding isomorphism classes of line bundles and derived line bundles, respectively. The functoriality and additivity requirement above are subsumed in asking for \emph{natural transformations} $\Lscr_{-}^{\Z}\colon \Bil \to \Pic^{\Z}$ and $\Lscr_{-}\colon \Bil \to \Pic^{\top}$. 

To construct these, we consider a fourth functor $\Lat^{\op}\to \AbGroup$, namely the functor $\Z[(-)^{\vee}]$, mapping $\Lambda$ to the free abelian group $\Z[\Lambda^{\vee}]$ on the dual of $\Lambda$. This functor is free in the sense that a natural transformation out of it to some $F$ is equivalent data to specifying an element in $F(\Z)$ as the image of $1\cdot [\id]\in \Z[\Z^{\vee}]$. Thus, we obtain natural transformations $\Z[()^{\vee}]\to \Pic^{\Z}$ and $\Z[()^{\vee}]\to \Pic^{\top}$ sending $1\cdot [\id]$ to $\omega$ and $\Otop(e)[-2]$, respectively. 

Likewise, there is a natural transformation $\Z[()^{\vee}]\to \Bil$, sending $1\cdot [\id]$ to $(1)$. Concretely, this associates to each $\lambda\in \Lambda^{\vee}$ the bilinear form $(v,w)\mapsto \lambda(v)\cdot \lambda(w)$. The following elementary result will appear in \cite{MeierTwistings}.
\begin{Crucial} The map $\Z[()^{\vee}]\to \Bil$ is a surjection with kernel generated by three explicit relations. Thus, defining integral and derived Looijenga line bundles reduces to checking these three relations. \end{Crucial}

In our situation, these three relations boil down to the following: 
\begin{enumerate}
    \item The pullbacks of $\Lscr_{(1)} \simeq \Otop_{\Escr}(e)[-2]$ and $\Lscr_{(1)}^{\Z} \cong \Oscr_{\Escr}(e) \tensor \omega$ along $e\colon \Mscr \to \Escr$ are trivial, i.e.\ the structure sheaf. In the derived setting, we will show this in \cref{lem:Normalization} below. 
    \item The pullback $[-1]^*\Lscr_{(1)}$ of $\Lscr_{(1)} \simeq \Otop_{\Escr}(e)[-2]$ along $[-1]\colon \Escr \to \Escr$ is equivalent to $\Lscr_{(1)}$ again; likewise, $[1]^*\Lscr_{(1)}^{\Z} \cong \Lscr_{(1)}^{\Z}$. This follows (in the derived case) since $[-1]e = e$ implies that $[-1]^*\Otop_{\Escr}\to [-1]^*e_*\Otop_{\Mscr}$ is equivalent to the morphism $\Otop_{\Escr}\to e_*\Otop_{\Mscr}$ whose fiber is the dual of $\Otop_{\Escr}(e)$. The non-derived case follows by passing to $\pi_0$.  
    \item The (derived) theorem of the cube applied to $\Lscr_{(1)} \simeq \Otop_{\Escr}(e)[-2]$ and $\Lscr_{(1)}^{\Z} \cong \Oscr_{\Escr}(e) \tensor \omega$. The classical form can be found e.g.\ in \cite{AndoHopkinsStrickland}, and the derived form will be proven using equivariant elliptic cohomology in \cite{MeierTwistings}.
\end{enumerate}

Thus, we obtain in the derived case the following result, a full proof of which will appear in the forthcoming \cite{MeierTwistings}: 
\begin{thm}\label{thm:DerivedLooijenga}
    We can associate to bilinear forms derived Looijenga line bundles satisfying additivity, functoriality and normalization, i.e.\ a natural transformation $\Bil\to \Pic^{\top}$ sending $(1)$ to $\Otop_{\Escr}(e)[-2]$.\end{thm}

For many purposes, it is not enough to treat the derived line bundles just up to equivalence, but we also have to keep track of equivalences between them. Thus, we will need to refine the previous statement to one on the level of \emph{Picard groupoids}. Let $\mathrm{PicGrpds}$ be the $2$-category of Picard groupoids, i.e.\ symmetric monoidal groupoids in which every object is invertible: its morphisms are (strongly) symmetric monoidal functors and the $2$-morphisms are monoidal natural transformations. We denote by $\mathcal{P}ic^{\top}$ the (pseudo-)functor $\Lat^{\op}\to \mathrm{PicGrpds}$  sending each $\Lambda$ to the Picard groupoid with objects derived line bundles on $(\Escr\tensor \Lambda, \Otop_{\Escr\tensor\Lambda})$ and as morphisms equivalences between them up to homotopy. Note that we can view $\AbGroup$ as a subcategory of $\mathrm{PicGrpds}$, namely as those Picard groupoids with no non-identity morphisms. 
\begin{conj}\label{conj:hypothesis}
    There is an (essentially unique) natural transformation $\Lscr$ from the functor 
    \[\Bil\colon \Lat^{\op} \to \AbGroup \subset \mathrm{PicGrpds}\]
    to the functor 
    \[\mathcal{P}ic^{\top}\colon \Lat^{\op} \to \mathrm{PicGrpds}\]
    that sends $vw \in \Bil(\Z)$ to $\Otop(e)[-2]$.\footnote{As these are functors going into a $2$-category, one should be a bit more precise on what this means. (This becomes a little easier when formulated using the corresponding unstraightened fibrations.) For us it is especially important that we obtain \emph{canonical} equivalences $\Lscr_{b+b'} \simeq \Lscr_b \otimes \Lscr_{b'}$ and $\Lscr_0 \simeq \Oscr_{\Escr^{\mathrm{top}}\tensor \Lambda}$ if $0$ is the zero-form on some lattice $\Lambda$. This implies in particular that $\Lscr_{-b} \simeq \Lscr_{b}^{\vee}$ canonically.}
\end{conj}

Finally, we explain how to associate $\TMF$-modules to bilinear forms. To that purpose, let us denote for an object $X$ by $X^{\vee}$ the appropriate \emph{dual} of $X$. Thus, for a $\TMF$-module $M$, we define $M^{\vee}$ as the $\Hom$-spectrum $\Hom_{\TMF}(M, \TMF)$; for a derived line bundle $\Lscr$ on a derived stack $(\Xscr, \Otop_{\Xscr})$, we define $\Lscr^{\vee}$ as the Hom-sheaf $\mathcal{H}om_{\Otop_{\Xscr}}(\Lscr, \Otop_{\Xscr})$. In the former case, $M^{\vee}$ retains the structure of a $\TMF$-module, and in the latter case, $\Lscr^{\vee}$ is a derived line bundle again. 
\begin{defi}\label{def:LbLb}
    For a bilinear form $b$, we define 
    \begin{align*}
        L^b &= \Gamma(\Lscr_b), \\
        L_b &= \Gamma(\Lscr_b^{\vee})^{\vee}
    \end{align*}
\end{defi}
The modules $L_b$ will figure in the definition of our potential $(3+1)$-dimensional TQFT with values in $\TMF$-modules. Since $\Lscr_{b}^{\vee}\simeq \Lscr_{-b}$ by the additivity of derived Looijenga line bundles, we have $L_b \simeq (L^{-b})^{\vee}$. 

The notation with the upper and lower index is motivated by the relation to equivariant $\TMF$. Building upon the works of Grojnowski and Lurie, Gepner and the third author have defined in \cite{GepnerMeier} a theory of $U(1)^d$-equivariant $\TMF$. The key is the definition of a symmetric monoidal (reduced) equivariant elliptic cohomology functor
\[\Oscr_{\Escr^{\times_{\Mscr}d}}^{-}\colon (\text{finite pointed }U(1)^d\text{-CW complexes})^{\op} \to \QCoh(\Escr^{\times_{\Mscr}d}, \Otop_{\Escr^{\times_{\Mscr}d}})\]
to quasi-coherent $\Otop_{\Escr^{\times_{\Mscr}d}}$-modules. The actual equivariant $\TMF$-cohomology of a $U(1)^d$-space $X$ is then defined as the (homotopy groups of the) global sections of $\Oscr_{\Escr^{\times_{\Mscr}d}}^{X_+}$, where $X_+$ is $X$ with a disjoint basepoint.

Since the value on $S^0$ is the structure sheaf, the $U(1)^d$-equivariant $\TMF$-cohomology of a point is computed as (the homotopy groups of) $\Gamma(\Otop_{\Escr^{\times_{\Mscr}}d})=L^0$. There is the general principle that twisted cohomology or cohomology with local coefficients arises as global sections of invertible systems. As the global sections of the invertible $\Otop_{\Escr^{\times_{\Mscr}d}}$-modules $\Lscr_b$, the (homotopy groups of the) $\TMF$-modules $L^b$  can be seen as twisted equivariant $\TMF$-\emph{cohomology} of a point. (We will comment more about this twisting in the next subsection.) In contrast, $L_b$ may be seen as the twisted $U(1)^d$-equivariant $\TMF$-\emph{homology} of a point.
\begin{remark}
    We will show below in \cref{thm:Lbduals} that $L_b \simeq L^b[d]$ for $b$ defined on $\Z^d$, though potentially in a non-obvious way (see \cref{prop:explicitduality}).
\end{remark}

\subsection{The relationship to anomalies and Lurie's work on $2$-equivariant $\TMF$}\label{sec: relation to 2-equivariant TMF}
The discussion in this subsection is not essential for the later mathematical content of this paper, but puts it into context, both with regard to twistings of equivariant $\TMF$ and with regard to the theory of anomalies in quantum field theories. 

To a quadratic form $q$, we can associate the bilinear form 
\[b(v,w) = q(v+w)-q(v)-q(w),\]
defining a map $\QF(\Z^d) \to \Bil(\Z^d)$ from quadratic forms on $\Z^d$ to bilinear forms. The inverse of this transformation associates to a bilinear form $b$ the quadratic form $q(x) = \frac12 b(x,x)$, which is in general not integral. Thus the map $\QF(\Z^d) \to \Bil(\Z^d)$ is only an injection and not a bijection.

As a quadratic form on $\Z^d$ is just a quadratic polynomial in $d$ variables, we can view it as a class in $H^4((\CP^{\infty})^d; \Z)$. Using the representability of cohomology, we obtain a natural isomorphism
\[\QF(\Z^d)\cong [(\CP^{\infty})^d, K(\Z,4)] \cong [BU(1)^d, K(\Z,4)].\]

In \cite[Section 5.1]{LurieSurvey}, Lurie sketches a construction which associates to each class in $q\in H^4(BG; \Z)$ a derived line bundle $\Lscr_q$ on a derived stack $\MM_G$ associated with $G$. This is a special case of his theory of \emph{$2$-equivariant elliptic cohomology}. More precisely, the  $\mathbb{G}_m$-torsor on $\MM_G$ associated to $\Lscr_q$ is the geometric object Lurie's theory associated to a certain Lie-$2$-group, namely the central extension of $G$ by $BU(1)$ classified by the map $BG\to K(\Z,4)$ associated to $q$. In the case of $G =U(1)^d$, the derived stack $\MM_G$ specializes to $\Escr^{\times_{\Mscr} d}$, and we may view $q$ as a quadratic form on $\Z^d$. We expect that $\Lscr_q$ coincides in this case with our derived Looijenga line bundle $\Lscr_b$ for the bilinear form $b$ associated to $q$ (see also \cref{rem:Poincare}). 

The inclusion $\QF(\Z^d)\to \Bil(\Z^d)$ can actually be realized on the level of topology. There is a fiber sequence 
     \[BString \to BO \to P^4BO,\]
     where $P^4BO$ is the fourth Postnikov stage, having homotopy groups $\pi_1BO = \Z/2$, $\pi_2BO = \Z/2$, $\pi_3BO = 0$, $\pi_4BO = \Z$ and all higher homotopy groups vanishing. The following result is presumably already known. 

\begin{prop}\label{prop:bilP4}
    The inclusion $\QF(\Z^d)\to \Bil(\Z^d)$ is naturally isomorphic to the map $[BU(1)^d, K(\Z,4)] \to [BU(1)^d, P^4BO]$. 
\end{prop}
\begin{proof}[Sketch of proof: ]
    We define $BO\langle 2,3,4\rangle$ to be the universal cover of $P^4BO$. As $BU(1)^d$ is simply-connected, the map 
    $$[BU(1)^d, BO\langle 2,3, 4\rangle] \to [BU(1)^d, P^4BO]$$ 
    induced by $BO\langle 2,3,4\rangle \to P^4BO$ is a bijection. Thus, we obtain a short exact sequence 
    \[0 \to H^{4}(BU(1)^d; \mathbb{Z})\to [BU(1)^d, P^{4}BO]\to H^{2}(BU(1)^d; \mathbb{Z}/2)\to 0.\]
    Utilizing that the non-trivial k-invariant of $BO\langle 2,3,4\rangle$ is $\beta\mathrm{Sq}^2\colon K(\Z/2,2)\to K(\Z,5)$ (for $\beta\colon K(\Z/2,4)\to K(\Z,5)$ the Bockstein), we obtain a map into the short exact sequence
    \[0\to H^{4}(BU(1)^d; \mathbb{Z})\to H^{4}(BU(1)^d; \mathbb{Z})\to H^4(BU(1)^d;\mathbb{Z}/2)\to 0,\]
    which is the squaring map $\mathrm{Sq}^2$ on $H^{2}(BU(1)^d; \mathbb{Z}/2)$ and the identity on $H^{4}(BU(1)^d; \mathbb{Z})$. Thus, $[BU(1)^d, P^{4}BO]$ is the subset of $\QF(\Z^d)$, consisting of those quadratic forms whose mod-$2$ reduction is a square (i.e.\ those where the coefficients of mixed terms are even). Sending such a quadratic form $q$ to $\frac12(q(v+w)-q(v)-q(w))$ defines a isomorphism from this subgroup to $\Bil(\Z^d)$. 
\end{proof}
In \cite{AndoBlumbergGepner}, it is explained that one can utilize the string-orientation of $\TMF$ to twist (non-equivariant) $\TMF$-cohomology of some space $X$ by maps $X\to P^4BO$. (The key point is that $B\mathrm{String}$ is the fiber of $BO\to P^4BO$.) Thus, it is plausible to assume that $T$-equivariant $\TMF$ (for a torus $T$) can be twisted by maps $BT \to P^4BO$. In personal communication, Lurie indicated that this should follow indeed from an \emph{equivariant} string orientation.\footnote{This seems to presume the statement that the correct $T$-equivariant classifying space of $BString$ is the pullback of $B_TO \times_{\mathrm{Map}(ET, BO)} \mathrm{Map}(ET, BString)$.} Our derived Looijenga line bundles $\Lscr_b$ allow to define the twisting associated to a map $BT \to P^4BO$ (corresponding to a bilinear form $b$) without recourse to an equivariant string orientation. (See also the end of the previous subsection) \\

There is a further viewpoint, based on \emph{anomalies} of quantum field theories. We will use a principle describing the space of anomalies in homotopical terms, which goes back to work of many people including Kapustin, Freed and Hopkins; for background and explanations we refer to \cite{FreedLectures} and the earlier physics papers \cite{Kapustin, KapustinThorngrenTurzilloWang}. In contrast to the purely mathematical earlier part of this section, we will employ a mixture of mathematical and physical language in the rest of the section. For a homomorphism $H\to O$ to the infinite orthogonal group, we will denote by $MTH$ the Thom spectrum whose $k$-th homotopy group classifies $k$-dimensional manifolds with \emph{tangential} $H$-structure (in contrast to the more standard $MH$ classifying manifolds with $H$-structure on the stable \emph{normal} bundle). We are mainly interested in the case of $H=Spin \times G\xrightarrow{\pr_1}Spin\to O$. It is equivalent to put a spin-structure on tangent or normal bundle, and thus $MTH \simeq MH \simeq MSpin\sm G_+$ in this case.  

\begin{principle}\label{principle}
The space of anomalies of unitary $d$-dimensional quantum field theories on space-times with tangential structure defined by $H\to O$ is equivalent to the mapping space $\mathrm{Map}(MTH,(I_{\Z}\mathbb{S})[d+2])$ from the Thom spectrum $MTH$ to the shifted Anderson dual of the sphere. 

In particular, if we consider spin space-time and an internal $G$-symmetry (so $H=G\times Spin$), the space of anomalies is $\mathrm{Map}(MSpin \sm BG_+,(I_{\Z}\mathbb{S})[d+2])$.\footnote{More precisely, \cite{FreedLectures} and \cite{FreedHopkins} work in this context in the Wick-rotated setting, where it is more appropriate to speak of \emph{reflection-positive} instead of \emph{unitary}.  As explained in \cite[Section 11.4]{FreedLectures}, any $d$-dimensional reflection-positive spin theory should give rise to an anomaly, which is described by an invertible truncated reflection-positive $(d+1)$-dimensional spin theory. Here, \emph{truncated} means that we only evaluate the (extended) theory on manifolds of dimension at most $d$. We will only consider anomalies where this truncated invertible theory is extended to an actual invertible reflection-positive $(d+1)$-dimensional spin theory. We call the space of such the \emph{space of anomalies of unitary $d$-dimensional spin-theories}. (As there is no generally accepted definition of a non-topological quantum field theory, this space is only defined on the physics-level of rigour.) \cite[Ansatz 8.71]{FreedLectures} and \cite[Conjecture 8.37]{FreedHopkins} suggest that this space is equivalent to $\mathrm{Map}(MSpin, I_{\Z}(\mathbb{S})[d+2])$ (and to $\mathrm{Map}(MTH, I_{\Z}(\mathbb{S})[d+2])$ for more general tangential structures). A key piece of evidence is that invertible reflection-positive $(d+1)$-dimensional spin theories map to this space (as torsion classes in $\pi_0$).} 
    \end{principle}
 For Anderson duality, we refer to  e.g.\ \cite[Section 2]{HeardStojanoska}. We recall here only that for a spectrum $X$ with finitely generated homotopy groups, $\pi_{-k}I_{\Z} X$ is isomorphic to the sum of the (dual of the) free part of $\pi_kX$ and the (dual of the) torsion of $\pi_{k-1}X$. Moreover, $I_{\Z}X$ is the \emph{mapping spectrum} $\mathrm{map}(X, I_{\Z}\mathbb{S})$. 

\begin{prop}\label{prop:P4anomaly}
    There is an equivalence of infinite loop spaces from $P^4(BO\times \Z)$ to the space of anomalies $\mathrm{Map}(MSpin,(I_{\Z}\mathbb{S})[4])$ of unitary $2$-dimensional spin-theories. 
\end{prop}
\begin{proof}For a spectrum $X$, denote by $X_{[k,\dots, l]}$ the truncation of $X$ so that we have $\pi_iX_{[k,\dots, l]} = \pi_iX$ for $k\leq i\leq l$ and it vanishes otherwise. Upon taking associated infinite loop spaces, our claim follows from the equivalence of spectra $KO_{[0,\dots,4]}$ with $(I_{\Z}MSpin)[4]_{[0, \dots, 4]} \simeq (I_{\Z}MSpin)_{[-4, \dots, 0]}[4]$, which we will now establish. Here, $KO$ denotes the spectrum of real K-theory, satisfying $\Omega^{\infty}KO\simeq BO\times\Z$. 

The Atiyah--Bott--Shapiro orientation $MSpin\to KO$ is an isomorphism on $\pi_i$ for $-1\leq i\leq 4$. Thus, $I_{\Z}MSpin \to I_{\Z}KO$ is an isomorphism on $\pi_i$ for $-4\leq i\leq 0$. By \cite[Theorem 8.1]{HeardStojanoska} and Bott periodicity, $I_{\Z}KO \simeq KO[4]\simeq KO[-4]$. Thus, 
\[(I_{\Z}MSpin)_{[-4,\dots, 0]} \simeq (I_{\Z}KO)_{[-4,\dots, 0]} \simeq (KO[-4])_{[-4,\dots,0]}.\]
Taking four-fold shift yields our claim. 
\end{proof}

\begin{remark}
    The role of $P^4BO\times \Z$ (or at least $P^4BSO$) is well known in the context of the \emph{level} of Chern--Simons theories; see e.g.\ \cite[Section 4.4]{jenquin2005classical}. 
\end{remark}

Combining \cref{prop:bilP4}, \cref{principle} and \cref{prop:P4anomaly} with our construction of derived Looijenga line bundle yields for a torus $T = U(1)^d$ the following construction: 
\begin{construction}
    For every anomaly for unitary $2$-dimensional spin-theories with $T$-symmetry, i.e.\ for every map $BT \to \mathrm{Map}(MSpin, (I_{\Z}\mathbb{S})[4])$, we obtain a derived line bundle on $\Escr^{\times_{\Mscr}d}$. 
\end{construction}

Assume for the rest of the section that all quantum field theories are unitary and spin. One can refine the Segal--Stolz--Teichner program to conjecture that $2$-dimensional minimally supersymmetric quantum field theories with $G$-symmetry with fixed anomaly are classified by twisted $G$-equivariant $\TMF$-\emph{cohomology} of a point. Here, we use that according to \cref{principle} and \cref{prop:P4anomaly} possible anomalies are classified by maps $BG\to P^4 BO$ and $G$-equivariant $\TMF$ can \emph{conjecturally} be twisted by such maps. We have learned from Mayuko Yamashita that in contrast, 2-dimensional minimally supersymmetric quantum field theories living at the boundary of a $G$-Chern--Simons theory of fixed level should be classified by twisted $G$-equivariant $\TMF$-\emph{homology} of a point (with twist corresponding to the level). 

In the case that $G=T$ is a torus, the constructions and results of this section make this more concrete. In this case, \cref{prop:bilP4} allows identifying the anomaly with a bilinear form $b$. As discussed above, $L^b$ and $L_b$ may be viewed as twisted $T$-equivariant $\TMF$-cohomology and $\TMF$-homology of a point, respectively. In particular, $L_b$ should classify 2-dimensional minimally supersymmetric quantum field theories living at the boundary of an abelian Chern--Simons theory of level $b$. We will comment more about this in the last section, in particular in \cref{sec:ChernSimons}.

\section{Properties of Looijenga line bundles} \label{sec: Properties of Looijenga line bundles}
In this section, we will show some basic properties about derived Looijenga line bundles and the associated modules $L_b$ and $L^b$. We will use the notation introduced in \cref{sec: bilinear to TMF}.
\subsection{Basic Properties} \label{sec: basic properties}
\begin{lemma}\label{lem:Normalization}
    We have an equivalence \[e^*\Lscr_{(1)}\simeq e^*\Otop_{\Escr}(e)[-2] \simeq \Otop_{\Mscr},\] where $(1)$ denotes the bilinear form $(x,y)\mapsto xy$ on $\Z$ and $e\colon \Mscr\to \Escr$ is the unit section.
\end{lemma}
\begin{proof}
    The first equivalence comes from the definition property $\Lscr_{(1)}\simeq \Otop_{\Escr}(e)[-2]$. For the second equivalence, it will be convenient to use the language of $U(1)$-equivariant elliptic cohomology from \cite{GepnerMeier} to show this result.

    According to \cite[Lemma 8.1]{GepnerMeier}, we obtain $\Otop_{\Escr}(-e)$ as the reduced $U(1)$-equivariant elliptic cohomology $\Oscr_{\Escr}^{S^{\C}}$ of the Riemann sphere $S^{\C}$ with the tautological $U(1)$-action. It is shown in \cite[Section 2.3]{GepnerMeier2} that $e^*\Oscr_{\Escr}^X \simeq \Oscr_{\Mscr}^{\res X}$, where $\res X$ is the restriction of a pointed $U(1)$-space to a non-equivariant pointed space. Assuming this, we see that 
    \[e^*\Oscr_{\Escr(-e)} \simeq \Oscr_{\Mscr}^{S^2} \simeq \Otop_{\Mscr}[-2],\] 
    and hence $e^*\Otop_{\Escr(e)} \simeq \Otop_{\Mscr}[2]$ as claimed. 
\end{proof}

\begin{lemma}\label{lem:directsum}
Let $b$ and $c$ be defined on $\Z^d$ and $\Z^{d'}$, respectively. Consider the projections $\pr_1$ and $\pr_2$ from $\Z^d\oplus \Z^{d'}$ to $\Z^d$ and $\Z^{d'}$, and denote the corresponding projections from $\Escr^{\times_{\Mscr} (d+d')}$ to $\Escr^{\times_{\Mscr} d}$ and $\Escr^{\times_{\Mscr} d'}$ by the same symbols. Then 
\[\Lscr_{b\oplus c} \simeq \pr_1^*\Lscr_b \tensor \pr_2^*\Lscr_c.\]
\end{lemma}
\begin{proof}
    This follows from functoriality and additivity of derived Looijenga line bundles since $b\oplus c \cong \pr_1^*b + \pr_2^*c$.
\end{proof}

\begin{lemma}\label{lem:stable}
    Suppose that two bilinear forms $b$ and $b'$ are stably isomorphic, i.e.\ there is an isomorphism $b \oplus c \cong b'\oplus c$. Then $\Lscr_b \simeq \Lscr_{b'}$
\end{lemma}
\begin{proof}
    Let $b$ and $b'$ be defined on $\Z^d$ and $c$ be defined on $\Z^{d'}$. Consider the projections $\pr_1\colon \Escr^{\times_{\Mscr} d+d'} \to \Escr^{\times_{\Mscr} d}$ and $\pr_2\colon \Escr^{\times_{\Mscr} d+d'} \to \Escr^{\times_{\Mscr} d'}$ and the unit map $\epsilon\colon \Escr^{\times_{\Mscr} d} \to \Escr^{\times_{\Mscr} d+d'}$. By the last lemma, we have 
    \[\pr_1^*\Lscr_b \otimes \pr_2^*\Lscr_c \simeq \Lscr_{b\oplus c} \simeq \Lscr_{b' \oplus c} \simeq \pr_1^*\Lscr_{b'} \otimes \pr_2^*\Lscr_c.\]
    Applying $\epsilon^*$, we obtain $\Lscr_b\tensor \epsilon^*\pr_2^*\Lscr_c \simeq \Lscr_{b'}\tensor \epsilon^*\pr_2^*\Lscr_c$. Tensoring with the inverse of $\epsilon^*\pr_2^*\Lscr_c$, yields $\Lscr_b\simeq \Lscr_{b'}$.  
\end{proof}

\begin{lemma} \label{lem: direct sum}
    Let $b \oplus b'$ the direct sum of two bilinear forms. Then 
    \begin{align*}L^{b\oplus b'}&\simeq L^b \tensor_{\TMF} L^{b'}, \text{ and}\\
    L_{b\oplus b'}&\simeq L_b \tensor_{\TMF} L_{b'}.
    \end{align*} 
\end{lemma}
\begin{proof}
Denote by $\Escr^b$ and $\Escr^{b'}$ the products of elliptic curves $\Lscr_b$ and $\Lscr_{b'}$ live on. Consider the cartesian square
\[
\xymatrix{\Escr^b \times_{\Mscr} \Escr^{b'} \ar[r]^{p'}\ar[d]^{p} & \Escr^b \ar[d]^{q} \\
\Escr^{b'} \ar[r]^{q'} & \Mscr
}
\]
Using the projection formula and the push-pull formula, we have the following chain of equivalences: 
\begin{align*}
    q_*'p_*\Lscr_{b \oplus b'} &\simeq q_*'p_*((p')^*\Lscr_{b} \tensor p^*\Lscr_{b'}) \\
    &\simeq q_*'((p_*(p')^*\Lscr_b)\tensor \Lscr_{b'}) \\ 
    &\simeq q_*'(((q')^*q_*\Lscr_b )\tensor \Lscr_{b'} ) \\
    &\simeq q_*\Lscr_b \tensor q_*'\Lscr_{b'} .
\end{align*}
Since pushforward preserves global sections, $\Gamma(q_*'p_*\Lscr_{b \oplus b'}) \simeq L^{b\oplus b'}$. Furthermore, global sections are symmetric monoidal on $\Mscr_{ell}$ by \cite{MathewMeier}. Thus, we obtain $L^{b\oplus b'} \simeq L^{b} \tensor_{\TMF} L^{b'}$. 

To obtain $L_{b\oplus b'}\simeq L_b \tensor_{\TMF} L_{b'}$, observe that the dual $\Lscr_b^{\vee}$ is by additivity equivalent to $\Lscr_{-b}$ and thus $L_b \simeq (L^{-b})^{\vee}$. Thus, 
\begin{align*}L_{b\oplus b'} &\simeq (L^{-b\oplus -b'})^{\vee}\\
&\simeq (L^{-b}\tensor_{\TMF}L^{-b'})^{\vee} \\
&\simeq (L^{-b})^{\vee}\tensor_{\TMF}(L^{-b'})^{\vee}\\
&\simeq L_b \tensor_{\TMF}L_{b'}.\qedhere\end{align*}
\end{proof}

\subsection{Adding trivial summands and examples} \label{sec: examples computations}
\begin{example}\label{ex:Oke}
    By construction, the bilinear form $(n)\colon (v,w) \mapsto nvw$ of rank $1$ is sent to $\Lscr_{(n)} = \Otop_{\Escr}(ne)[-2n]$. Thus, it is important to compute the global sections of $\Otop_{\Escr}(ne)$, i.e.\ the line bundle describing functions where we allow poles of order $n$ at $e$ (or requiring zeros of order $-n$ if $n$ is negative). 
    
    In some cases, we consider more generally an elliptic curve $p\colon E \to Y$ over any suitable base $Y$, and discuss the pushforward $p_*\Otop_E(ne)$; this has the same global sections as $\Otop_E(ne)$.     
    As a consequence of \cite[Theorem 10.1]{GepnerMeier} and its proof one obtains
    \begin{itemize}
        \item $p_*\Otop_E = p_*\Otop_E(0\cdot e) \simeq \Otop_Y \oplus \Otop_E[1]$
        \item $p_*\Otop_E(e) \simeq \Otop_Y$
        \item $p_*\Otop_E(-e) \simeq \Otop_Y[1]$. 
    \end{itemize}
     In particular, 
     \begin{align*}
         L^{(-1)} \simeq \TMF[3], L^{(0)} \simeq \TMF \oplus \TMF[1], &\text{ and } L^{(1)} \simeq \TMF[-2], \text{ and}\\
         L_{(-1)} \simeq \TMF[2], L_{(0)} \simeq \TMF \oplus \TMF[-1], &\text{ and } L_{(1)} \simeq \TMF[-3].
     \end{align*}
     We will be more explicit about these equivalences and their relationship in \cref{sec:exampleL0}. We will denote $L_{(n)}$ by $U(1)_n$ in \cref{sec:ChernSimons} due to their relationship with the $U(1)$ Chern--Simons theory at level $n$.
    
    The computation of $\Gamma(\Otop_E(ne)) = L^{(n)}[2n] $ for $|n|>1$ is more subtle. In \cite{BauerJacobi}, we will produce for $n\geq 1$ a cofiber sequence\footnote{Roughly speaking, a cofiber sequence is a sequence $A \to B \to C \to A[1] = \Sigma A$ in some derived context (such as spaces, spectra, $\TMF$-modules or $\Otop_Y$-modules) such that taking homotopy groups results in a long exact sequences 
    \[\cdots \to \pi_*A \to \pi_*B \to \pi_*C \to \pi_{*-1}A \to \cdots.\] 
    (In spaces, we would have to take indeed homology groups instead of homotopy groups.) A standard example in spaces is $S^k \to X \to X\cup_{S^k}D^{k+1} \to S^{k+1}$, where $X\cup_{S^k}D^{k+1}$ is what we get from $X$ by attaching a $k+1$-cell along the map $S^k\to X$. } 
    \begin{equation}\label{eq:BauerSequence} \TMF\sm \CP^{n-1}[1] \to \TMF \to \Gamma(\Otop_{\Escr}(ne)) \to \TMF \sm \CP^{n-1}[2],\end{equation}
    where the first map is induced from composing the inclusion $\CP^{n-1}[1] \to \CP^{\infty}[1]$ with the stable transfer map $\CP^{\infty}[1] \to S^0$. (The latter can be defined in many ways; if we identify $\CP^{\infty}$ with $BU(1)$, this stable map is part of the standard package of stable equivariant homotopy theory and can e.g.\ be constructed from the Adams isomorphism or equivariant Atiyah duality; see e.g.\ \cite[Appendix A]{BauerJacobi}.) In \cref{thm:Lbduals}, we will show that $L^{(-n)}$ is the shifted dual $(L^{(n)})^{\vee}[1]$. From the preceding cofiber sequence, we thus obtain for $n\geq 1$, by dualizing and shifting by $2n+1$, a cofiber sequence
    \[\TMF^{\CP^{n-1}}[2n-1] \to L^{(-n)} \to \TMF[2n+1] \to \TMF^{\CP^{n-1}}[2n].\]
    We remark that the homotopy groups of $\TMF^{\CP^{n-1}}$ are the reduced $\TMF$-\emph{co}homology of $\CP^{n-1}$, while the homotopy groups of $\TMF\sm \CP^{k-1}$ are the reduced $\TMF$-homology. The corresponding results for $L_{(k)}$ follow since $L_{(k)}\simeq L^{(k)}[-1]$ by \cref{thm:Lbduals}. 

    The easiest non-trivial case is $n=2$. By \eqref{eq:BauerSequence}, $L^{(2)}$ is the $[-4]$-fold shift of the cofiber of a map $\TMF[3]\simeq \TMF\sm \CP^1[1] \to \TMF$. The map is induced by the composite 
    \[S^3\cong \CP^1[1] \to \CP^{\infty}[1] \to S^0,\]
    where the last map is the transfer. This composite will be identified with the Hopf map $\nu\colon S^3\to S^0$ at the end of \cref{ex:CP2}. Thus, $L^{(2)} \simeq \TMF \sm \mathrm{Cone}(\nu)[-4]$. Since the dual of $\mathrm{Cone}(\nu)$ is $\mathrm{Cone}(\nu)[-4]$, we deduce 
    \[U(1)_2=L_{(2)} \simeq \TMF\sm\mathrm{Cone}(\nu)[-5] \;\text{ and }\; U(1)_{-2}=L_{(-2)} \simeq \TMF \sm \mathrm{Cone}(\nu).\]
    
    The above implies that upon inverting $6$, the $\TMF$-module $\Gamma(\Otop_{\Escr}(ke))$ is for $k\geq 1$ equivalent to a sum of $k$ shifted copies of $\TMF$, but the torsion will be quite different. See \cite{BauerJacobi} for more detailed information about the homotopy groups of $\Gamma(\Otop_{\Escr}(ke))$ and their interpretation as topological Jacobi forms. 
        \end{example}
     
    The first part of the example allows us to prove the following crucial lemma. 

    \begin{lemma}\label{lem:KirbyMove}
        Let $b$ be a bilinear form of rank $d$. We have canonical equivalences $L_b[-3] \simeq L_{b \oplus (1)}$ and $L_b[2] \simeq L_{b\oplus (-1)}$. 
        \end{lemma}
        \begin{proof}Since $L_b \simeq (L^{-b})^{\vee}$, it will be enough to show that $L^b[-2] \simeq L^{b \oplus (1)}$ and $L^b[3] \simeq L^{b\oplus (-1)}$.
        
        We will apply results from \cref{ex:Oke} in the case that $Y = \Escr^{\times_{\Mscr} d}$ and $E = \Escr^{\times_{\Mscr} (d+1)}$. (The map $p$ is projection onto the first $d$ coordinates.) By construction, the line bundle $\Lscr_{b\oplus (k)}$ is the external tensor product of $\Lscr_q$ on $Y$ and $\Lscr_{(k)}$ on the last factor $\Escr$. Considering the diagram
        \[
        \xymatrix{
        E \ar[r]^{\pr_{d+1}}\ar[d]^p & \Escr \ar[d]^{p'} \\
        Y \ar[r]^{f} & \Mscr
        }
        \]
        we can also view this external tensor product as $p^*\Lscr_b \tensor_{\Otop_E} \pr_{d+1}^*\Lscr_{(k)}$. We can use the projection formula and a push-pull formula to obtain
        \begin{align*}
            p_*(p^*\Lscr_b \tensor_{\Otop_E} \pr^*\Lscr_{cvw}) &\simeq \Lscr_b \tensor_{\Otop_Y} p_*\pr_{d+1}^*\Lscr_{(k)} \\
            &\simeq \Lscr_b \tensor_{\Otop_Y} f^*(p')_*\Lscr_{(k)}
        \end{align*}
        Using \cref{ex:Oke} and $\Lscr_{(k)} \simeq \Otop_{\Escr}(ce)[-2k]$, we see that that $(p')_*\Lscr_{(k)}$ is $\Otop_{\Mscr}[-2]$ if $k=1$ and $\Otop_{\Mscr}[3]$ if $k= -1$. Taking $f^*$ results correspondingly in $\Otop_{Y}[-2]$ and $\Otop_{Y}[3]$. Thus, $p_*\Lscr_{b\oplus (k)} \simeq \Lscr_b[-2]$ or $\Lscr_b[3]$, depending on whether $k=\pm 1$. As pushforward does not change global sections, we see that $L_{b\oplus(k)}$ is $L^b[-2]$ or $L^b[3]$, depending on whether $k=\pm 1$.
        \end{proof}

        \begin{example}\label{ex:unimodular}
    Let $b$ be an unimodular bilinear form (i.e.\ the bilinear form defines an isomorphism of $\Z^d$ with its dual). The corresponding real bilinear form is diagonalizable with $b^+$ many $1$ and $b^-$ many $-1$ on the diagonal.
    
    Every such unimodular $b$ is by \cite[Chapter V.2.1]{SerreArithmetic} stably isomorphic to a bilinear form $d$ represented by a diagonal matrix with entries $\pm 1$. Necessarily there are $b^+$ many $1$ and $b^-$ many $-1$.  \cref{lem:stable} implies $\Lscr_b \simeq \Lscr_d$ and hence $L^b\simeq L^d$ and $L_b\simeq L_d$.  \cref{lem:KirbyMove} implies inductively that $L^b \simeq \TMF[3b^- - 2b^+]$ and $L_b \simeq \TMF[2b^--3b^+]$. 
\end{example}

\begin{remark}\label{rem:Poincare}
    As mentioned in Section \ref{sec: relation to 2-equivariant TMF}, we conjecture that our construction of line bundles should coincide with that of Lurie in the case of quadratic forms. Unfortunately, the details of Lurie's construction are not yet available, but he states that his construction is determined by $\Lscr_{xy}$ being the Poincaré line bundle on $\Escr \times_{\Mscr} \Escr$ (with $xy$ standing for the bilinear form associated to this quadratic form, i.e.\ that given by the matrix $\begin{pmatrix}0&1\\1&0\end{pmatrix}$). This is the line bundle witnessing the autoduality of $\Escr$. But we can show that our construction of $\Lscr_{xy}$ agrees with the Poincaré line bundle at least on underlying classic (i.e.\ non-derived) line bundles and up to powers of $\omega$. We can rewrite the matrix associated to the bilinear form $v_1w_2+v_2w_1$:
    \[\begin{pmatrix}0&1\\1&0\end{pmatrix} = A^T\begin{pmatrix}1&0&0\\0&-1&0\\0&0&-1\end{pmatrix}A \quad \text{ with }\quad A = \begin{pmatrix} 1&1\\-1&0\\0&1\end{pmatrix}.\] 
    Thus, to compute $\Lscr_{xy}$, we have to pull back the line bundle 
    \[\pr_1^*\Oscr(e) \tensor \pr_2^*\Oscr(-e) \tensor \pr_2^*\Oscr(-e)\]
    along the morphism $a\colon \Escr^{\times_{\Mscr} 2} \to \Escr^{\times_{\Mscr} 3}$ induced by $A$. We compute 
    \[A\begin{pmatrix}x\\y\end{pmatrix} = \begin{pmatrix}x+y\\-x\\y\end{pmatrix}.\]
    The first entry is zero iff $x= -y$, the second iff $x=0$, the third iff $y = 0$. Thus, pulling back the divisor $\pr_1^*e$ gives the antidiagonal, pulling back $\pr_2^*e$ the divisor $e\times \Escr$ and pulling back $\pr_3^*e$ the divisor $\Escr\times e$. Thus, the underlying classical line bundle of $\Lscr_{xy}$ is associated to the divisor.  $\Delta^{\mathrm{anti}} - e\times \Escr  -\Escr\times e$. According to \cite[Section 9.4]{Polishchuk}, this agrees with the divisor of the Poincaré line bundle, at least when working over a field instead of $\Mscr$ as a base. Two line bundles on a product of universal elliptic curves over $\Mscr$ are isomorphic up to a power of $\omega$ if they are after base change along an arbitrary morphism $\Spec k\to \Mscr$ from a field.\footnote{Indeed, given two such line bundles $\Lscr$ and $\Lscr'$, the pushforward of the Hom-sheaf $\mathcal{H}om(\Lscr, \Lscr')$ to $\Mscr$ will be a line bundle by ``cohomology and base change'' and thus of the form $\omega^{\omega m}$; see e.g.\ \cite[Corollary 12.9]{Hartshorne}, \cite[Lemma A.7]{MeierDecompositions}, \cite{FultonOlsson}. Thus, the pushforward of $\mathcal{H}om(\Lscr\tensor \omega^{\tensor m}, \Lscr')$ is $\Oscr_{\Mscr}$. The morphism $\Lscr\tensor \omega^{\tensor m}\to \Lscr'$ corresponding to $1\in \Gamma(\Oscr_{\Mscr})$ is an isomorphism because it is one in the fiber over any $\Spec k\to \Mscr$.} Thus, we see indeed that our $\Lscr_{xy}$ agrees with Lurie's derived Poincare line bundle on the level of classical line bundles and up to powers of $\omega$. 
\end{remark}

        \section{Transfers and duality} \label{sec: cobordisms}
        
        This section has two goals. First, we explain how $L^b$ has contravariant functoriality in the bilinear form $b$ and $L_b$ has covariant functoriality. Motivated by equivariant homotopy theory, we call these functorialities \emph{restriction} and \emph{transfer}. Second, we explain how $L^b$ and $L_b$ are equivalent up to a shift and make this equivalence explicit in the case of $b = (0)$. 

        \begin{defi}
        Let $f\colon F \to G$ be a morphism of finitely generated free abelian groups and let $b$ be a symmetric bilinear form on $G$. We have \emph{restriction} morphisms
        \[\res_f\colon \Lscr_b \to f_*\Lscr_{f^*b}\quad \text{ and }\quad \res_f\colon L^{b} \to L^{f^*b},\]
        where the first is adjoint to the equivalence $f^*\Lscr_b \simeq \Lscr_{f^*b}$ and the latter is obtained by applying global sections. 

        We further define the \emph{transfer} 
        \[\tr_f\colon L_{f^*b} \to L_b\]
        as the dual of $\res_f\colon L^{-b}\to L^{f^*(-b)}$, using the equivalences $L_b = (L^{-b})^{\vee}$ and $L_{f^*b} = (L^{f^*(-b)})^{\vee}$. 
        \end{defi}

        Essentially by construction, restriction and transfer are compatible with direct sums and compositions. Let us state the compatibility with direct sums explicitly. 
        \begin{lemma}\label{lem:RestrictionTransferDirectSum}
            Given a homomorphism $f\colon F\to G$ of finitely generated free abelian groups and symmetric bilinear forms $b$ and $b'$ on $G$, the square
            \[ \xymatrix{
L^{f^*b}\otimes_{\TMF}L^{f^*b'} \ar[rr]^{\res_f\otimes\res_f}\ar[d] & &
L^b\otimes_{\TMF}L^{b'}\ar[d]\\
L^{f^*(b\oplus b')} \ar[rr]^{\res_f} && L^{b\oplus b'} 
}
            \]
            commutes. Here, the vertical equivalences are those from \cref{lem: direct sum}. Dually, the square involving $L_b$, $L_{b'}$ and $\tr_f$ commutes as well.
        \end{lemma}

        \begin{construction}\label{constr:dhedafunction}
            Let $b$ be a unimodular bilinear form on $\Z^d$ with $b^{\pm}$ many eigenvalues $\pm 1$. By \cref{ex:unimodular}, $L^b \simeq \TMF[3b^- - 2b^+]$ and $L_b \simeq \TMF[2b^--3b^+]$. Along $f\colon 0\to \Z^d$, the bilinear form $b$ restricts to the trivial bilinear form $()$. Thus, $\res_f$ defines a $\TMF$-linear map
            \begin{equation}\label{eq:db}\TMF[3b^- - 2b^+] \simeq L^b \to L^{()}\simeq \TMF,\end{equation}
            corresponding to an element $\mathfrak{d}_b\in \pi_{3b^--2b^+}\TMF$; the same element corresponds to the dual map
            \[ \TMF = L_{()} \xrightarrow{\tr_f} L_{-b} \simeq \TMF[2b^+-3b^-].\]
            As we did not specify the equivalences in \eqref{eq:db}, the map in \eqref{eq:db} is only well-defined up to $\TMF$-linear automorphisms of $\TMF$. Since the invertible elements in $\pi_0\TMF\cong \Z[j]$ are just $\pm 1$, this means that $\mathfrak{d}_b$ is well-defined up to sign. 
        \end{construction}

\subsection{Duality}
Our main goal in this section is to show that $L^b$ and $L_b$ are dual up to a shift. Recall that $L^b$ are the global sections of the derived line bundle $\Lscr_b$ on $\Escr^{\times_{\Mscr}d}$ and $L_b$ is the dual of the global sections of the dual $\Lscr_{-b}$ of $\Lscr_b$. Thus, we have to investigate how taking dual and global sections interact. 

As a warm-up, we consider instead of a derived line bundle on a stack a usual line bundle $\Lscr$ on a smooth projective variety $X$ of dimension $d$. Serre duality tells us that $H^i(X; \Lscr)^{\vee} \cong H^{d-i}(X;  \Lscr^{\vee}\tensor \Omega^d_X)$, where $\Omega^d_X$ is the highest exterior power of the sheaf of differentials. Thus, up to cohomological degree and the dualizing sheaf $\Omega^d_X$, dualizing plays well with taking derived global sections, i.e.\ cohomology. 

We need to upgrade this in two ways: first, we do not have a variety over a field but a ``relative variety'' over the stack $\MM$. Secondly, we are not working in classical algebraic geometry, but in spectral algebraic geometry. Such a generalization is exactly the topic of \emph{spectral Grothendieck duality} as in \cite{SAG}. The change of cohomological degree in Serre duality will translate in our setting into a shift by the relative dimension (namely $d$) and the identification of the dualizing sheaf will result in another shift (by $-2d$), resulting in the overall result $L^b \simeq L_b[-d]$ in \cref{thm:Lbduals}.

Instead of just having a result of the global sections of $\Lscr_b$, we will give a more generally applicable result on more general pushforwards. The case of global sections follows by pushing forward to $\MM$ (since derived global sections on $\MM$ preserve duals), as explained in the proof of \cref{thm:Lbduals}.

         \begin{thm}\label{prop:duals}
            Let $f\colon F \to G$ be a homomorphism between finitely generated free abelian groups, and denote the induced morphism $\Escr \tensor F \to \Escr \tensor G$ by $f$ as well. Let $\Lscr$ be a quasi-coherent $\Otop_{\Escr\tensor F}$-module (e.g.\ a derived line bundle on $\Escr \otimes F$). Then there is a natural equivalence \[(f_*\Lscr)^{\vee} \simeq f_* (\Lscr^{\vee})[\rk G - \rk F].\] Here, $\Lscr^{\vee}$ denotes the dual $\mathcal{H}om_{\Otop_{\Escr\tensor F}}(\Lscr, \Otop_{\Escr\tensor F})$. 
        \end{thm}

  \begin{proof}[Proof of \cref{prop:duals}: ]
We first reduce to the connective case, i.e.\ we will show our proposition with $\Otop_{\Escr\tensor F}$ and $\Otop_{\Escr\tensor G}$ replaced by their connective covers $\tau_{\geq 0}\Otop_{\Escr\tensor F}$ and $\tau_{\geq 0}\Otop_{\Escr\tensor G}$. Applying $\otimes_{\tau_{\geq 0}\Otop_{\Escr\tensor G}}\Otop_{\Escr\tensor G}$ recovers the original proposition (using that the canonical map $\tau_{\geq 0}\Otop_{\Escr\tensor F}\otimes_{\tau_{\geq 0}\Otop_{\Escr\tensor G}}\Otop_{\Escr\tensor G} \to  \Otop_{\Escr\tensor F}$ is an equivalence). We will leave the connective covers implicit in the rest of the argument. 
            
            The theory of spectral Grothendieck duality from \cite{SAG} (and \cite[Corollary 6.4.2.7]{SAG} in particular) implies that it suffices to identify the relative dualizing complex $\omega^{\top}_{\Escr\tensor F/\Escr\tensor G}$ of 
            \[f\colon (\Escr\tensor F, \Otop_{\Escr\tensor F}) \to (\Escr\tensor G, \Otop_{\Escr\tensor G})\]
            with $\Otop_{\Escr\tensor F}\tensor S^{\rk G- \rk F}$. The transitivity of dualizing complexes (\cite[Corollary 6.4.2.8]{SAG}) implies that it suffices to establish the existence of an equivalence  $h\colon \Otop_{\Escr\tensor F}\tensor S^{-\rk F} \to \omega^{\top}_{\Escr\tensor F/\Mscr}$.

            By \cite[Proposition 5.2.3]{LurEllI},  $\omega^{\top}_{\Escr\tensor F/\Mscr} \simeq p^*e^* \omega^{\top}_{\Escr\tensor F/\Mscr}$ for 
            \[\label{eq:ep} \Mscr \xrightarrow{e} \Escr \tensor F \xrightarrow{p}\Mscr\]
            being unit and projection, respectively. Applying \cite[Corollary 6.4.2.8]{SAG} to this composition and noting $pe = \id_{\Mscr}$, we obtain 
            \begin{equation}\label{eq:transitivity}\Otop \simeq \omega^{\top}_{\Mscr/\Mscr} \simeq e^*\omega^{\top}_{\Escr\tensor F/\Mscr} \otimes \omega^{\top}_{\Mscr/\Escr\tensor F}.\end{equation}
  
            Thus, $e^*\omega^{\top}_{\Escr\tensor F/\Mscr}$ is an invertible $\Otop$-module. 
            By \cite[Theorem A]{MathewStojanoska} and the equivalence $\QCoh(\Mscr, \Otop) \simeq \Mod_{\TMF}$ from \cite{MathewMeier}, we know that the Picard group $\mathrm{Pic}(\Mscr, \Otop)\cong \Z/576$ is generated by $\Otop[1]$. Thus, every invertible $\Otop$-module is of the form $\Otop[n]$ for some $n$ (well defined modulo $576$). Thus, $\omega^{\top}_{\Escr \tensor F/\Mscr} \simeq \Otop_{\Escr}[n]$. The proof of \cref{thm:Lbduals} below shows that this implies that $L_b^{\vee} \simeq L_{-b}[n]$ as $\TMF$-modules. Recall from \cref{lem:KirbyMove} that $L_{b} \simeq  \TMF[-3\rk F]$ and $L_{-b} \simeq \TMF[2\rk F]$ if $b$ is a bilinear form on $F\cong \Z^{\rk F}$ represented by a diagonal matrix with ones on the diagonal. This implies that $n$ can be taken to be $-\rk F$. 
            \end{proof}

            \begin{remark}
                By \eqref{eq:transitivity}, the preceding proposition is equivalent to \begin{equation}\label{eq:omegaM}\omega^{\top}_{\Mscr/\Escr\tensor F} \simeq \Otop[\rk F].\end{equation}
                Moreover, the claimed equivalences $(f_*\Lscr)^{\vee} \simeq f_* (\Lscr^{\vee})[\rk G - \rk F]$ only depend on the choice of  equivalence \eqref{eq:omegaM} for each $F$. Since $\pi_0(\Gamma(\Otop))^\times \cong \Z[j]^\times = \{\pm 1\}$, we thus have an ambiguity at most up to a sign. This can be fixed as well by specifying the equivalence in a single example for each $F$.    \end{remark}

            \begin{remark}
                The proof of \cref{prop:duals} depends on \cite{MathewMeier} and \cite{GepnerMeier} (and thus, indirectly, on equivariant homotopy theory). We outline a possible alternative proof of \eqref{eq:omegaM} and thus \cref{prop:duals} without such a reliance:  The dualizing sheaf $\omega^{\top}_{\Mscr/\Escr\tensor F}$ agrees with $\omega^{\top}_{\Mscr/\widehat{\Escr\tensor F}}$, where $\widehat{\Escr\tensor F}$ denotes the completion of $\Escr \tensor F$ at $\Mscr$ (we skip here over subtleties of formal spectral algebraic geometry.) The orientation of $(\Escr, \Otop_{\Escr})$ (as in \cite{LurEllII}) identifies $\widehat{\Escr\tensor F}$ with the relative formal spectrum of the function spectrum of maps from $({\CP^{\infty}})^{\rk F}$ to   $\Otop$ on $\Mscr$. We expect that, analogously to \cite[Example 6.4.2.9]{SAG}, this implies that \[\omega^{\top}_{\Mscr/\widehat{\Escr\tensor F}}\simeq \Otop[2\rk F - \rk F] \simeq \Otop[\rk F].\]  
            \end{remark}

\begin{thm}\label{thm:Lbduals}
    Let $b$ be a bilinear form on a finitely generated abelian group $F$ of rank $d$. Then there is an equivalence of $\TMF$-modules of $L^b\simeq L_{-b}^{\vee}$ with $L_{b}[d]$. This equivalence can be chosen to be functorial in the bilinear form and is canonically defined up to sign. 
         \end{thm}
        \begin{proof} 
        We apply \cref{prop:duals} with $G = 0$ and $\Lscr = \Lscr_b$. By additivity, we know that $(\Lscr_b)^{\vee} \simeq \Lscr_{-b}$. If $f\colon \Escr \tensor F \to \Mscr$ is the projection, we thus obtain $(f_*\Lscr_b)^{\vee} \simeq f_*\Lscr_{-b}[-d]$. 

        From \cite{MathewMeier}, we know that the global sections functor from quasi-coherent $\Otop_{\Mscr}$-modules to $\TMF$-modules is symmetric monoidal (in fact, a symmetric monoidal equivalence). In particular, it preserves duals. As applying $f_*$ does not change global sections, we obtain indeed $(L^b)^{\vee} \simeq L^{-b}[-d]$, where now the dual is taken in $\TMF$-modules. Dualizing this equivalence, yields indeed $L^b \simeq (L^{-b}[-d])^{\vee} L_b[d]$. 
        \end{proof}

        This equivalence allows to see $L_b$ as both a covariant and contravariant functor in bilinear forms, via transfer and restriction.

\subsection{The example of $L_{(0)}$}\label{sec:exampleL0}
In this section, we will look in more detail at the example of $L^{(0)} \simeq \Gamma(\Otop_{\Escr})$, where $(0)$ denotes the zero bilinear form on $\Z$.. It was already mentioned in \cref{ex:Oke} that $L^{(0)} \simeq \TMF \oplus \TMF[1]$. By definition, $L_{(0)} \simeq (L^{(0)})^{\vee} \simeq \TMF\oplus \TMF[-1]$. Thus, \cref{thm:Lbduals} seems evident in this case. We will explain, however, below that the canonical equivalence from \cref{thm:Lbduals} is \emph{not} the ``obvious'' equivalence. 

For this, we first need to recall \emph{how} $\Gamma(\Otop_{\Escr})$ was identified with $\TMF \oplus \TMF[1]$ in \cite{GepnerMeier}. The map $\TMF = \Gamma(\Otop) \to \Gamma(\Otop_{\Escr})$ is pullback along $p\colon \Escr \to \Mscr$. The map $\TMF[1] \to \Gamma(\Otop_{\Escr})$ is the \emph{degree-shifting transfer}, which is constructed using \emph{$U(1)$-equivariant elliptic cohomology}. As mention earlier, (reduced) $U(1)$-equivariant elliptic cohomology is a functor 
\[\Oscr_{\Escr}^{(-)}\colon (\text{finite pointed }U(1)\text{-CW complexes})^{\op} \to \QCoh(\Escr, \Otop_{\Escr}).\]
Two key examples are $\Oscr_{\Escr}^{S^0} = \Otop_{\Escr}$ and $\Oscr_{\Escr}^{U(1)_+} = e_*\Otop$, where the plus stands for adding a disjoint basepoint and $e\colon \Mscr\to \Escr$ is the inclusion of the unit of the universal elliptic curve.  By \cite[Proposition 9.2]{GepnerMeier}, $U(1)$-equivariant elliptic cohomology factors over the stable category and thus we obtain an induced functor 
\[\Oscr_{\Escr}^{(-)}\colon (\text{finite }U(1)\text{-spectra})^{\op} \to \QCoh(\Escr, \Otop_{\Escr}).\]

To actually construct the transfer, embed $U(1)$ into the one-point compactification $S^{\C}$ of $\C$. This defines a $U(1)$-equivariant Thom collapse map from $S^{\C}$ to the Thom space of the normal bundle. Desuspending by $S^{\C}$ defines a stable $U(1)$-equivariant map $t\colon S^0 \to \Sigma^{-1}U(1)_+ = U(1)_+[-1]$, using the left-invariant framing of the tangent bundle of $U(1)$. (By the usual Pontryagin--Thom construction, postcomposing this stable map by $\Sigma^{-1}U(1)_+ \to \Sigma^{-1}\pt_+ = S^{-1}$ is the stable homotopy element corresponding to the framed manifold $U(1)$, i.e.\ $\eta\in \pi_1^{\mathrm{st}}S^0\cong \Z/2$, when disregarding the $U(1)$-action.) Applying equivariant elliptic cohomology to $t$, yields a map
\[e_*\Otop[1] = \Oscr_{\Escr}^{U(1)_+[-1]} \to  \Oscr_{\Escr}^{S^0}=\Otop_{\Escr}. \] 
Upon taking global sections, this gives the desired map $$\TMF[1] = \Gamma(\Otop)[1] \simeq \Gamma(e_*\Otop)[1]  \to \Gamma(\Otop_{\Escr}).$$ By  \cite[Theorem 10.1]{GepnerMeier}, this gives the desired equivalence \begin{equation}\label{eq:equivalence}
    (p^*, \mathrm{tr})\colon \TMF\oplus \TMF[1]\to \Gamma(\Otop_{\Escr}) =L^{(0)}.
\end{equation}

\begin{lemma}\label{lem:restriction}
Using the identifications above, the restriction 
\[e^*\colon L^{(0)} \simeq \Gamma(\Otop_{\Escr}) \simeq \TMF\oplus \TMF[1] \to L^{()} \simeq \Gamma(\Otop) = \TMF\]
is given by the matrix $(1\, \eta)$. 
\end{lemma}
\begin{proof}
    The first entry follows since $pe\colon \MM\to \Escr\to \MM$ is the identity and thus so is $e^*p^*$. 

    To show that $\TMF[1] \to \Gamma(\Otop_{\Escr}) \xrightarrow{e^*} \TMF$ is multiplication by $\eta$, observe that the map is induced by the $U(1)$-equivariant map $U(1)_+[1] \to S^1 \xrightarrow{t[1]} U(1)_+$ after taking $U(1)$-equivariant elliptic cohomology. Equivalently, it is induced by the composite 
\[S^1 \to U(1)_+[1] \to S^1  \xrightarrow{t[1]} U(1)_+ \to S^0\]
in non-equivariant $\TMF$.\footnote{There are several ways to see it. One is using just the formal theory of adjunctions from the observation that $U(1)_+ \simeq \ind_{\{1\}}^{U(1)}\res_{\{1\}}^{U(1)}S^0$.} 
The composite of the first two arrows is the identity and the composite of the last two is $\eta$, as observed above. 
\end{proof}

\begin{prop}\label{prop:explicitduality}
    Identifying $L^{(0)}$ with $\TMF\oplus \TMF[1]$ and $L_{(0)}$ with $\TMF\oplus \TMF[-1]$ via \eqref{eq:equivalence}, the equivalence $L_{(0)}\simeq L^{(0)}[-1]$ is given by the matrix $\begin{pmatrix}0 & \pm 1\\
    \pm 1 & \eta\end{pmatrix}$.
\end{prop}
\begin{proof}The upper left entry in the matrix must be $0$ as 
$$\Hom_{\Ho(\Mod_{\TMF})}(\TMF[-1], \TMF])=\pi_{-1}\TMF = 0.$$ 
As the determinant of the matrix must be a unit in $\pi_0\TMF = \Z[j]$, i.e.\ $\pm 1$, the antidiagonal entries must indeed be $\pm 1$. It remains to show that the lower right entry is $\eta$. 

    It is shown in the proof of \cite[Theorem 10.1]{GepnerMeier} that the composition 
    $$\TMF \xrightarrow{p^*} L^{(0)}=\Gamma(\Otop_{\Escr}) \xrightarrow{i} \Gamma(\Otop_{\Escr}(e))$$
    is an equivalence, while $\TMF[1] \xrightarrow{\mathrm{tr}} L^{(0)}=\Gamma(\Otop_{\Escr}) \xrightarrow{i} \Gamma(\Otop_{\Escr}(e))$ is zero. Thus, $i$ corresponds to the matrix $(\pm 1 \, 0)\colon \TMF\oplus \TMF[1] \to \TMF$. Consequently,
    the lower right entry in the matrix corresponds to the composite 
    \[\TMF = \TMF^{\vee} \xrightarrow{i^{\vee}} (L^{(0)})^{\vee} = L_{(0)} \simeq L^{(0)}[-1] \xrightarrow{i[-1]}\TMF[-1].\]

    Using the naturality of \cref{prop:duals} for $\Otop_{\Escr}\to \Otop_{\Escr}(e)$, we obtain a commutative square
        \[
\xymatrix{
\TMF \simeq L_{(-1)} \simeq \Gamma(\Otop_{\Escr}(e))^{\vee}\ar[r]^-{\simeq} \ar[d]^{i^{\vee}}& \Gamma(\Otop_{\Escr}(-e))[-1] \simeq L^{(-1)}[-1]\ar[d]\\
L_0 \simeq \Gamma(\Otop_{\Escr})^{\vee} \ar[r]^-{\simeq} & \Gamma(\Otop_{\Escr})[-1] \simeq L^0[-1]\ar[d]^i 
\\
&\Gamma(\Otop_{\Escr}(e))[-1] \simeq L^{(-1)}[-1]\simeq \TMF[-1]
        }\]
Thus, it suffices to see that the composite 
\[\TMF\simeq \Gamma(\Otop_{\Escr}(-e))[-1] \to \Gamma(\Otop_{\Escr}(e))[-1] \simeq \TMF[-1].\]
is multiplication by $\eta$. For this, in turn, it suffices to show that the matrix $\begin{pmatrix}a\\b\end{pmatrix}$ describing
\[\iota\colon \TMF[1]\simeq \Gamma(\Otop_{\Escr}(-e)) \to \Gamma(\Otop_{\Escr})\simeq \TMF\oplus\TMF[1]\]
equals $\begin{pmatrix}\eta\\ ?\end{pmatrix}$. 

To establish this, apply $U(1)$-equivariant elliptic cohomology to the cofiber sequence $U(1)_+ \to S^0 \to S^{\C}$, where the composite is (as in every cofiber sequence) nullhomotopic. We obtain a (co)fiber sequence
\[e_*\Otop \leftarrow \Otop_{\Escr} \leftarrow \Otop_{\Escr}(-e)\simeq \TMF[1],\]
using \cite[Lemma 9.1]{GepnerMeier}, and on global sections a (co)fiber sequence
\begin{equation}\label{eq:splitcofiber}\TMF \xleftarrow{e^*} \Gamma(\Otop_{\Escr})\simeq \TMF \oplus \TMF[1] \xleftarrow{\iota^*} \Gamma(\Otop_{\Escr}(-e)).\end{equation}
In particular, the composite is zero. By \cref{lem:restriction}, $e^*$ is given by $(1\, \eta)$. Thus, 
\begin{equation}\label{eq:ab}a + \eta\cdot b  =0.\end{equation}

Moreover, $e^*$ is split by $p^*$ and thus \eqref{eq:splitcofiber} is split. This means that
\[\TMF\oplus \Gamma(\Otop_{\Escr}(-e)) \simeq \TMF \oplus \TMF[1] \xrightarrow{\begin{pmatrix}
    p^*& \iota
\end{pmatrix} = \begin{pmatrix}1 &a\\0& b\end{pmatrix}} \TMF\oplus \TMF[1]\]
is an equivalence. Hence, $b=\pm 1$. Thus, \eqref{eq:ab} implies $a= \eta$ as claimed. 
\end{proof}

\section{Invariants of 3- and 4-manifolds} \label{sec: invariants of manifolds}
In the preceding sections, we have seen how to associate $\TMF$-modules to bilinear forms. In this section, we will use this to define invariants of $3$- and $4$-dimensional manifolds. To every closed $3$-manifold we associate a $\TMF$-module and to every simply-connected closed $4$-manifold a class in $\pi_*\TMF$. We will make a proposal how these invariants fit into a TQFT-framework. 

\emph{\textbf{Convention: }All manifolds are assumed to be oriented in this section and all diffeomorphisms are orientation-preserving.}

Moreover, $\Mod_{\TMF}$ will always denote the \emph{homotopy category} of $\TMF$-modules and $\gMod_{\TMF}$ will always denote the \emph{graded homotopy category} of $\TMF$-modules, i.e.\ $\Hom_{\gMod_{\TMF}}(M,N) = \bigoplus_{n\in \Z} [M[n], N]_{\TMF}$. 

\subsection{Invariants of $3$-manifolds}
Recall first that every connected closed $3$-manifold is the boundary of a simply-connected $4$-manifold. 
\begin{construction}\label{constr:3manifold}
            Let $M$ be a connected closed $3$-manifold. Consider a simply-connected oriented $4$-manifold $W$ bounding $M$. Let $b =b(W)$ be the intersection form of $W$, and $b^+$ and $b^-$ the number of positive and, respectively, negative eigenvalues. Then we set $\Zscr(M, W) := L_b[3b^+-2b^-]$.\footnote{We can write $H_2(W; \R)$ as $B^+ \oplus B^- \oplus B^0$, where the intersection form is positive definite on $B^+$, negative definite on $B^-$ and zero on $B^0$. For better functoriality, we might abstain from choosing isomorphisms $B^+\cong \R^{b^+}$ and $B^- \cong \R^{b^-}$ and rather define $\Zscr(M, W) = L_B \otimes S^{3B^+} \otimes S^{-2B^-}$, where $S^{V}$ denotes the one-point compactification of a vector space $V$.} 
        \end{construction}

        \begin{prop}\label{prop:welldefined}
            The $\TMF$-module $\Zscr(M,W)$ in the preceding construction depends up to equivalence only on $M$. 
        \end{prop}
        \begin{proof}
            As explained in \cref{sec: data}, $H_1(M)$ together with the linking form 
            \[\tau H_1(M)\times \tau H_1(M) \to \Q/\Z\] on its torsion subgroup, determine the intersection form $b$ of $W$ up to a $\pm$-equivalence. This means that  if there is another bounding simply-connected $4$-manifold $W'$ with intersection form $b'$, then 
            \[b\oplus \langle 1\rangle^r \oplus \langle -1\rangle^s \cong b'\oplus \langle 1\rangle^{r'} \oplus \langle -1\rangle^{s'} \]
            for suitable $r,r', s, s'\geq 0$. By \cref{lem:KirbyMove}, 
            \begin{align*}\Zscr(M, W) &= L_b[3b^+-2b^-] \simeq L_{b\oplus \langle 1\rangle^r \oplus \langle -1\rangle^s}[3(b^+ +r) -2(b^-+s)] \\&\simeq L_{b'\oplus \langle 1\rangle^{r'} \oplus \langle -1\rangle^{s'}}[3((b')^+ +r') - 2((b')^- + s')] 
            \\&\simeq L_{b'}[3(b')^+-2(b') ^-] = \Zscr(M, W').\qedhere\end{align*}
        \end{proof}
        For this reason, we will often write $\Zscr(M)$ for $\Zscr(M, W)$. If $M$ is disconnected, we set $\Zscr(M)$ to be the tensor product of the $\TMF$-modules associated to its components. 

\begin{prop}\label{prop:Zduality}
    Let $M$ be a closed $3$-manifold and $\overline{M}$ be the same $3$-manifold with the opposite orientation. Then $\Zscr(\overline{M}) \simeq \Zscr(M)^{\vee}[-b_1(M)]$, where $()^{\vee}$ denotes the $\TMF$-linear dual. 
\end{prop}
\begin{proof}
    Let $W$ be a simply-connected $4$-manifold bounding $M$ with intersection form $b$. Let $b^+$, $b^0$ and $b^-$ the number of positive, zero and negative eigenvalues, respectively. From the relation of $b$ with $H_1(M)$ explained in \cref{sec: homology}, we see that $b^0 = b_1(M)$. The $3$-manifold $\overline{M}$ is bounded by $\overline{W}$ (i.e.\ $W$ with the opposite orientation), whose intersection  form is $-b$. 

    Using \cref{thm:Lbduals}, we compute 
    \begin{align*}
        \Zscr(\overline{M}) &= L_{-b}[3b^- - 2b^+] \\
        &\simeq (L_b)^{\vee}[-b^+-b^0-b^-][3b^- - 2b^+]\\
        &\simeq (L_b)^{\vee}[-(3b^+-2b^-)][-b^0]\\
        &\simeq \Zscr(M)^{\vee}[-b_1(M)]. \hfill\qedhere
    \end{align*}
\end{proof}

\subsection{Functoriality of the $3$-manifold invariants}
        We stress a potential pitfall when defining $\Zscr(M)$ as in the last section: $\Zscr(M)$ is well-defined up to equivalence, but we have (a priori) not defined it up to \emph{canonical} equivalence. Since ``canonical'' is not a mathematical term, we want to rephrase the problem in terms of defining $\Zscr$ as a functor. (When defining a potential TQFT, we phrase our problem in terms of cobordism categories later.) 

        \begin{question}\label{q:functorialwelldefinedness1}
            Does $\Zscr$ refine to a functor $\mathrm{Mfld}_3\to \Mod_{\TMF}$ to the homotopy category of $\TMF$-modules? Here, $\mathrm{Mfld}_3$ is the category of closed connected $3$-manifolds and (isotopy classes of) diffeomorphisms between them.
        \end{question}
        
        To understand the relation between this question and the construction in the last section better, we need to introduce the following categories: 
        \begin{itemize}
            \item Let $\mathrm{Kirby}$ be the following category: 
            Its objects are framed links, and its morphisms are freely generated by \emph{Kirby moves}. These moves are of two types: one move replaces a link component by a framed version of a band sum with another component, and the second type introduces or removes a $\pm 1$ framed unknot which is split from the rest of the link; see \cite{GS} for a detailed discussion. A framed link determines a $4$-manifold obtained by attaching $2$-handles to $D^4$ along the link, so we consider an object in Kirby as a 4-manifold with a particular presentation as a 2-handlebody with the given attaching framed link. Note that Kirby moves (generating morphisms in the category) define diffeomorphisms of boundary $3$-manifolds.
            \item Let $\Bil$ be the category whose objects are finitely generated free abelian groups $\Lambda$ with a symmetric bilinear form $\Lambda\otimes \Lambda\to \Z$. Its morphisms are homomorphisms of abelian groups that are compatible with the bilinear form. 
            \item Let $\Bil_{\pm}$ the category with the same objects as $\Bil$, but whose morphisms are freely generated by those of $\Bil$ and isomorphisms $b\to b\oplus \langle 1\rangle$ and $b\to b\oplus \langle -1\rangle$.\footnote{Thus morphisms are sequences of morphisms in $\Bil$ and these extra isomorphisms and their inverses. One could (and probably should) put some reasonable relations on these, but this will not be relevant for us.}  
            \item $\mathrm{TorsionForms}$ is the category whose objects are finitely generated abelian groups $A$ together with a symmetric  bilinear form on the torsion subgroup\[\tau(A)\tensor \tau(A) \to \Q/\Z.\]
            Morphisms are morphisms of abelian groups that are compatible with the pairing. 
        \end{itemize}
We have a commutative diagram 
\begin{equation}\label{eq:BilDiagram}
\xymatrix{
&\mathrm{Kirby}\ar[d]_{\text{Intersection Form}}\ar[r]^{\partial}  &\mathrm{Mfld}_3 \ar[d]^{\text{Linking Form}} \\
\Bil\ar[r]\ar[dr]^{L'}&\Bil_{\pm} \ar[r]\ar@{-->}[d]^{L'} & \mathrm{TorsionForms}\ar@{-->}[dl]^?\\
&\Mod_{\TMF}&
}
\end{equation}
The commutativity of the square is essentially explained in \cref{sec: data}. More precisely, we have a functor $\Bil \to \mathrm{TorsionForms}$, sending a finitely generated free abelian group $A$ with a symmetric bilinear pairing $A\tensor A \to \Z$ to the cokernel of the adjoint map $A \to A^*$ with the induced torsion pairing. Adding $\langle 1 \rangle$ or $\langle -1\rangle$ to $A$ does not change the cokernel and thus the functor $\Bil \to \mathrm{TorsionForms}$ factors over $\Bil_{\pm 1}$. 

The diagonal arrow $L'$ sends a bilinear form $b$ to its associated $\TMF$-module $L_b[3b^+ - 2b^-]$, where $b^+$ and $b^-$ are the number of positive and negative eigenvalues.  By \cref{lem:KirbyMove}, this factors over $\Bil_{\pm}$. 

The more precise version of \cref{q:functorialwelldefinedness1} becomes the following: 
\begin{question}\label{q:functorialwelldefinedness2}
    Does the functor $\Zscr\colon \mathrm{Kirby} \to \Mod_{\TMF}$ factor over $\mathrm{Mfld}_3$? 
\end{question}
Recall that all 3-manifolds are assumed to be oriented, and all diffeomorphisms are orientation-preserving in this section. 
By the Lickorish--Wallace theorem, every closed connected $3$-manifold can (up to diffeomorphism) be presented as the boundary of a $4$-ball with $2$-handles attached, so  $\partial$ is essentially surjective. Moreover, every diffeomorphism of $3$-manifolds arises from Kirby moves \cite[Theorem 5.3.6]{GS} and thus the functor is also full. Thus, to factor $\Zscr$ through $\mathrm{Mfld}_3$, we do not need extra data but just to establish a property: we need to show that any sequence of Kirby moves defining (up to isotopy) the identity diffeomorphism on the boundary, is mapped by $\Zscr$ (up to homotopy) to the identity morphism of the associated $\TMF$-module. 

It would also suffice for a positive resolution of \cref{q:functorialwelldefinedness2} to work on the more algebraic level, namely to show that $L'\colon \Bil_{\pm} \to \Mod_{\TMF}$ factors through $\Bil_{\pm}\to \mathrm{TorsionForms}$.

\subsection{$4$-manifold invariants and a potential TQFT}
Our goal in this subsection is to define invariants of (simply-connected) closed $4$-manifolds in $\pi_*\TMF$. We moreover give a partial definition of a $\TMF$-valued $(3+1)$-dimensional topological quantum field theory that combines these $4$-manifolds invariants and the $3$-manifold invariants defined above.  

Since \cref{q:functorialwelldefinedness1} and \cref{q:functorialwelldefinedness2} remain unresolved for now, we will at first base our partial definition of a TQFT on a bordism category whose objects are pairs of a $3$-manifold and a bounding $4$-manifold. More precisely, consider a category $\widetilde{\Cob}^{\mathrm{con}}$ whose objects are pairs $(M, V)$ where $M$ is a connected closed $3$-manifold and $V$ is a simply-connected $4$-manifold with $\partial V=M$. Given two objects $(M_0, V_0)$ and $(M_1, V_1)$, a morphism between them is a cobordism $W$ from $M_0$ to $M_1$ such that $\pi_1(W, M_0)$ is trivial and $V_1 \cong V_0 \cup_{M_0} W$. The diffeomorphism between $V_1$ and $V_0 \cup_{M_0} W$ is considered to be part of the data of the morphism.
An example of morphisms satisfying the relative $\pi_1$ condition is given by 4-manifold cobordisms admitting a relative handle decomposition with all handles of index $2$.

Given an object $(M,V)$ in $\widetilde{\Cob}^{\mathrm{con}}$, define $\Zscr(M) = \Zscr(M,V)$ as in \cref{constr:3manifold}, using the specific intersection form $b(V)$. We are in a position to define maps of TMF-modules associated with 4-manifold cobordisms in this context.

        \begin{construction}\label{constr:cobordism}
           Let $W$ be a morphism between objects  $(M_0, V_0)$ and $(M_1, V_1)$ of $\widetilde{\Cob}^{\mathrm{con}}$.
            We define $\Zscr(W)\colon \Zscr(M_0)[d] \to \Zscr(M_1)$ 
            to be a shift of the transfer \[ L_{b(V_0)} \to L_{b(V_1)}\] induced by the homomorphism $H_2(V_0) \to H_2(V_1)$, which  in turn is induced by the map $V_0\hookrightarrow V_0\cup_{M_0} W \cong V_1$. Here 
            \begin{equation} \label{eq:degree}d = 3(b^+(V_1)-b^+(V_0)) - 2(b^-(V_1)-b^-(V_0)).
            \end{equation}
        \end{construction}

        This construction gives rise to a functor 
        \[ \widetilde{\Zscr}\colon\widetilde{\Cob}^{\mathrm{con}} \to \Mod_{\TMF},\]
to the category of TMF-modules. This gives rise to invariants of closed $4$-manifolds as follows: 

\begin{example}[Invariant of closed 4-manifolds]\label{closed 4manifolds}  Let $X$ be a closed simply-connected $4$-manifold.  Removing two $4$-balls from $X$, one has a cobordism $W$ with $M_0\cong M_1\cong S^3$. In the notation of \cref{constr:cobordism}, take $V_0$ to be a $4$-ball; then $V_1$ is $X$ with a $4$-ball removed. 

        Since $M_0\cong M_1\cong S^3$ is bounded by $D^4$ with trivial intersection form, $\Zscr(M_0)\simeq \Zscr(M_1)\simeq \TMF$. These identifications are well-defined up to a unit in $\pi_0\TMF\cong \Z[j]$, i.e.\ up to a sign. 
        Thus, the map associated with the cobordism is isomorphic to $\Zscr(W)\colon \TMF[d]\to\TMF$. Here the degree shift is given by equation \eqref{eq:degree}, $d=3b^+(X) - 2b^-(X)$. This map gives an element of $\pi_d\TMF$ (well-defined up to a sign), which is an invariant of $X$ mentioned in the introduction. Equivalently, one can construct the invariant as $\Zscr(X)$ when $X$ is viewed as a cobordism from $\varnothing \to \varnothing$. Moreover, the construction implies that it only depends on the intersection form $b(X)$ and equals $\mathfrak{d}_{-b(X)}$ (see \cref{constr:dhedafunction}). 

        If $X$ has only $2$-handles, the functor $\widetilde{\Zscr}$ constructed above can be used to show that $\Zscr(W)$ factors as a composition of maps corresponding to individual $2$-handles. 
        \end{example}

        \begin{remark}
            Let $X$ and $Y$ be simply-connected closed $4$-manifolds. The intersection form of the connected sum $X\# Y$ is the direct sum of the intersection forms of $X$ and $Y$. Thus, \cref{lem:RestrictionTransferDirectSum} implies that $\Zscr(X\# Y) = \Zscr(X)\cdot \Zscr(Y)\in \pi_*\TMF$ (up to sign). 
        \end{remark}

        Let $\Cob^{\mathrm{con}}$ be the category whose objects are connected closed $3$-manifolds and whose morphisms are cobordisms with vanishing relative fundamental groups for the inclusion of incoming and outgoing boundary, respectively. This comes with an obvious functor $\widetilde{\Cob}^{\mathrm{con}}\to \Cob^{\mathrm{con}}$, forgetting the bounding four-manifold. 

        \begin{question}\label{q:Zfactoring}
            Does $\widetilde{\Zscr}$ factor through a functor $\Zscr\colon \Cob^{\mathrm{con}} \to \Mod_{\TMF}$?
        \end{question}

        Under the hood, our $\widetilde{\Zscr}$ can be written as the left vertical composition in the following diagram
        \[
\xymatrix{\widetilde{\Cob}^{\mathrm{con}} \ar[r]\ar[d]_I & \Cob^{\mathrm{con}} \ar[d]^? \\ 
\Bil  \ar[r]^-? \ar[d]_{L'} &\text{Algebraic Cobordism Category?} \ar@{-->}[dl]^?  \\ 
\Mod_{\TMF},
}
        \]
        where $I(M, V)$ is the intersection form of $V$ and $L'(b) = L_b[3b^+-2b^-]$ as before. Ideally, one would answer \cref{q:Zfactoring} by writing down an ``algebraic cobordism category'', playing the role of torsion forms in \eqref{eq:BilDiagram} and receiving functors from $\Bil$ and $\Cob^{\mathrm{con}}$. The main point would be to show that $L'$ factors through this algebraic cobordism category. In \cref{sec: cobordism}, we collect some ingredients, which might allow us to define such an algebraic cobordism category. But we do not pursue this here any further and also not the relationship to the algebraic cobordism categories by Ranicki \cite{Ranicki} and \cite{9authors}.

Construction \ref{constr:cobordism} associates maps of TMF-modules with $4$-dimensional cobordisms given by $2$-handles. To define an actual topological quantum field theory $\Zscr\colon \Cob \to \Mod_{\TMF}$, we need (beside a positive answer to \cref{q:Zfactoring}) to extend this to $1$-handles and $3$-handles. We propose a partial construction. 
 
Let us first consider a $1$-handle attached to $S^3\times I$, that is a cobordism from $S^3$ to $S^2\times S^1$.
Using the bounding $4$-manifolds $D^4$, $S^2\times D^2$ with intersections forms $(\; ),  (0)$ respectively, the associated modules are seen to be $\TMF$ and  $(\TMF \oplus \TMF[1])^{\vee}\simeq \TMF \oplus \TMF[-1]$; see \cref{ex:S2xS2}.  We propose that the necessary map $\TMF[d]\to  \TMF \oplus \TMF[-1]$ corresponding to the $1$-handle cobordism has $d=-1$ and is the inclusion as the second summand. Let us generalize this construction to arbitrary $1$-handles:

\begin{construction}\label{otherhandles}
    Let $M$ be a connected closed $3$-manifold, and consider the $4$-dimensional cobordism $W=M\times I\cup 1$-handle. The two boundary components of $W$ are $M_0=M$ and $M_1\cong M\# (S^2\times S^1)$.  It follows from \cref{lem: direct sum} that 
\[\Zscr(M_1)\simeq \Zscr(M)\tensor_{\TMF} (\TMF \oplus \TMF[-1])\simeq \Zscr(M) \oplus \Zscr(M)[-1].
\]
As above, we define the map $\Zscr(W)\colon \Zscr(M)[-1]\to \Zscr(M_1)$ to be the inclusion 
as the second summand. Using calculations as in \cref{ex:S2xS2}, one expects that the map corresponding to first attaching a $1$-handle to $M\times I$ and then canceling it with a $2$-handle attached to $(M\# (S^2\times S^1))\times I$ is the identity on $\Zscr(M)$. This is a particular instance of functoriality of our theory.
\end{construction}

We conclude by stating a conjecture, motivated by physics considerations and by the functoriality results in restricted settings established in this section.

\begin{conj} \label{conj: TQFT}
    The invariants $\Zscr$ defined or conjectured above give rise to a $(3+1)$-dimensional TQFT, valued in modules over the ring spectrum $\TMF$.
\end{conj}

Verification of TQFT properties is work in progress, and the extent to which all of Atiyah--Segal axioms hold for this theory remains to be established. Since \cref{prop:Zduality} holds only up to a degree-shift, it would be necessary to also incorporate a degree-shift into the Atiyah--Segal axioms.  

Let us discuss one duality statement which would follow from the TQFT axioms and which we will use later.

 \begin{question}\label{quest:cobordismupsidedown}
            Does turning a cobordism upside down give rise to a dual map of TMF-modules?

            More precisely, consider a 4-dimensional cobordism $W$, consisting of 2-handles, from $M_0$ to $M_1$, and let $W^*$ denote the ``up-side down'' cobordism from $\overline{M}_1$ to $\overline{M}_0$. Choose simply-connected bounding $4$-manifolds $V_0$ and $V_1$ for $M_0$ and $M_1$, respectively with $V_1\cong V_0\cup_{M_0}W$. Consider $V_0' := \overline{V}_1\cup_{\overline{M}_1}W^*$. Let 
            \[\overline{d} = 3(b^+(V_0')-b^+(\overline{V_1})) -2(b^-(V_0')-b^-(\overline{V_1})).\]
            From the construction above, we obtain a map $\Zscr(\overline{M}_1)[\overline{d}] \to \Zscr(\overline{M}_0)$. By \cref{prop:Zduality}, we can identify this with a map $\Zscr(M_1)^{\vee}[-b_1(M_1)+\overline{d}] \to \Zscr(M_0)^{\vee}[-b_1(M_0)]$.  The following calculation of the degrees shows that the domain and the range of this map coincide with those of the dual of the map $\Zscr(M_0)[d] \to \Zscr(M_1)$, up to an overall degree shift, with $d$ as in \cref{constr:cobordism}.

            For this calculation, first observe that $V'_0=\overline{V}_0\cup_{\overline{M}_0}\overline{W}\cup_{\overline{M}_1} {W^*}$. The union $\overline{W}\cup_{\overline{M}_1} {W^*}$ is a ``relative double'' whose Kirby diagram is obtained by adding a zero-framed meridian to each framed curve representing a 2-handle of $\overline{W}$, cf.\ \cite[Example 5.5.4]{GS}. 
            Using these meridians, in a Kirby diagram for $V'_0$ the attaching curves of the $2$-handles corresponding to $\overline V_0$ can be slid off those for $\overline W$, and moreover $\overline{W}\cup_{\overline{M}_1} {W^*}$ is seen to be diffeomorphic to a connected sum of copies of $S^2\times S^2$ and $S^2\widetilde{\times} S^2$. 
            It follows that $b^{\pm}(V'_0)=b^{\pm}(\overline{V}_0)+{\rm rk}(H_2W)$. Then $$\overline{d}= -d+b^-(V_0)+b^+(V_0)-b^-(V_1)-b^+(V_1)+{\rm rk}(H_2W).$$
            Next observe that $b_1(M_0)={\rm rk}(H_2 V_0)-b^+(V_0)-b^-(V_0)$, with the analogous equality for $M_1$. Combining these observations, the overall degree shift of the map $\Zscr(M_1)^{\vee}[-b_1(M_1)+\overline{d}] \to \Zscr(M_0)^{\vee}[-b_1(M_0)]$ by $b_1(M_0)$ is seen to be equal to $\Zscr(M_1)^{\vee}[-d] \to \Zscr(M_0)^{\vee}$.
            \end{question}

        \section{Examples} \label{sec: computations of cobordisms and transfers}
        In this section we give some concrete calculations of our invariants of 3- and 4-manifolds.
        \subsection{Examples of modules associated to $3$-manifolds}
        \subsubsection{} The module associated with the sphere $S^3$, $\Zscr(S^3)$, is isomorphic to $\TMF$. This may be seen for example by representing $S^3$ as the boundary of the $4$-ball $D^4$, so the bilinear form   is of rank zero. From the TQFT perspective, the $4$-ball induces an isomorphism from the ground ring $\TMF$ associated with the empty $3$-manifold to $\Zscr(S^3)$. More generally, since the invariant of a 3-manifold $M$ depends only on the rank of $H_1(M)$ and the torsion linking form, $\Zscr(M)\simeq \TMF$ for any integral homology 3-sphere $M$. 
        \subsubsection{} Consider $S^2\times S^1$
        as the boundary of $D^4$ with a single $0$-framed 2-handle attached to the unknot. The intersection form used to define the $\TMF$-module is the bilinear form $(0)$ of rank $1$. According to \eqref{ex:Oke},
        \[ \Zscr(S^2\times S^1) = L_{(0)} \simeq \TMF \oplus \TMF[-1]. \]
        Again, the same result holds for any $3$-manifold which is an integer homology $S^2\times S^1$. 
        \subsubsection{} More generally, let $\Sigma_g$ denote the closed orientable surface of genus $g$, and consider $\Sigma_g\times S^1$ as the boundary of the 4-manifold $W$ with the Kirby diagram shown in Figure \ref{fig:surface_times_S1}. 
        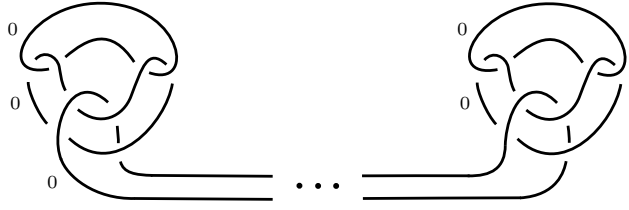
\begin{figure}[ht]
\includegraphics[height=3cm]{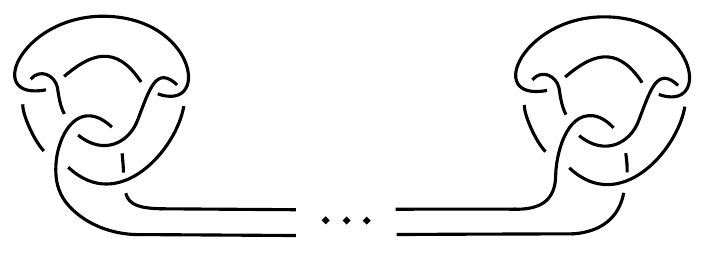}
{\scriptsize
\put(-240,68){$0$}
\put(-239,40){$0$}
\put(-70,68){$0$}
\put(-69,40){$0$}
\put(-225,10){$0$}
}
\caption{A 4-manifold $W$ with $\partial W=\Sigma_g\times S^1$.}
\label{fig:surface_times_S1}
\end{figure}
The second homology $H_2(W)$ is the free abelian group of rank $2g+1$, and the intersection pairing on $H_2(W)$ is identically zero since the linking numbers of the link in Figure \ref{fig:surface_times_S1} are all trivial.
By Lemma \ref{lem: direct sum}, \[ \Zscr(\Sigma_g\times S^1)  \simeq (\TMF \oplus \TMF[-1])^{\otimes (2g+1)}.\]
        \subsubsection{} Consider a lens space $L(n,1)$, as the boundary of $D^4$ with an $n$-framed $2$-handle attached to the unknot. The module $\Zscr(L(n,1))$ is defined using the bilinear form $(n)$; it is discussed in \cref{ex:Oke}.  In particular, for $n\geq 1$, after inverting 6, it is equivalent to a sum of $n$ shifted copies of $\TMF$, but the torsion is quite subtle, see \cite{BauerJacobi} for more detailed information. 
        
         \subsection{Torsion classes associated to $4$-manifolds}\begin{example}\label{ex:S2xS2}
            The closed $4$-manifold $S^2\times S^2$ has intersection form $\begin{pmatrix}0&1\\1&0\end{pmatrix}$. By \cref{closed 4manifolds}, we obtain an associated map $\Zscr(S^2\times S^2)\colon \TMF[1] \to \TMF$. We want to show that this is multiplication by $\eta$, i.e.\ $\Zscr(S^2\times S^2) = \eta \in \pi_1\TMF$. (Note that there is no potential sign ambiguity here since $\eta = -\eta$.)  

            We can decompose the cobordism $S^2\times S^2$ from $\varnothing$ to $\varnothing$ into two cobordisms: $S^2\times D^2\colon\varnothing \to S^2 \times S^1$ and $S^2\times D^2\colon S^2\times S^1 \to S^2\times S^2$. Since $S^2 \times D^2$ has intersection from $(0)$, \cref{constr:cobordism} allows to write $\Zscr(S^2\times S^2)$ as the composition
            \[\widetilde{\Zscr}(\varnothing,\varnothing)[1] \to \widetilde{\Zscr}(S^2\times S^1, S^2\times D^2)[1] \to \widetilde{\Zscr}(\varnothing, S^2\times S^2) .\]
            Using the definitions and our earlier computations (in particular \cref{ex:Oke} and \cref{prop:welldefined}), we can identify this with 
\begin{equation}\label{eq:S2S2}L_{()}[1] \simeq \TMF[1] \to L_{(0)}[1]\simeq \TMF\oplus \TMF[1] \to L_{()}\simeq \TMF.\end{equation}
            The first map is dual to the restriction $L^{(0)}\simeq \TMF\oplus \TMF[1] \to L^{()}\simeq \TMF$, which is by \cref{lem:restriction} given by $(1\; \eta)$. Thus, the first map itself is given by $\begin{pmatrix}\eta \\ 1 \end{pmatrix}$. 

            To compute the second map, we \emph{assume a positive answer to \cref{quest:cobordismupsidedown}}. Thus, the second map is dual to the first map, using the duality equivalence $L_{(0)}^\vee \simeq L_{(0)}[-1]$. This was made explicit in \cref{prop:explicitduality}, and the self-equivalence of $\TMF\oplus \TMF[1]$ induced by it is given by the matrix $\begin{pmatrix}\pm 1 &\eta \\
    0 & \pm1 \end{pmatrix}$. Thus, the whole composite is given as 
    \[\begin{pmatrix}1&\eta\end{pmatrix}\cdot \begin{pmatrix}\pm 1 &\eta \\
    0 & \pm1 \end{pmatrix}\cdot \begin{pmatrix}\eta\\ 1\end{pmatrix} = \eta. \]   
    Chasing through the proof of \cref{prop:explicitduality}, $\eta$ appears here as coming from the element in $\pi_1^{\mathrm{st}}S^0$ corresponding to the framed manifold $U(1)$. Physically, this corresponds to the $\sigma$-model with target $U(1)$. 
            
        \end{example}

        \begin{remark}
            We can generalize the preceding example to the double $DW$ of any compact simply-connected $4$-manifold $W$ with boundary and trivial intersection form. If $H_2(W)$ has rank $d$, the map $\Zscr(DW)\colon \TMF[d] \to \TMF$ turns out to be the $d$-th tensor power (over $\TMF$) of \eqref{eq:S2S2}; cf.\ \cref{lem: direct sum} and \cref{lem:RestrictionTransferDirectSum}. (Here, we still assume a positive answer to \cref{quest:cobordismupsidedown}.) Thus, the $\TMF$-element $\Zscr(DW)\in \pi_d\TMF$ associated to $DW$ is $\eta^d$. 
       \end{remark}

        \begin{example}\label{ex:CP2}
            Consider $\CP^2$ with intersection form $(1)$. We claim that $\Zscr(\CP^2) = \pm \nu\in \pi_3\TMF$. 
            
            By construction, $\Zscr(\CP^2)$ arises as the transfer $L_{()} \to L_{(1)}$ or, dually, as the restriction $\TMF[3] \simeq L^{(-1)} \to L^{()}=\TMF$. By construction, $L^{(-1)} \simeq \Gamma(\Otop_{\Escr}(-e)[2])$. 

            We will compute this restriction by reinterpreting it in equivariant elliptic cohomology. Recall from \cref{sec:exampleL0} or \cite{GepnerMeier} that there is a functor 
           \[\Oscr_{\Escr}^{(-)}\colon (\text{finite pointed }U(1)\text{-CW complexes})^{\op} \to \QCoh(\Escr, \Otop_{\Escr})\]
           (called $\mathcal{E}ll_{U(1)}$ in \cite{GepnerMeier}). Postcomposing with $\Gamma$ defines a functor 
           \[\TMF_{U(1)}^{(-)}\colon (\text{finite pointed }U(1)\text{-CW complexes})^{\op} \to \Mod_{\TMF}.\] 
           Since $\Oscr_{\Escr}^{S^{\C}}\simeq \Otop_{\Escr}(-e)$ by \cite[Lemma 9.1]{GepnerMeier}, we have $\TMF_{U(1)}^{S^{\C}} = L^{(-1)}[-2]$. The restriction map $L^{(-1)} \to L^{()}$ can be identified with a $[2]$-shift of the restriction map $\TMF_{U(1)}^{S^{\C}} \to \TMF^{S^2} \simeq \TMF[-2]$. 

           Note that we have a natural transformation $\TMF_{U(1)}^{(-)} \to \TMF_{U(1)}^{(-) \wedge EU(1)_+} \simeq \TMF^{() \wedge_{U(1)}EU(1)_+}$ to the Borel theory. Applying this to the map $S^0 \to S^{\C}$ and restriction, we obtain a commutative diagram
            \[
\xymatrix{
\TMF_{U(1)}^{S^0}\simeq \TMF\oplus \TMF[1] \ar[r] & \TMF^{\CP^{\infty}_+}  \\
\TMF_{U(1)}^{S^{\C}} \simeq \TMF[1] \ar[d]^r\ar[r]^t\ar[u]^f & \TMF^{\CP^{\infty}} \ar[u]\ar[d]^r \\
\TMF[-2] \ar[r]^{\simeq} & \TMF^{\CP^1}
}
            \]
We claim that one obtains $t$ by applying $\TMF^{(-)}$ to the reduced transfer $\CP^{\infty} \to \CP^{\infty}_+ \to S^{-1}$ (see e.g.\ \cite[Appendix A]{BauerJacobi} for the definition and properties of the reduced transfer). This follows from the following observations about the diagram above:
\begin{itemize}
    \item The map $\TMF^{\CP^{\infty}}\to \TMF^{\CP^{\infty}_+}$ has a splitting induced by the inclusion $\CP^{\infty} \to \CP^{\infty}_+$. In $\TMF$-cohomology, this corresponds to projecting away the first summand in $\TMF^{\CP^{\infty}_+}\simeq \TMF \oplus \TMF^{\CP^{\infty}}$. In particular, the map $\TMF[1]\to \TMF^{\CP^{\infty}}$ is the composite of 
    \[\TMF[1] \to \TMF_{U(1)}^{S^0} \to \TMF^{\CP^{\infty}_+}\to \TMF^{\CP^{\infty}}.\] 
    \item The map $f\colon \TMF[1]\to \TMF_{U(1)}^{S^0}\simeq \TMF\oplus \TMF[1]$ is the sum of the transfer (i.e.\ the inclusion of the second copy) plus the map $\eta\colon \TMF[1]\to \TMF$, as established in the proof of \cref{prop:explicitduality}.
    \item Thus, $\TMF[1]\to \TMF^{S^0}_{U(1)}\to \TMF^{\CP^{\infty}_+}$ is the sum of $\TMF[1] \xrightarrow{\eta} \TMF \to \TMF^{\CP^{\infty}_+}$ and the transfer map $\TMF[1]\to \TMF^{\CP^{\infty}_+}$ induced by the transfer map $\CP^{\infty}_+\to S^{-1}$ by applying $\TMF^{(-)}$. The former map vanishes upon projection to $\TMF^{\CP^{\infty}}$ and thus $t$ is as claimed.
\end{itemize}
Thus, the map $r$ is $\TMF^{(-)}$ applied to the composite
            \[S^2=\CP^1 \to \CP^{\infty} \to \CP^{\infty}_+ \to S^{-1}.\]
            This composite is known to be $\pm \nu$. One can e.g.\ argue as follows: As recalled e.g.\ in \cite[Appendix A]{BauerJacobi}, the cofiber of $\CP^{\infty}_+ \to S^{-1}$ is the stunted projective space $\CP^{\infty}_{-1}$, i.e.\ the Thom spectrum of minus the tautological line bundle on $\CP^{\infty}$. The operations $\mathrm{Sq}^4\colon H^{-2}(\CP^{\infty}_{-1}; \mathbb{F}_2) \to H^{2}(\CP^{\infty}_{-1}; \mathbb{F}_2)$ and $\mathcal{P}^1\colon H^{-2}(\CP^{\infty}_{-1}; \mathbb{F}_3) \to H^{2}(\CP^{\infty}_{-1}; \mathbb{F}_3)$ are non-trivial.\footnote{This can e.g.\ be deduced from the known Steenrod operations on the cohomology of $\CP^{\infty}$ together with the James-type periodicity results in \cite{DavisStunted}; while they are only stated for $\CP_m^n$ with $m\geq 0$, the same proof applies also to negative $m$.} Thus, the analogous fact is true for the $2$-cell complex arising as the cofiber of our map $S^2\to S^{-1}$. This means that the corresponding element in $\pi_3^{\mathrm{st}}S^0$ is detected in the $1$-line of the Adams spectral sequence both $2$-locally and $3$-locally. The only such element is $\pm \nu$. 
        \end{example}

        \begin{example} The intersection form of $\overline{\CP}^2$ is $(-1)$. Thus, $\Zscr(\overline{{\mathbb C}{\mathbb P}}^2) \in \pi_{-1}\TMF$. Since the whole group is zero, $\Zscr(\overline{{\mathbb C}{\mathbb P}}^2)$ must vanish.\end{example}

\section{Relation to theta functions}\label{sec:thetaFunctions}
Let $Q$ be a positive definite and unimodular quadratic form $\Z^d\to \Z$ and let $b$ be the corresponding even bilinear form. In \cref{constr:dhedafunction}, we have defined  an element $\mathfrak{d}_b\in \pi_{-2d}\TMF$. The goal of this section is to explain its conjectural relation to theta-functions (\cref{conj:dbgoestotheta}) and describe a possible approach to resolving this conjecture. Note that the element $\mathfrak{d}_b$ agrees with the invariant we associate with a closed simply-connected $4$-manifold with intersection form $-b$ in \cref{closed 4manifolds}. 

We start by describing a physical motivation for \cref{conj:dbgoestotheta}: Assuming the Stolz--Teichner program (see \cite{stolz2011supersymmetric} or the next section for details), we expect that $\mathfrak{d}_b\in \pi_{-2d}\TMF$ corresponds to the lattice conformal field theory defined by $b$, which can mathematically be modeled by a lattice vertex operator algebra (see e.g.\ Section 7.3 of the contribution of Mason and Tutte in \cite{Window}). Thus, $\mathfrak{d}_b$ should be a refinement of the partition function of this theory, which is, up to powers of the discriminant, the theta function 
\[\Theta_Q=\Theta_b\colon \H \to \C, \qquad \tau \mapsto  \sum_{v\in \Z^d}q^{Q(v)},\]
with $q = e^{2\pi i\tau}$; see loc.cit., Formula (75). (This is a modular form of weight $d/2$ \cite[Section VII.6.5]{SerreArithmetic}.) More precisely, we conjecture:\footnote{When $b$ is positive definite, the lattice CFT only has left-moving degrees of freedom, and has a trivial $(0,1)$ supersymmetry with the supercharge $Q$ acting by zero on any states. The partition function of this theory is $\eta^{-d}\cdot \Theta_Q$. The Witten genus is a normalization of the partition function by an additional factor of $\eta^{-2d}$ \cite{witten1987elliptic}, leading to $\eta^{-3d}\cdot \Theta_Q=\Delta^{-\frac{1}{8}d}\cdot \Theta_Q.$ When $b$ is not positive definite, there will be right-moving degrees of freedoms, including free fermions that cause the partition function to vanish. }
\begin{conj}\label{conj:dbgoestotheta}
The image of $\mathfrak{d}_b$ under the edge homomorphism $\pi_{-2d}\TMF \to M_{-d}$ is $\pm \Delta^{-\frac18d}\Theta_Q$, where $M_{-d}$ denotes integral weakly holomorphic modular forms of weight $-d$ .\footnote{By \cite[Section V.2.1]{SerreArithmetic}, the rank of every positive definite unimodular even bilinear form is divisible by $8$ so that $\Delta^{-\frac18d}$ is well-defined.} 
\end{conj}

We present an approach to this conjecture by employing the multi-variable theta-function $\theta_b\colon \H\times\C^d\to \C$ from \eqref{eq:multivariabletheta}. By definition, $\Theta_b = \theta_b(-, 0)$. As noted in \cref{sec:ComplexPicture}, the function $\theta_b$ defines a holomorphic section of the Looijenga line bundle $\Lscr_b^{\C}\tensor \omega^{\otimes d/2}_{\C}$. While the line bundle $\Lscr_b^{\C}$ was constructed as a holomorphic line bundle, we have seen in \cref{sec: bilinear to TMF} that it arises as the complexification of the algebraic line bundle $\Lscr_b^{\Z}$ on $\Escr^{\times_{\Mscr}d}$. 

\begin{conj}\label{conj:thetaalgebraic}
    The holomorphic section $\theta_b$ of $\Lscr_b^{\C}\tensor \omega^{\tensor d/2}_{\C}$ refines to an algebraic section of $\Lscr_b^{\Z}\tensor \omega^{\tensor d/2}$. 
\end{conj}

The goal of the rest of this section is to show how \cref{conj:dbgoestotheta} is implied by \cref{conj:thetaalgebraic}, which we will assume from now on. We denote the section of $\Lscr^{\Z}_b\tensor \omega^{\tensor d/2}$ refining $\theta_b$ by $\theta_b$ as well. 

    Recall that we define the element $\mathfrak{d}_b\in \pi_{-2d}\TMF$ as the restriction of 
    \[1\in \pi_{-2d}\TMF[-2d] \simeq \pi_{-2d}\Gamma(\Lscr_b)\]
    to $\Mscr$, i.e.\ the image of this $1$ along $\pi_{-2d}\Gamma(\Lscr_b) \to \pi_{-2d}\Gamma(\Otop_{\Mscr}) = \pi_{-2d}\TMF$.  The edge homomorphism defines a diagram
    \begin{equation}\label{Lscrbedge}
    \xymatrix{
        \pi_{-2d}\Gamma(\Lscr_b) \ar[r]\ar[d] & H^0(\Escr^{\times_{\MM} d}; \Lscr_b^{\Z}\tensor \omega^{\tensor (-d)})\ar[d] \\
        \pi_{-2d}\TMF \ar[r] & H^0(\Mscr; \omega^{\tensor (-d)}).
    }\end{equation}
Here, we leave, as often, the pullback of $\omega^{\tensor (-d)}$ to $\Escr^{\times_\Mscr d}$ implicit. We will show the following result later in this section.

\begin{prop}\label{prop:edgehomom}
    The edge homomorphism $\pi_{-2d}\Gamma(\Lscr_b) \to H^0(\Escr^{\times_{\MM} d}; \Lscr_b^{\Z}\tensor \omega^{\tensor (-d)})$ is an isomorphism.
\end{prop}

\begin{proof}[Proof of \cref{conj:dbgoestotheta} assuming \cref{conj:thetaalgebraic} and \cref{prop:edgehomom}: ]
    Let $d_b$ be the image of a generator of $\pi_{-2d}\Gamma(\Lscr_b)$ as a module over the ring $\Z[j]$ of integral weakly holomorphic modular forms under the edge homomorphism to $ H^0(\Escr^{\times_{\MM} d}; \Lscr_b^{\Z}\tensor \omega^{\tensor (-d)})$. By \cref{prop:edgehomom}, $d_b$ is a generator of $H^0(\Escr^{\times_{\MM} d}; \Lscr_b^{\Z}\tensor \omega^{\tensor (-d)})$ as a $\Z[j]$-module. By \cref{conj:thetaalgebraic}, $\Delta^{-\frac{d}8}\theta_b \in H^0(\Escr^{\times_{\MM} d}; \Lscr_b^{\Z}\tensor \omega^{\tensor (-d)})$ and thus there is a $g\in \Z[j]$ such that $\Delta^{-\frac{d}8}\theta_b = gd_b$.

    We claim that $\Delta^{\frac{d}8}d_b=\theta_b/g$ is a Jacobi form of index $b$ and weight $\frac{d}2$ (see \cref{sec:ComplexPicture} for a short recollection of Jacobi forms). It is indeed clear that it defines a holomorphic section of $\Lscr_b^{\C}\tensor \omega^{\tensor \frac{d}2}$, using the identification of $\Lscr_b^{\C}$ with the complexification of $\Lscr_b^{\Z}$ from \cref{UniquenessOfComplexLooijenga}. It moreover has a Fourier expansion of the form 
    \begin{equation} \sum_{n\geq 0}q^n\sum_{r\in L^{\vee}: b(r,r)\leq 2n}c(n,r)e^{2\pi ib(z,r)}.\end{equation}
    since the same is true for $\theta_b$ and $g = a_{-n}q^{-n} + O(q^{-n+1}).$

    By \cite[Theorem 5.1]{BoylanSkoruppa}, there exists thus a holomorphic modular form $h$ such that $\theta_b/g = h \theta_b$. By weight considerations, $h$ must be of weight $0$, i.e.\ $h$ is a constant in $\C$. From the resulting equation $gh =1$, we conclude $g \in \Z$. 

    The restriction $\Theta_Q = \sum_{n\geq 0} \#\{z\in \Z^d: Q(z)=n\} q^n$ of $\theta_b$ along the zero section $\Mscr \to \Escr^{\times_{\Mscr}d}$ is not divisible by any integer not equal to $\pm 1$ as a section of $\omega^{\tensor \frac{d}2}$. Thus, the same is true for $\theta_b$ as a section of $\Lscr_b\tensor \omega^{\tensor \frac{d}2}$. Hence, $g = \pm 1$ and $\theta_b = \pm  d_b$. 

    The commutativity of the diagram \eqref{Lscrbedge} thus implies the result. 
\end{proof}

    It remains to show \cref{prop:edgehomom}. For this, we need first to show an algebraic analogue of \cref{ex:unimodular}. 
    \begin{lemma}\label{lem:VanishingofHigherDirectImages}
        The pushforward of $\Lscr^{\Z}_b$ to $\Mscr$ is isomorphic to $\omega^{\tensor d}$ and all higher direct images vanish. Thus, \[H^0(\Escr^{\times_{\MM} d}; \Lscr_b^{\Z}\tensor \omega^{\tensor (-d)}) \cong \Z[j],\]
        the ring of (weakly holomorphic) modular forms of weight $0$.    
    \end{lemma}
    \begin{proof}
        By \cite[Section V.2.2, Theorem 4]{SerreArithmetic}, $b\oplus (-1) \cong (1)^d \oplus (-1)$. We have corresponding projections $\Z^{d+1} \to \Z^d$ onto the summands $(1)^d$ and $b$, respectively, and these induce projections 
        \[q_1, q_b\colon \Escr^{\times_{\MM} d+1}= \Escr\otimes\Z^{d+1} \to \Escr^{\times_{\MM} d} =\Escr\otimes\Z^{d}.\]
        We have a further projection $r\colon \Escr^{\times_{\MM} d+1}= \Escr\otimes\Z^{d+1} \to \Escr$ corresponding to the $(-1)$-summand in $\Z^{d+1}$. We further denote the projection of $\Escr^{\times_{\Mscr} n} \to \Mscr$ by $p_n$. 

        Using the additivity and functoriality properties of $\Lscr^{\Z}$, we have 
        \[ q_b^*\Lscr^{\Z}_b\tensor r^*\Lscr^{\Z}_{(-1)}\cong \Lscr^{\Z}_{b\oplus (-1)}\cong \Lscr^{\Z}_{(1)^d\oplus (-1)}\cong q_1^*\Lscr^{\Z}_{(1)^d}\tensor r^*\Lscr^{\Z}_{(-1)}.\]

         Recall that $\Lscr_{(1)}^{\Z}\cong \Oscr_{\Escr}(e)\tensor \omega$ and $\Lscr_{(-1)}^{\Z}$ is its dual, i.e.\ $\Oscr_{\Escr}(-e)\tensor \omega^{(-1)}$. Since 
        \[H^i(E; \Oscr_E(e)) = \begin{cases} R & \text{ for } i=0\\
        0&\text{ else}\end{cases}\] for every elliptic curve $E$ over a ring $R$, we have 
        \begin{equation}\label{eq:O(e)pushforward}R^i(p_1)_*\Oscr_{\Escr}(e) \cong \begin{cases} \Oscr_{\Mscr} & \text{ for } i=0\\
        0&\text{ else}\end{cases},\end{equation}
        where $R^i(p_1)_*$ denotes the $i$-th higher direct image, i.e.\ the $i$-th right derived functor of pushforward. 
        Grothendieck duality (cf.\ \cite[p.44-45]{DeligneRapoport})  implies that 
          \begin{align*}R^i(p_1)_*\Oscr_{\Escr}(-e) \tensor \omega &\cong  R^i(p_1)_*\mathcal{H}om(\Oscr_{\Escr}(e), \omega)\\ &\cong \mathcal{H}om_{\Mscr}(R^{1-i}(p_1)_*\Oscr_{\Escr}(e), \Oscr_{\Mscr})\\ &\cong \begin{cases} \Oscr_{\Mscr} & \text{ for } i=1\\
        0&\text{ else}\end{cases}.\end{align*}
        
        Using the projection formula and the commutation of higher direct images and flat base change, this implies        \begin{align*}R^i(q_1)_*\Lscr^{\Z}_{(1)^d \oplus (-1)}&\cong R^i(q_1)_*(q_1^*\Lscr^{\Z}_{(1)^d}\tensor r^*\Lscr^{\Z}_{(-1)})\\
        &\cong \Lscr^{\Z}_{(1)^d} \tensor R^i(q_1)_*r^*\Lscr^{\Z}_{(-1)}\\ &\cong \Lscr^{\Z}_{(1)^d} \tensor (p_d)^*R^i(p_1)_*\Lscr^{\Z}_{(-1)}
        \\
        &\cong \begin{cases}  \Lscr_{(1)^d}^{\Z}\tensor p_d^*\omega^{\tensor (-2)} &\text{ if } i=1 \\ 0 & \text{ else}\end{cases}.
        \end{align*}
        The same argument shows that the corresponding formula with $\Lscr_{b}^{\Z}\tensor p_d^*\omega^{\tensor (-2)}$ holds for $R^i(q_b)_*\Lscr^{\Z}_{b\oplus (-1)}$.

        Utilizing \eqref{eq:O(e)pushforward} and the composite functor Grothendieck spectral sequence , $(p_d)_*\Lscr_{(1)^d}^{\Z} \cong \omega^{\tensor d}$, while the higher direct images vanish. Using the collapsing composite functor spectral sequence 
        \[E_2^{ij} = R^i(p_d)_*R^j(q_1)_*\Lscr^{\Z}_{(1)^d\oplus (-1)} \Rightarrow R^{i+j}(p_dq_1)_*\Lscr^{\Z}_{(1)^d \oplus (-1)} \]
        and the projection formula again, 
        we calculate 
        \begin{equation}\label{eq:RiComposite}R^i(p_dq_1)_*\Lscr^{\Z}_{(1)^d \oplus (-1)} \cong \begin{cases} \omega^{\tensor (d-2)} &\text{ if } i=1 \\ 0 & \text{ else}\end{cases}.\end{equation}
        Using that the equally collapsing composite functor spectral sequence 
        \[E_2^{ij} = R^i(p_d)_*R^j(q_b)_*\Lscr^{\Z}_{b\oplus (-1)} \Rightarrow R^{i+j}(p_dq_b)_*\Lscr^{\Z}_{b\oplus (-1)} \]
        converges to the same target, we deduce that $(p_d)_*\Lscr^{\Z}_{b} \cong \omega^{\tensor d}$ and $R^i(p_d)_*\Lscr^{\Z}_{b}$ vanishes for $i>0$. 
        
        The projection formula implies that 
    \begin{align*}H^0(\Escr^{\times_{\Mscr d}}; \Lscr_b^{\Z}\tensor p_d^*\omega^{\tensor (-d)}) &\cong H^0(\Mscr; (p_d)_*(\Lscr_b^{\Z}\tensor p_d^*\omega^{\tensor (-d)})) \\&\cong H^0(\Mscr; \omega^{\tensor d} \tensor \omega^{\tensor (-d)}) \\&\cong H^0(\Mscr, \Oscr_{\Mscr}) \\&\cong \Z[j]. \qedhere\end{align*}
    \end{proof}

\begin{proof}[Proof of \cref{prop:edgehomom}: ]
    Let $p_d\colon \Escr^{\times_{\Mscr}d} \to \Mscr$ be the projection. The naturality of the edge homomorphism yields a commutative diagram 
    \[
\xymatrix{
\pi_{-2d}\Gamma((p_d)_*\Lscr_b) \ar[r]\ar[d]^{\cong} & H^0(\Mscr; \pi_{-2d}(p_d)_*\Lscr_b) \ar[d] \\
\pi_{-2d}\Gamma(\Lscr_b) \ar[r] & H^0(\Escr^{\times_{\Mscr}d}; \Lscr_b^{\Z}\tensor \omega^{\tensor -d}).
}
    \]
    To show that the right vertical arrow is an isomorphism it is enough to show that the natural map 
    \[\pi_{-2d}(p_d)_*\Lscr_b \to (p_d)_*(\Lscr^{\Z}_b \tensor \omega^{\tensor -d})\]
    is an isomorphism. 
    Consider the relative descent spectral sequence
    \[E_2^{ij} \cong R^j(p_d)_*\pi_i\Lscr_b \; \Rightarrow \; \pi_{i-j}(p_d)_*\Lscr_b. \]
    The odd homotopy groups of $\Lscr_b$ vanish, while for the even ones we have $\pi_{2i}\Lscr_b \cong \Lscr_b^{\Z}\tensor \omega^{\tensor i}$. Thus, the projection formula and \cref{lem:VanishingofHigherDirectImages} imply
    \[R^j(p_d)_*\pi_{2i}\Lscr_b \cong (R^j(p_d)_*\Lscr_b^{\Z})\tensor \omega^{\tensor i} \cong \begin{cases} (p_d)_*(\Lscr_b^{\Z}\tensor \omega^{\tensor i}) & \text{ if }j=0\\
    0 & \text{ if }j>0,\end{cases}\]
    which gives for $i=-d$ the required isomorphism. 

    Arguing as in \cref{ex:unimodular}, the proof of \cref{lem:KirbyMove} yields that $(p_d)_*\Lscr_b\simeq \Otop[-2d]$. Thus, the top horizontal arrow identifies with the edge homomorphism for $\Otop$, i.e.\ $\pi_0\TMF \to \Z[j]$, which is known to be an isomorphism.\footnote{The fact that the edge homomorphism $\pi_0\TMF \to \Z[j]$ is an isomorphism is well-known and essentially stated e.g.\ in \cite[Section 13.2]{TMFBook}, but without proof. The fact is equivalent to no class in the descent spectral sequence for $\TMF$ surviving in the $0$-th column above Row $0$. This descent spectral sequence is calculated in \cite[Section 8, Corollary E]{carrick2024descent}.} Thus, the top horizontal arrow is an isomorphism and it follows that the same is true for the bottom horizontal arrow, as claimed.   
\end{proof}

\section{TMF-modules and quantum field theories} \label{sec:TMFQFT} 

In this section, we will explain in more detail how TMF-modules arise from quantum field theory in three dimensions, how various results in previous sections can be understood from the physics point of view, and how TMF-modules can be a useful tool in the study of quantum field theories. We do not aim for mathematical rigor in this section, use physical language freely, and allow ourselves to speculate.  
To make sure that readers from a mathematical background can follow some of the discussion in this section, we first start by explaining how the physics setup is related to and generalizes the Stolz--Teichner program \cite{segal1988elliptic,stolz2004elliptic,stolz2011supersymmetric}. 

The Stolz--Teichner program can be understood in physics as a statement about the homotopy type of the space $\CT$ of 2d $(0,1)$ supersymmetric quantum field theories. Namely, 
$$\CT\simeq \bigcup_{d\in\Z} \TMF_d,$$
where $\TMF_d$ is the $(-d)$-th space of the $\Omega$-spectrum $\TMF$. To make sense of the statement, one needs to define a grading for $\CT$ so that it is identified with the grading in TMF. In physics terms, the grading is given by the ``gravitational anomaly" of the 2d theory. This was dealt with in \cite{stolz2004elliptic,stolz2011supersymmetric} with twisted theories, while another point of view, and in a sense more widely used in today's theoretical physics community, is to view the 2d theory as living on the boundary of a 3d theory, sometimes refered to as the ``anomaly theory'' which captures the anomaly via a procedure known as ``anomaly inflow.'' See Section 3 of \cite{GPPV} for a more detailed summary of the Stolz--Teichner program from the physics perspective, and \cite{FreedHopkins, FreedLectures} for more mathematical discussions related to anomaly theories.

What we will discuss in this section is in fact a generalization of this setup to the case where the three-dimensional ``anomaly theory'' is replaced with a more general quantum field theory.

Given a 3d quantum field theory $\FT$, one can consider the space of its 2d $(0,1)$ boundary conditions. In general, we will allow the boundary theory to have gravitational anomaly. So it is in fact a boundary condition for $\FT\otimes \FI_d$ (or alternatively, an interface between $\FT$ and $\FI_{-d}$), where $\FI_d$ denotes the invertible theory given by a gravitational Chern--Simons theory at level $d$. This is an ``almost trivial'' theory as there is only a single ground state on any closed 2-manifold, but nonetheless its partition function on a given three-manifold $Y$ will depend on the metric as well as the spin structure. One way to define this partition function is by fixing a spin 4-manifold $W$ with $\partial W=Y$, and then the partition function is\footnote{Re-expressing $p_1$ on a 4-manifold as a boundary term is in fact the original motivation for the original work of Chern and Simons \cite{chern1974characteristic}.}
\begin{equation*}
    \FI_d(Y)=e^{-2\pi i \cdot d\int_W \frac{p_1}{48}},
\end{equation*}
where $p_1$ is the (metric-dependent) local Pontryagin form. By 
\cref{principle} and \cref{prop:P4anomaly}, the group of deformation classes of invertible (unitary, spin) 3d-theories is $\Z \cong \Hom(\Omega_4^{\mathrm{Spin}}, \Z)$. The gravitational Chern--Simons theory at level $d$ corresponds to the homomorphism $M\mapsto -\frac{d\cdot p_1(M)}{48}$, which for $d=1$ is a generator of this group. 
 
We denote the space of such boundary conditions (or interfaces) as $\CB_d(\FT)$.\footnote{We will choose the orientation to be that of the outgoing boundary. For an interface between $\FT_1$ and $\FT_2$, the convention for orientation is such that it is an outgoing boundary for $\FT_1$ and incoming boundary for $\FT_2$.  }
When $\FT=\FI_0$ is the trivial theory, 
\begin{equation*}
    \CB_d(\FI_0)=\CT_d
\end{equation*}
is the same as the space of $(0,1)$ theories with gravitational anomaly $d$. 

For each $d$, the space $\CB_d(\FT)$ is a module over $\CT_0$ by the stacking of 2d QFTs on top of a boundary condition, while a theory in $\CT_{d'}$ gives a map from $\CB_d(\FT)$ to $\CB_{d+d'}(\FT)$. Therefore, if one combines all $d$,
\begin{equation*}
\CB(\FT):=\bigcup_{d\in\Z}\CB_d(\FT)
\end{equation*}
is a $\CT:=
\CB(\FI_0)$ module. The space $\CB(\FT)$ is expected to be itself an $\Omega$-spectrum with the structure map $\CB_{d+1}(\FT) \rightarrow \Omega \CB_d(\FT)$.\footnote{Notice that $\CB_d$ is the $(-d)$-th space for the $\Omega$-spectrum. The structure map can be constructed by adding a free $U(1)$ vector multiplet which results is an $S^1$-family parametrized by the theta angle, while the inverse map is given by promoting the parameter for the $S^1$-family to a dynamical field. See \cite{Sigma} for more details and also \cite{GJ2,Johnson-Freyd:2020itv} for an alternative and possibly related proposal for $\CT$.}  
With $(0,1)$ supersymmetry, one conjectures that $\CT$ and TMF are actually equivalent as spectra  and that this equivalence can be promoted to be one of $E_{\infty}$-ring spectra. The $\CT$-action on $\CB(\FT)$ suggests that under this equivalence, $\CB(\FT)$ becomes a TMF-module. We will assume this throughout this section. 

At this stage, we have not made any assumptions of $\FT$ except that it is a 3d theory. For our applications, it would be useful if the homotopy type of $\CB(\FT)$ is first of all interesting, and, furthermore, invariant under ``continuous deformations'' of $\FT$. For a generic 3d theory $\FT$, one would not expect that this will be the case (for example, $\FT$ can possibly be deformed into the trivial theory, due to the lack of obstructions). There are two subclasses of 3d theories where one can avoid this problem. One is the class of topological theories. Under certain assumptions, they are rigid and cannot be deformed (e.g.~for fusion categories, which describe the behavior of the line operators of the topological theory, this is often known as Ocneanu rigidity \cite{etingof2005fusion}). Another bigger class is supersymmetric theories, whose deformation is highly constrained by supersymmetry. And one expects that these are the cases when $\CB(\FT)$ can be more interesting. In general, topological theories can also be viewed as supersymmetric with the action of the supercharge being trivial. Such an action is only consistent in a topological theory, as the commutator of two supercharges is a space-time translation. And from now on we will assume $\FT$ is supersymmetric, which implicitly includes also topological theories.

Operations for 3d bulk theories can lead to interesting consequences for their space of boundary conditions. One such operation is the tensor product $\FT_1\otimes\FT_2$ given by stacking the two theories together in a non-interacting way. This leads to a map
\begin{equation}\label{TensorMap}
\CB(\FT_1)\otimes_\CT \CB(\FT_2)\rightarrow\CB(\FT_1\otimes\FT_2).
\end{equation}
On the left, it is a tensor product over $\CT$ as one can move decoupled 2d theories between the two factors. Naively, one should not expect this to be an equivalence, as there can be more interesting boundary conditions that couple the two factors together. On the other hand, it is natural to associate to $\FT_1$ and $\FT_2$ two derived vector bundles over $\mathcal{M}$, and interpret $\CB(\FT_1)$ and $\CB(\FT_2)$ as the global sections of these two vector bundles.\footnote{Over complex numbers, this becomes the more familiar story of a 3d theory giving rise to a vector bundle over the moduli space of elliptic curve by considering its Hilbert space on a torus.} Then one expects this is actually an equivalence since taking global sections on $\MM$ commutes with taking the tensor product of two derived vector bundles by the main theorem in \cite{MathewMeier}. This motivates the following physics question.

\begin{question}
    Is the map \eqref{TensorMap} an equivalence? If yes, what is the physical reason for it and how general is this phenomenon?
\end{question}

If the second theory $\FT_2$ is a gravitational Chern--Simons theory at level $d$, then  $\CB(\FT_2)=$TMF$[-d]$, and tensoring with it becomes the shift operation,
\begin{equation}\label{ShiftMap}
\CB(\FT_1)\otimes_\CT \CB(\FI_d)\rightarrow\CB(\FT_1)[-d].
\end{equation}
For this reason we will abbreviate $\FT\otimes\FI_d$ as $\FT[-d]$.

The space of boundary condition on an incoming boundary, $\bar{\CB}(\FT)$, is also naturally a $\CT$-module, and it can be paired with $\CB(\FT)$ to give a map to TMF. This gives a map to the dual of $\CB(\FT)$,$$\bar{\CB}(\FT)\rightarrow\CB(\FT)^\vee.$$
There are reasons to believe that this is an equivalence if $\FT$ is a sufficiently nice QFT (e.g.~unitary, local, and Lorentz invariant). 

This can be thought of as the outgoing boundary of the orientation reversed theory $\bar{\FT}_1$
 which is obtained from ${\FT}_1$ by flipping the orientation of the spacetime,
 $$\bar{\CB}(\FT)\simeq\CB(\bar{\FT}).$$
 From the point of view of derived vector bundle over $\MM$, one should associate to it the dual bundle to that of $\FT$.

One can combine the tensor product and orientation reversal to relate the space of interfaces between $\FT_2$ and $\FT_1$ with $\CB({\FT}_1\otimes\bar\FT_2)$ using the ``folding trick.''\footnote{This is most commonly used in topological theories, where folding along the interface between two theories leads to a boundary condition for the tensor product of one theory with the dual of the other. In the framework of fully extended TQFTs, when $\FT_1=\FT_2$, the category of boundary conditions is what the TQFT assigns to a point, while the category of interfaces is associated with $S^0$ which consists of two points with the opposite orientation. The TQFT axioms then identify this with the category of boundary conditions of the tensor product of the original theory with the one obtained by orientation reversal.  For non-topological theories, there is in general no guarantee that the space of interfaces and boundary conditions are isomorphic. One scenario for this to fail is when the TMF-modules are of infinite rank. Also notice that, even when $\frak{T}$ is a 3d TQFT, our approach is different from that of the extended TQFT, as we obtain a space, as opposed to a 2-category, of boundary conditions. In particular, our machinery can handle TQFTs that do not have topological boundary conditions such as $U(1)$ Chern--Simons theory at a non-zero level.} Any such interface $\mathcal{I}$ gives a map
\begin{equation*}
\mathcal{I}: \quad\CB(\FT_2)\rightarrow \CB(\FT_1).
\end{equation*}
This is because, given any boundary condition $B$ of $\FT_1$, one can collide it with $\mathcal{I}$ by collapsing $\FT_1$ and they together become a boundary condition of $\FT_2$. 

The question above can also be stated in terms of interfaces. Naively, one would expect that there are more general interfaces between two theories then these obtained from colliding boundary conditions. For example, one can consider interfaces obtained by colliding interfaces with a third theory $\FT'$,
$$
\CB({\FT}_1\otimes \bar\FT')\otimes_{{\mathrm{End}(\FT')}} \CB({\FT}'\otimes \bar\FT_2)\rightarrow \CB({\FT}_1\otimes\bar\FT_2),
$$
and it would be surprising that for $\FT'$ being the trivial theory one can already have all classes of interfaces. On the other hand, it \emph{is} true for derived vector bundles $\Fscr_1$, $\Fscr'$ and $\Fscr_2$ on $\MM$ that 
\[(\Fscr_1\tensor (\Fscr')^{\vee})\tensor_{\mathcal{E}nd(\Fscr')}(\Fscr'\tensor \Fscr_2^{\vee}) \to \Fscr_1\tensor \Fscr_2^{\vee}\] 
is an equivalence, as can be seen locally where the vector bundles are free, and this equivalence remains valid on passing to global sections everywhere by the main result of \cite{MathewMeier}.

The last operator is ``addition,'' making the sum $\FT_1\oplus\FT_2$ from $\FT_1$ and $\FT_2$ viewed as two superselection sectors. Then there is a map $$\CB(\FT_1)\times\CB(\FT_2)\rightarrow\CB(\FT_1\oplus\FT_2).$$
This is expected to be an equivalence in general for quantum field theories. 

\subsection{TMF-modules and phases.}    Beside mathematical applications, we also expect that the TMF-module invariants of 3d supersymmetric theories can be useful in the study of the structure of phases and dynamics of such theories. We now take a physics detour to discuss this in further details.

This new invariant is expected to detect finer structure then the more familiar invariants such as the Witten index \cite{WITTEN1982253}. For example, during a phases transition, the Witten index can remain constant while the TMF-module can change. This is because the Witten index for a 3d supersymmetric theory is ultimately the index of its KK-reduction on $T^2$, and different phases in 3d can become connected once reduced to 1d.  For example, the Fayet--Iliopoulos (FI) parameter \cite{FAYET1974461} for a 3d $\mathcal{N}=2$ theory is real, and there can be walls separating different phases, but it will become complexified once we reduce the theory to 2d, allowing one to circumvent the transition point by turning on the imaginary part of the FI parameter. Other more familiar invariants (such as the algebra of BPS line operators) are often also subject to the same weakness. On the other hand, the TMF-module is a genuine invariant of 3d theories, and could better distinguish phases. We now give an example of this.

\begin{example}
    Consider a family of 3d $\mathcal{N}=2$ theory given by $U(1)_{k/2}$ with $k$ charge-1 hypermultiplets. When we vary the FI parameter $\zeta$ associated with the $U(1)$ gauge group, we get in the region $\zeta < 0$ a $U(1)_k$ Chern--Simons theory at low energy, while in the region $\zeta > 0$ one gets a $\mathbb{CP}^{k-1}$ sigma model. See \cite[Sec.~7]{kapustin2013wilson} and \cite[Sec.~3]{Gukov:2015sna} for more details. The Witten index in either phase is $k$, and even the BPS algebra of line operators are isomorphic on both sides. However, if one considers their TMF-modules, then they are not the same. In the sigma-model phase, the TMF-module is expected to be $\TMF\otimes \,\mathbb{CP}^{k-1}_+$, associated with the unreduced TMF homology of $\mathbb{CP}^{k-1}$. This is rank-$k$ module, which, after inverting 6, becomes TMF$\oplus$TMF$[2]\oplus\ldots$TMF${[2k-2]}$. (More on this in the next subsection.) In the Chern--Simons phase, as we will see later, one has the module $L_k$, which is also of rank-$k$, but different as module (often even after inverting 6). When $k=2$, the former is TMF$[2]\oplus$TMF, while the latter, up to a shift, is the mapping cone of $\nu$ (cf.~ \cref{ex:Oke}). This demonstrates that these are in fact two different 3d phases. Notice that under dimensional reduction, we end up with the same KU-module given by KU$\oplus$KU.
\end{example}

Phase transitions often happen when the spectrum of the theory becomes massless. At the phase boundary, the Witten index is ill-defined, but one can still consider the space of boundary conditions as a TMF-module. Denote the space at the phase transition as $\mathcal{B}_0$, and these on the two sides as $\mathcal{B}_+$ and $\mathcal{B}_-$, one expects that there are maps of TMF-modules
$$
\mathcal{B}_-\rightarrow \mathcal{B}_0 \leftarrow \mathcal{B}_+
$$
but maybe not in the opposite directions due to the emergence of new massless degrees of freedom at the phase transition point allowing special boundary conditions that won't exist away from this point.

What happens at the phase boundary? One can get some idea by compactifying and consider the space of boundary conditions for 1d $\CN=2$ supersymmetric theories. The space of boundary condition in a 1d theory is almost the same as its space of quantum states. If one considers supersymmetric boundary conditions, then one gets the space of supersymmetric (a.k.a.~BPS) states. The TMF-module and its rank is respectively a refinement of the BPS Hilbert space and the Witten index. In the 1d case, there are two kinds of phase transitions. One is often referred to as ``wall-crossing'' which is when the spectrum of the theory become continuous. Another is when pairs of bosonic and fermionic states enter or leave the ground state. Describing and classifying the behavior of TMF and KO-module at phase boundary in a similar fashion should shed much light on the physics of phase transition in two and three dimensions.  It would be interesting to understand whether there are additional constraints of these spaces and maps between them, and, furthermore, whether homotopy theory can be effectively used to classify phase transitions of 3d theories.

\begin{remark}
    One can also use TMF-modules to tackle higher-dimensional theories, and there are different ways of achieving this. For example, one can consider the space of 2d $(0,1)$ surface operators in higher-dimensional theories, which will again be TMF-modules. Another strategy is to first compactify the $D$-dimensional theory to 3d, and then consider the space of 2d $(0,1)$ boundary conditions. (The first approach roughly corresponds to compactifying on a sphere $S^{D-3}$.) This can lead to more interesting structures, as there are often more ways to compactify to 3d, such as considering different manifolds to compactify on and different backgrounds (e.g.~holonomy and fluxes of global symmetries). Yet another --- which is in a sense more natural from the TQFT perspective --- method is to associate TMF-linear categories to 4d theories and higher categories to theories in even higher dimensions. We will give more comments about this generalization toward the end of this section. It would be interesting to see whether these approaches can be helpful in the study of phases of higher-dimensional theories. 
\end{remark}

\subsection{Boundary conditions of sigma models}

When the fields of the sigma model are non-dynamical and treated as background fields, the space of boundary conditions is naturally given by the space of maps $[X,\TMF]$, whose homotopy group is TMF$^*(X)$ --- the TMF cohomology of $X$ --- as, for any given value of the background $p \in X$, one has a well-defined 2d theory. When the 3d bulk is a genuine sigma model with dynamical fields, this is no longer correct, as 1) the ``family of theories'' has to be now a supersymmetric bulk-boundary coupled system and 2)  one can now has Dirichlet boundary conditions given by submanifolds of $X$. In this case, instead of getting the TMF-cohomology, one expects to have the \emph{homology} version:

\begin{conj}
    The homotopy type of the space of boundary conditions of a 3d $\mathcal{N}=1$ supersymmetric sigma model to $X$ is described by TMF$\,\otimes X_+$.
\end{conj}
 
    A physics argument for it uses a deformation of the 3d $\mathcal{N}=1$ sigma model by a real superpotential and the  analysis of the low-energy theory \cite{Sigma}, similiar to Witten's work on the 1d supersymemtric sigma model \cite{Witten:1982im}. In fact, this is the 3d lift of Witten's setup, as 3d $\mathcal{N}=1$ theories exactly reduce to 1d $\mathcal{N}=2$ quantum mechanics studied in \cite{Witten:1982im}, and the question about the space of boundary conditions for the 3d theory becomes a question about the Hilbert space for 1d quantum mechanics. Given a generic superpotential with isolated critical points, the 3d theory at low energy has massive vacua labeled by the critical points. The Morse index gives the number of fermions with the negative mass, which shifts the degree. For Morse flows connecting different critical points, they lead to supersymmetric domain walls in the low-energy effective theory connecting the corresponding vacua. They act as ``differentials'' by relating the boundary conditions for the different vacua and eliminating them when the domain wall is condensed. This is very analogous to how the Morse--Smale--Witten complex is constructed and gives a model for computing the homology $H_*(X)$ of the sigma-model target $X$, but, in the 3d case, leads to the $E_1$-page of the Atiyah--Hirzebruch spectral sequence for TMF$_*(X)$. One key difference is that the domain walls in 3d that correspond to higher-index flows are non-trivial and give rise to higher differentials in later pages.\footnote{The physics of this system is incredibly rich, and one important ingredient is that the theory on the domain wall being itself a sigma model to the moduli space of Morse flows, and one has to use the entire framed flow category as defined in \cite{cohen1995morse} and subsequence works to encode the information of the system at low energy. See \cite{Sigma} for more details.}

\subsection{KU- and KO-modules from 2d theories.} Similarly, one can associate to a 2d supersymmetric theory (or, even simpler, a 2d TQFT) $\FT$ the space of 1d $\mathcal{N}=1$ supersymmetric boundary conditions $\CB(\FT)$. We can also restrict to these boundary conditions with time reversal symmetry, $\CB^{\mathbb{R}}(\FT)$. They are expected to be respectively a KU- and a KO-module,  assuming that the space of such 1d $\mathcal{N}=1$ theories is the KU/KO-spectrum (see \cite{berwickevans2023familiesanalyticindex11dimensional} for an exposition of this following the Stolz--Teichner program). 

To have this space of boundary conditions well-behaved under deformations of the bulk theory, we need to assume that the bulk has 2d $(1,1)$ supersymmetry, and for the case of KO-modules, the bulk should also have a time-reversal symmetry.

One reason to consider KO-modules associated with 2d theories is that they are related to the TMF-modules associated with 3d theories via circle compactification, and can serves as the ``stepping stone'' to understand the latter. In fact, since one can decompose a theory on $S^1$ according to the KK-momentum, one actually gets a KO$[[q]]$-module. 

\begin{remark}
For topological (or ``gapped'') theories, sectors with higher KK-momentum are trivial, and considering the KO$((q))$-module is not expected to gain more information than only keeping the zero-momentum part of the boundary theory and view it as a KO-module. 
\end{remark}

For a general TMF-module, the reduction to KO$((q))$-module doesn't factor through a KO-module. However, an interesting class, given by TMF$\otimes X$, with $X$ a spectrum, does factor. For all examples we encountered with bulk being gapped, the TMF-module is always of this type. This motivates the following question.

\begin{question}
If $\FT$ is gapped, is $\CB(\FT)$ as a TMF-module always of the form TMF$\otimes X$ with $X$ a spectrum?    
\end{question}

Another, and perhaps more interesting from the physics perspective, motivation to consider KO-modules associated to 2d theories is to distinguish gapped phases, which are system with no massless degree of freedom. For such a gapped bulk theory, while information in the KU-module is basically equivalent to that of the ground states Hilbert space, the KO-module captures finer information of the bulk theory and can be used to distinguish phases.  

\begin{example}
    There are three closely related phases in 2d: the $\Z_n$ gauge theory (also known as Dijkgraaf--Witten theory \cite{dijkgraaf1990topological} or $\Z_n$ topological order), the $U(1)_n$ (or $U(1)_{-n}$) gauged WZW model, which can be viewed as the dimensional reduction of the $U(1)_n$ Chern--Simons theory, and the supersymmetric $\mathbb{CP}^{n-1}$ sigma model. They all have the same algebra for local operators and are sometimes believed to be identical as 2d phases. However, the KO-module associated to them are in general different, suggesting that they are actually distinct in these cases when considered as phases with time reversal symmetry. Their KO-modules are listed below. 
\begin{center}
\begin{tabular}{c|c|c|c|c|c}
    Theories & $n=1$ & $n=2$ & $n=3$ &  $n=4$ \\\hline
    $\Z_n$ & KO & KO$\oplus$KO & KO$\oplus$KU & KO$\oplus$KU$\oplus$KO & \ldots\\\hline
    $U(1)_{-n}$ & KO[2] & KO$\oplus$KO[4] & KO[6]$\oplus$KU & KO$\oplus$KU$\oplus$KO &\ldots\\\hline
    $\mathbb{CP}^{n-1}$ &  KO & KO$\oplus$KO[2] & KO$\oplus$KU & KO$\oplus$KU$\oplus$KO[6] & \ldots
\end{tabular}
\end{center}
The result for $\Z_n$ is standard;\footnote{For a finite group $G$, the category of finite-dimensional real $G$-representations decomposes as a product of categories indexed over the set of irreducible real $G$-representations $V$. Each of these categories is equivalent to that of vector spaces over the division algebra of endomorphisms of $V$, which can be $\R$ (real type), $\C$ (complex type) or $\H$ (quaternionic type). Correspondingly, the $G$-equivariant KO-theory of a point (also known as the $G$-fixed points of $\mathrm{KO}_G$) is equivalent to a product of KO, KU and $\mathrm{KSp} \simeq \mathrm{KO}[4]$. For $\Z_n$, the real irreducible representations are indexed by the sets $\{[\pm k]\}$ for $[k]\in\Z_n$, with complex type if the set has two elements and real type if it has one element.}  for $U(1)_{-n}$ it is obtained from reducing the TMF-module; for $\mathbb{CP}^{n-1}$ it is given by KO$\wedge\mathbb{CP}_+^{n-1}$.
\end{example}

For the non-extended ``Pin$^-$-TQFTs'' (i.e.~TQFTs with time-reversal symmetry), it has two additional structures. One is a $\Z_4$ action on Hilbert spaces from the time reversal symmetry $T$ that satisfies $T^2=(-1)^F$ in the Euclidean signature. Another is a ``1-form monoid symmetry'' from connected sum with $\mathbb{RP}^2$. In simple cases, their action on the TQFT Hilbert spaces on $S^1$ with the two Pin$^-$-structures determine whether one gets KO or KU and the degree shifts. 

To actually study the KO-module that describes the space of boundary conditions, one would have to first find a right framework to define Pin$^-$-TQFTs as an extended TQFT. We will leave this more general and careful study to future work.

\subsection{TMF-modules from Abelian Chern--Simons theory}\label{sec:ChernSimons}

Recall from \cref{sec:PhysicsMotivation} that we took a 6d theory $T$ as the physical starting point of our paper. 
When $T$ is an abelian tensor multiplet, the theory $T[M_3]$, after removing a certain decoupled degree of freedom, is given by an abelian Chern--Simons theory. The Chern--Simons level is now a matrix identified with the ``linking matrix'' of $M_3$.\footnote{A presentation for the linking form of $M_3$ will have redundancies given by the 3d Kirby moves, which on the Chern--Simons side corresponds to quantum equivalences. See \cite{Belov:2005ze, Kapustin:2010hk} for classification of abelian Chern--Simons theories up to such equivalences. However, as what $T[M_3]$ contains is the supersymmetric version of the Chern--Simons theory with additional fermions, some of the equivalence only holds up to a shift, which exactly matches the shifts seen from TMF-modules.
}
 One way to encode this data is exactly by a symmetric form on a lattice $\mathbf{\Z}^r$, which determines a Chern--Simons form for $U(1)^r$. The condition that the form is integral match that the condition for the Chern--Simons theory to be well-defined spin-TQFT at the quantum level. See \cite{Witten:2003ya,Belov:2005ze} for a review of quantum abelian Chern--Simons theories.  

The Hilbert space of a Chern--Simons theory on a Riemann surface $\Sigma$ can be obtained via geometric quantization, where it is identified with certain space of theta functions. More precisely, the level matrix $k$ determines a line bundle $L_k$ on the moduli space of flat $U(1)^r$ connections on $\Sigma$, and the Hilbert space for the abelian Chern--Simons theory is identified with holomorphic sections of this line bundle. When $\Sigma=\mathcal{E}$ is an elliptic curve, $L_k$ is a line bundle over $(\mathcal{E}^\vee)^{\times r}$. One goal of the present paper, interpreted in terms of the Chern--Simons theory, is to ``upgrade'' this vector space of holomophic sections of $L_k$ to a TMF-module. In the remainder of this section, we illustrate that the mathematical results from the TMF theory is, in many cases, consistent with the physics predictions, and, in other cases, can be used to give new physics predictions.

\begin{example}
    $U(1)_0$. Although this is strictly speaking not a Chern--Simons theory due to the vanishing Chern--Simons level, it fits naturally in our construction, and the associated TMF-module is $$\CB_{0}=\TMF\oplus \TMF[-1].$$ The two copies are generated by the Dirichlet and Neumann boundary conditions. The shift of the second copy reflects the fact that, for the Neumann boundary condition, there is a left-moving fermion in the $(0,1)$ vector multiplet. In other words, the supersymmetric Neumann boundary condition is only a genuine boundary condition of $U(1)[1]:=U(1)\otimes \FI_{-1}$ where a gravitational Chern--Simons theory of level $-1$ is included. The $U(1)$ gauge theory is often considered to have an equivalent description of a dual photon (which can be thought of as a 3d sigma model to $S^1$).
    Interestingly, there is an overall shift if one considers the TMF-module associated with the dual photon,
    $$\CB_{\text{dp}}=\TMF\oplus\TMF[1],$$
    whose Neumann boundary condition (dual to the Dirichlet boundary condition for the $U(1)$ guage field) is in degree 1 as it has a boundary chiral multiplet with a right-moving fermion. Therefore, the dual photon description is not exactly equivalent with the gauge theory description in our setting, and the degree shift can be understood as fermion modes on the duality interface between them. The interface will look like either a Neumann boundary condition for the gauge field, with the theta angle identified with the dual photon on the other side, or alternatively a Neumann boundary condition for the dual photon with an $S^1$-valued chiral multiplet, whose $U(1)$ symmetry is coupled with the bulk gauge field on the other side. For either choice, one has to include an invertible theory with either gravitational Chern--Simons level $-1$ on the gauge theory side or level $1$ on the dual photon side, and the duality is really between $$U(1)[1]\longleftrightarrow \text{dual photon}.$$ Therefore, in our setting, gauge theory and the dual photon are only equivalent up to an invertible theory. 
\end{example}

\begin{remark} 
    For the non-supersymmetric version, the $U(1)$ gauge theory is parity invariant. For the supersymmetric version, we actually have $\overline{U(1)_0}=U(1)_0[1]$ is equivalent to the dual photon.\footnote{One way to think about this is that there is a background gravitational Chern--Simons level of $\frac12$ for the $U(1)$ theory while one with level $-\frac12$ for the dual photon. Such a fractional level is needed if we want to have an integral level once the massless fermions in the supersymmetric theory are integrated out. Physically, one can choose the level to be any integer plus $\frac12$, but the present choice is the one that is compatible with the convention we use on the TMF-module side. This level will be flipped upon parity reversal, leading the degree shift of $\pm1$.} This identification at the level of TMF-modules is  $$\CB_0[1]\simeq \bar{\CB}_0\simeq\CB_0^\vee=\TMF\oplus\TMF[1].$$
    In our construction for invariants of 3-manifolds, this is the source for the degree shift when reversing the orientation of $M$ to $\overline M$ in \cref{prop:Zduality}.
    This non-trivial degree shift can also be interpreted as the boundary condition for $U(1)\otimes\overline{U(1)}$ using the folding trick carries gravitational anomaly. Similarly, for the dual photon, we have 
    $$\CB_{\text{dp}}[-1]\simeq \bar\CB_{\text{dp}}\simeq \CB_{\text{dp}}^\vee=\TMF\oplus\TMF[-1].$$
    One might naively expect that these map are just the identity, but they are actually not. For the first map, one can view it as a pairing
    $$
    \CB_0\otimes \CB_0\rightarrow \TMF[-1],
    $$
    whose information is encoded in a 2-by-2 $\pi_*(\TMF)$-valued matrix giving the pairing between the two types of basic boundary conditions, with the four entries
    $a_{\text{NN}}\in \pi_{-1}(\TMF), \;a_{\text{ND}}=a_{\text{DN}}\in \pi_{0}(\TMF),$ and $a_{\text{DD}}\in \pi_{1}(\TMF)$. It is not hard to see that $a_{\text{ND}}$ is indeed the identity, as one gets a trivial 2d theory with the Dirichlet and Neumann boundary condition on each side. Simply for degree reason, one obtains that $a_{\text{NN}}=0$, as it is a 2d abelian gauge theory with a left-moving fermion in degree $-1$. A deformation to a theory that spontaneously break supersymmetry is given by turning on different theta angles on the two Neumann boundary conditions. On the other hand, $a_{\text{DD}}=\eta$, as the 2d theory now is an $S^1$ sigma model. The analysis from the perspective of the dual photon is similar, except the role of the two types of boundary conditions is switched -- one now gets an $S^1$ sigma model with two Neumann boundary conditions, while two Dirichlet boundary conditions lead to a theory that can spontaneously break supersymmetry if the fixed values of the dual photon are different on the two boundaries. This matrix is the same as in \cref{prop:explicitduality} up to a change of basis.
\end{remark}

\begin{remark}
The dual photon can be viewed as a 3d sigma model, with target space $S^1$, and the TMF-module assigned to it is also the TMF-homology of $S^1$,
\begin{equation}
\TMF_*(S^1)=\TMF\oplus \TMF[1].
\end{equation}
On the other hand, the $U(1)$ gauge theory is associated with the TMF-cohomology of $S^1$,
\begin{equation}
    \TMF^*(S^1)=\TMF\oplus \TMF[1],
\end{equation}
which classify $S^1$ families of 2d theories. From physics, this $S^1$ family can be viewed as parametrized by the boundary theta angle. 
\end{remark}

\begin{example}
    $U(1)_1$ and $U(1)_{-1}$. The TMF-modules associated with them are TMF$[-3]$ and TMF$[2]$ respectively (cf.~\cref{ex:Oke}). 
    This is compatible with the expectation that, at these two levels, the $U(1)$ Chern--Simons theory is equivalent to a gravitational Chern--Simons theory at level 3 and $-2$.\footnote{The non-supersymmetric $U(1)$ Chern--Simons theories at level $\pm 1$ is expected to be equivalent to a gravitational Chern--Simons theory of level $\pm2$. However, with 3d $\mathcal{N}=1$ supersymmetry, there is an additional unit of shift given by having the gaugino mass changing sign. Similar to the $U(1)_0$ case, one can also think of it as the theories have included a ``bare background gravitational Chern--Simons term'' at level $\frac{1}{2}$, which is shifted to $1$ (or 0) for $U(1)_1$ (or for $U(1)_{-1}$) after integrating out the gaugino.} The shift also enables two ``canonical'' boundary conditions, one given by a left-moving compact boson for $U(1)_{-1}$ and the other given by a right-moving compact boson with its fermionic partner for $U(1)_1$, to be placed in degree zero. 
    \end{example}
    
    \begin{remark}
        This is a good example to illustrate some differences from the physics perspective between the homology version $L_b$, which describes boundary conditions of the dynamical abelian Chern--Simons theory, and the cohomology version $L^b$, which corresponds to having a non-dynamical bulk theory. If one considers the cohomology version for $U(1)_{\pm1}$, then one has $L^{-1}=$TMF$[3]$ and $L^{1}=$TMF$[-2]$ generated by these same two boundary conditions,\footnote{Note that the $L^{-1}$ actually corresponds to $U(1)_1$ and vice versa due to the convention that we are using where $L_b=(L^{-b})^{\vee}$.} but as the bulk theory is non-dynamical, a gravitational Chern--Simons term is needed to cancel the anomaly, giving the shift in degree. These two special boundary conditions, when regarded as standalone 2d theories, is in the trivial TMF class, but one can put them together in a $U(1)$-equivariant way. The resulting 2d theory obtained by pairing up the left- and right-moving chiral boson is an $S^1$ sigma model. This can be formulated as the non-trivial pairing map
    $$
L^1\otimes L^{-1}=\TMF[1]\rightarrow \TMF
    $$
    given by $\eta$. One expects from physics that there is a similar map for general $b$. For homology, the bulk is dynamical, and one can not simply forget about it. Instead, there are maps in the opposite direction from TMF to $L_1$, $L_{-1}$ with the image being the natural boundary conditions mentioned above, as well as one to $L_{-1}\otimes L_{1}=\TMF[-1]$ again given by $\eta$.
    \end{remark}

\begin{example}
    $U(1)_2$ and $U(1)_{-2}$. By \cref{ex:Oke}, the TMF-modules associated with them are respectively the $(-5)$-fold shift of the cone of $\nu\colon \TMF[3]\to \TMF$ and $\mathrm{Cone}(\nu)$ itself. This shift is compatible with the duality \cite{Seiberg:2016rsg}
    $$U(1)_2\otimes U(1)_{-1} \simeq U(1)_{-2}\otimes U(1)_{1}$$
    where both sides are associated with Cone$(\nu)[-3]$.
\end{example}

\subsection{4d theories and TMF categories} One can go one categorical level higher by considering four-dimensional theories and their boundary conditions.\footnote{This is sometimes necessary as the 3d theory itself has anomaly, as is often the case when the theory is obtained from the compactification of a non-trivial 6d theory instead of just the free tensor multiplet. The 6d theories are often themselves ``relative theories,'' living on the boundary of 7d theories. Then $T[M_3]$ obtained from compactification also naturally lives on the boundary of a 4d anomaly theory (see, e.g.,~\cite{Gukov:2020btk} for a more detailed discussion about anomalies obtained from such compactification).} Fix a 4d supersymmetric theory ${T}$ and consider its category $\mathcal{C}_{{T}}$ of 3d $\mathcal{N}=1$ boundary conditions. Morphisms are identified with 2d $(0,1)$ interfaces between two such boundary conditions. As the space of such interfaces is naturally a TMF-module by stacking with 2d $(0,1)$ theories, one expects that $\mathcal{C}_{{T}}$ is naturally a TMF-linear category.

From the point of view of the 4d TQFT, this is about extending the TQFT functor down to associate TMF categories to 2-manifolds. And one can ask, in the ``toy model,'' what are these categories. 

\begin{example}
    Categories $\mathcal{C}_{S^2}$ associated with $S^2$. The 4d theory in this case is a ``tensor multiplet'' which can also be think of as a sigma model to $\mathbb{R}\times S^1$. The $\mathbb{R}$ direction represents a ``center of mass'' degree of freedom that we often decouple when building the TQFT from the toy model. So one can think of it as an $S^1$ sigma model. Then a natural candidate for the category $\mathcal{C}_{S^2}$ is then the category of sheaves of TMF-modules over $S^1$.
\end{example}

This generalizes to a 4d sigma model with arbitrary target $X$ with $G$ action, which is naturally associated with the category of $G$-equivarient sheaves on $X$. When $X$ is a point, one gets the category associated with the 4d $G$ gauge theory. 

Previously, we have seen that there are two versions of TMF-modules associated with a 3d theory by considering it as either dynamical or non-dynamical. Should we also have two categories associated with the 4d theory? Interestingly, this might not be necessary.

\begin{remark}
    From the category of sheaves of TMF-modules $\mathcal{C}(X)$, one can get both the cohomology TMF$^*(X)$ and the homology TMF$_*(X)$ from two different functors applied to the constant sheaf with value TMF. These are the right adjoint $p_*$ (global sections) and the left adjoint $p_!$ of the pullback functor $p^*$ along $p\colon X\to \mathrm{pt}$ or, equivalently, to the inclusion of constant sheaf;  the existence of $p_!$ follows from \cite{VolpeSix} or from the adjoint functor theorem.
\end{remark}

\appendix

\section{Intersection forms and linking pairings} \label{sec: homology}

The goal of this appendix is to summarize the relevant homological data of $3$- and $4$-manifolds, and in particular the interplay between the torsion linking pairings of $3$-manifolds and the intersection forms of $4$-manifold cobordisms.
This is used in the formulation of manifold invariants in Section \ref{sec: invariants of manifolds}. While much of this material is classically known, we could not find the precise statements in the literature. For convenience of the reader we derived and stated them, for example Lemma \ref{lem: pm equivalence} and the formulas in Section \ref{sec: cobordism} relating the linking forms on $Y_i$ and the intersection form on $W$ for a $4$-manifold cobordism $(W; Y_0, Y_1)$.

\subsection{Homological data needed to define $\TMF$ modules for $3$-manifolds.} \label{sec: data}
All homology and cohomology groups are taken with $\Z$ coefficients. Given an abelian group $A$, denote $A^*:={\rm Hom}(A, \Z)$.

Given a connected, closed $3$-manifold $Y$, consider a $4$-manifold $X$ bounding it, with $H_1(X)=0$. More concretely, a simply-connected $X$ with $\partial X=Y$ can be constructed by attaching $2$-handles to the $4$-ball; this viewpoint is useful because \cref{sec: invariants of manifolds} defines maps of $\TMF$-modules associated with 4-dimensional cobordisms assembled of $2$-handles. The construction in Sections \ref{sec: Looijenga line bundles}, \ref{sec: Properties of Looijenga line bundles} assigns a TMF-module to the intersection form on $X$. Lemma \ref{lem:stable} shows that the construction depends only on the stable equivalence class of the intersection form. Here we formulate the homological data needed to carry out the construction.

Recall (cf.~\cite{GordonLitherland}) that the intersection form on the $4$-manifold $X$ as above gives a presentation of the linking pairing on $Y$, in the following sense.  Consider the long exact sequence
\begin{equation} \label{present}
H_2(X)\to H_2(X,Y)\to H_1(Y)\to 0.
\end{equation}
Since $H_1(X)=0$, $H_2(X)$ is a free abelian group.
By Poincar\'{e} duality $H_2(X,Y)\cong H^2(X)\cong H_2(X)^*$, and the first map in the above sequence corresponds to the intersection pairing on $H_2(X)$. The linking form on the torsion subgroup $\tau H_1(Y)$, 
\[ \lambda\colon\tau H_1(Y)\times \tau H_1(Y)\to \Q /\Z,\]
may be computed in terms of this intersection pairing, see \cite[p.~60]{GordonLitherland} and \cite[Section 2]{KrFr:2025} for more details. Next we show that, in fact, the linking form determines a stable equivalence class of such presentations, thus providing the homological data needed to define the TMF-module. Denote $A:=H_1(Y)$, $B:=H_2(X)$; the torsion subgroup of $A$ will be denoted $\tau A$.
Consider a presentation 
\begin{equation} \label{present1}
B
\mathrel{\mathop{\longrightarrow}^{\langle\, , \, \rangle}}
B^*\to A
\end{equation}
of the linking pairing on $TA$. We say that two such presentations (bilinear forms $(B, \langle\, , \, \rangle)$) are {\em $\pm$equivalent} is they become isomorphic after direct summing with some number of $\langle\pm 1\rangle$.

\begin{lemma} \label{lem: pm equivalence}
The linking pairing on $\tau A$ and the rank of $A$ determine the equivalence class of $(B, \langle\, , \, \rangle)$ up to  $\pm$equivalence.
\end{lemma}
\begin{proof}  We start by recalling the following classical results.
The Witt group $W(\Z)$ over the integers is isomorphic to $\Z$, where the isomorphism is given by the signature \cite{MilnorHusemoller}.
The Witt group consists of the equivalence classes of unimodular bilinear forms, where two forms are equivalent if they become isomorphic after stabilization by metabolic forms. 
Thus it follows from the classification of indefinite unimodular forms that any two {\em unimodular forms are $\pm$equivalent}.

Next consider integral symmetric bilinear forms $B\times B\to \Z$ which are non-degenerate over $\Q$ ({\em lattices} in the terminology of \cite{Nikulin}.) 

\begin{thm}(Nikulin, \cite[Theorem 1.3.1]{Nikulin})
Two lattices induce isomorphic bilinear forms on $B^*/B$ if and only if they become isomorphic after direct summing with suitable unimodular lattices.
\end{thm}

It follows from the preceding discussion of unimodular forms that two lattices induce isomorphic bilinear forms on $B^*/B$ if and only if they are $\pm$equivalent.

Finally, consider 
an arbitrary (no non-degeneracy assumption) bilinear form. The following proof is phrased in the manifold context, that is $A=H_1(Y)$, $B=H_2(X)$, which is of main interest in this paper, but the analogous proof may be given in the abstract setting of a bilinear form \eqref{present1}.
Extend the sequence \eqref{present} to the left; in the abstract setting this corresponds to adding the kernel of the bilinear pairing on the left in \eqref{present1}:
$$0\to H_2(Y)\mathrel{\mathop{\to}^{\alpha}} H_2(X)\to H_2(X,Y)\mathrel{\mathop{\to}^{\beta}} H_1(Y)\to 0.$$

Note that the following diagram commutes, giving an identification of $\alpha$ and $\beta^*$.
\[
\begin{tikzcd}
H_1(Y)^* \arrow{d}[swap]{\beta^*} \arrow[r,"\cong"]
& H^2(Y)^* \arrow[d] \arrow[r,"\cong"] &
H_2(Y)   \arrow{d}[swap]{\alpha} \\
H_2(X,Y)^*    \arrow[r,"\cong"] &
H^2(X)^*   \arrow[r,"\cong"] &
H_2(X) \\
\end{tikzcd}
\]
The isomorphisms on the left are given by Poincar\'e duality, and the isomorphisms on the right follow from the universal coefficient theorem for cohomology. Consider the decomposition of $H_1(Y)$ as the direct sum of the free and torsion subgroups, $H_1(Y)= F\oplus T$. Let $\pi$ denote the projection $H_1(Y)\longrightarrow F$. The composition $\pi\circ \beta$ equals $\alpha^*$. Since $\alpha^*$ is surjective, it follows that $H_2(X)$ splits as a direct sum ${\rm im}(\alpha)\oplus C$, where $C$ is free abelian. It follows that $C\longrightarrow C^*\longrightarrow T$ is a presentation of the linking form on $T$. The restriction of the intersection form to $C$ is non-degenerate over $\Q$, the case considered above.
This concludes the proof of Lemma \ref{lem: pm equivalence}.
\end{proof}

\subsection{Computing the torsion linking pairing in terms of the intersection pairing of $4$-manifold cobordisms} \label{sec: cobordism}
It is reasonable to expect that our construction of the invariants of 3- and 4-manifolds factors through a ``homological cobordism category'' formulated in terms of homology groups, torsion linking pairings of 3-manifolds, the intersection pairings of 4-manifolds, and their interaction. Defining such a category is outside the
scope of this paper. However, we make a step in this direction and show how to compute the torsion linking pairing in terms of the intersection pairing of $4$-manifold cobordisms. This is a generalization of the formula (cf. \cite[Section 3]{GordonLitherland} and \cite[Lemma 2.1]{KrFr:2025}) expressing the linking form of the boundary $\partial W$ of a compact 4-manifold $W$ with $H_1(M)=0$ in terms of the intersection form on $W$. 

Consider a compact, oriented $4$-manifold $W$ with $\partial W=\overline Y_0\sqcup Y_1$. We assume 
\begin{equation}
H_1(W,Y_i)=0, \; i=1,2,
\end{equation}
this is a homological analogue of $W$ obtained from $Y_0\times I$ by attaching $2$-handles. It follows that $H_2(W,Y_0)\cong H_2(W, Y_1)$ are free abelian groups. A key ingredient in our calculations is the following diagram consisting of segments of two long exact sequences: 
\begin{equation} \label{fibration diag2}
\centering
\begin{tikzcd}[row sep=tiny, column sep=small]
H_2(W, Y_0) \arrow[dr,shorten >=1.5ex] & & H_1(Y_0)\\
 & H_2(W, Y_0\cup Y_1)  \arrow[two heads,ur,shorten <=1.5ex,"\alpha_0", pos=0.7] \arrow[two heads,dr,shorten <=1.5ex,"\alpha_1", pos=0.7, swap]  & \\
H_2(W, Y_1) \arrow[ur,shorten >=1.5ex] & & H_1(Y_1)
\end{tikzcd}
\end{equation}

For $i=0,1$ consider the free and torsion subgroups of the first homology of $Y_i$, $H_1(Y_i)\cong F_i\oplus T_i$. The goal is to relate the linking forms on $T_0, T_1$ and the intersection pairing on the cobordism $W$. To this end, consider $a_0, b_0\in \tau H_1(Y_0)$, and let $x,y$ be elements in $H_2(W, Y_0\cup Y_1)$ with $\alpha_0(x)=a_0, \alpha_0(y)=b_0$.
Denote $a_1:={\alpha}_1(x), b_1:={\alpha}_1(y)$. A priori, there are three cases: (1) $a_1, b_1\in \tau H_1(Y_1)$, (2) one of $a_1, b_1$ is a torsion elements and the other one is in $F_1$, and (3) $a_1, b_1\in F_1$, for any $x\in {\alpha_0}^{-1}(a_0)$ and $y\in {\alpha_0}^{-1}(b_0)
$. However, it follows from the non-degeneracy of the torsion linking pairing 
$$\tau H_1 (W)\times \tau H_2(W, Y_0\cup Y_1)\longrightarrow \Q/\Z
$$
that the case (3) is possible only when $\lambda(a_0,b_0)=0$. Therefore we focus on computing the torsion linking pairing in the first two cases.

\begin{enumerate}
    \item 
Suppose $a_1, b_1\in \tau H_1(Y_1)$.
Consider $\Delta\in\Z$ such that $\Delta a_i=\Delta b_i=0.$ For $i=0,1$ consider relative cocycle representatives $X, Y$ of $x,y$ in $(W, Y_0\cup Y_1)$ and cocycle representatives $A_i, B_i$  of $a_i, b_i$ in $Y_i$. Let $U_i, V_i$ be cochains in $Y_i$ such that $\partial U_i=\Delta A_i$, $\partial V_i=\Delta B_i$. It follows that $$
\widetilde X := \Delta X -U_0-U_1, \; \widetilde Y := \Delta Y -V_0-V_1
$$
are cocycles in $W$. A geometric argument, outlined below and  similar to that in \cite[Lemma 2.1]{KrFr:2025}) implies a relation in $\Q/\Z$ between the intersection pairing $H_2(W)\times H_2(W)\longrightarrow \mathbb Z$ and the linking forms:
\begin{equation} \label{eq: linking intersection} 
\lambda (a_0, b_0) - \lambda(a_1, b_1) \equiv-\frac{[\widetilde X]\cdot [\widetilde Y]}{\Delta^2} \; \, ({\rm mod}\; 1) .
\end{equation}
Indeed, $$\lambda (a_i, b_i)\equiv \frac{U_i\cdot B_i}{\Delta} \;\,  ({\rm mod}\; 1).$$ Geometrically, $\widetilde X$ consists of $\Delta X\subset W$ capped off by $U_0$ in $Y_0$ and $U_1$ in $Y_1$, with the analogous description of $\widetilde Y$. Push $\widetilde X$ by an isotopy slightly into the interior of $W$. Then transverse intersections $\widetilde X\cap \widetilde Y$ fall into three categories: intersections $\Delta X\cap \Delta Y$, and $U_i\cap \Delta Y$, $i=0,1$. The intersection number $\Delta X \cdot\Delta Y$ is a multiple of $\Delta^2$, so it does not contribute to \cref{eq: linking intersection} (mod 1). The other two terms, $U_0\cdot \Delta Y$ and $U_1\cdot \Delta Y$ are precisely
$\lambda (a_0, b_0) - \lambda(a_1, b_1)$ with an additional factor $\Delta$ since the intersections are counted with $\Delta Y$.

\item
Suppose $a_1\in \tau H_1(Y_1)$ and $b_1\in F_1$. Using the notation as above, consider a cocycle $
\widetilde X := \Delta X -U_0-U_1$ in $W$ and a relative cocycle $\widetilde Y := \Delta Y -V_0
$ in $(W, Y_1)$. The analogue of \cref{eq: linking intersection} reads $$ \lambda (a_0, b_0)  \equiv -\frac{[\widetilde X]\cdot [\widetilde Y]}{\Delta^2} \; \, ({\rm mod}\; 1) .$$
The right-hand side refers to the intersection pairing ${H_2(W)\times H_2(W, Y_1)\rightarrow \mathbb Z}$.
\end{enumerate}

\bibliographystyle{alpha}
 \bibliography{bibl}
    \end{document}